\documentclass[12pt,a4paper]{amsart}
\usepackage{graphicx,latexsym,amsfonts,amsmath,amssymb,rotating,txfonts,mathrsfs,enumerate, mathtools}
\usepackage[latin1]{inputenc}
\usepackage{epic}
\usepackage{curves}
\usepackage{pdfsync}
\usepackage{ytableau}
\usepackage{tikz}
\usepackage{calrsfs}

\input xy
\xyoption{all}

\newcommand{\Spec}{{\rm Spec}}

\newcommand{\Hom}{ \,{\rm Hom} \,}
\newcommand{\Sym}{ \,{\rm Sym} \,}

\newcommand{\im}{ \,{\rm Im} \,}

{\bf}{\rm}
\newtheorem{theorem}{Theorem}[section]
\newtheorem*{theorem*}{Theorem}
\newtheorem{proposition}[theorem]{Proposition}
\newtheorem{corollary}[theorem]{Corollary}
\newtheorem{lemma}[theorem]{Lemma}
\newtheorem{definition}[theorem]{Definition}
\newtheorem{remark}[theorem]{Remark}
\newtheorem{conjecture}[theorem]{Conjecture}
\newtheorem{example}[theorem]{Example}

{\bf}{\it}

\newcommand{\RR}{{\mathbb R }}
\newcommand{\CC}{{\mathbb C }}

\newcommand{\ZZ}{{\mathbb Z }}
\newcommand{\PP}{ {\mathbb P }}
\newcommand{\QQ}{{\mathbb Q }}

\newcommand{\ff}{{\mathbf f }}

\newcommand{\calc}{\mathcal{C}}
\newcommand{\cald}{\mathcal{D}}

\newcommand{\calx}{\mathcal{X}}
\newcommand{\calo}{\mathcal{O}}

\newcommand{\cale}{\mathcal{E}}

\newcommand{\frakm}{\mathfrak{m}}

\newcommand{\tg}{\tilde{g}}

\newcommand{\reg}{\mathrm{reg}}
\newcommand{\codim}{\mathrm{codim}}

\newcommand{\Euler}{\mathrm{Euler}}

\newcommand{\mapd}[2]{J_k(#1,#2)}
\newcommand{\gnk}{\mathrm{Diff}_m \times \mathrm{Diff}_n}

\newcommand{\Tp}{\mathrm{Tp}}
\newcommand{\bbb}{\mathbf{b}}
\newcommand{\mult}{\mathrm{mult}}
\newcommand{\epd}[1]{\mathrm{eP}[#1]}

\newcommand{\mdeg}[1]{\mathrm{mdeg}[#1]}
\newcommand{\emu}{\mathrm{emult}}

\newcommand{\dist}{{\mathrm{dst}}}

\newcommand{\Span}{\mathrm{Span}}
\newcommand{\supp}{\mathrm{supp}}
\newcommand{\Curv}{\mathrm{Curv}}
\newcommand{\Conf}{\mathrm{Conf}}

\newcommand{\res}{\operatornamewithlimits{Res}}

\newcommand{\ires}{\res_{z_1=\infty}\res_{z_{2}=\infty}\dots\res_{z_k=\infty}}
\newcommand{\sires}{\res_{\mathbf{z}=\infty}}

\newcommand{\dbz}{\,d\mathbf{z}}

\newcommand{\HC}{\mathrm{HC}}

\newcommand{\coeff}{\mathrm{coeff}}

\newcommand{\symdot}{\mathrm{Sym}^{\le k}\CC^m}

\newcommand{\grass}{\mathrm{Grass}}

\newcommand{\flag}{\mathrm{Flag}}
\newcommand{\diff}{\mathrm{Diff}}
\newcommand{\TT}{\mathrm{T}}

\newcommand{\Lt}{\tilde{\Lambda}}

\newcommand{\bz}{\mathbf{z}}

\newcommand{\bx}{\mathbf{x}}
\newcommand{\tx}{\tilde{x}}
\newcommand{\ta}{\tilde{a}}
\newcommand{\bi}{\mathbf{i}}

\newcommand{\bff}{\mathbf{f}}

\newcommand{\bU}{\mathbf{U}}
\newcommand{\bW}{\mathbf{W}}

\newcommand{\kt}{{K}}

\newcommand{\OO}{\mathcal{O}}

\newcommand{\jetk}[2]{J_{k}({#1},{#2})}

\newcommand{\jetreg}[2]{J_{k}^{\mathrm{reg}}({#1},{#2})}

\newcommand{\jetnondeg}[2]{J_{k}^{\mathrm{nondeg}}({#1},{#2})}
\newcommand{\nondeg}{\mathrm{nondeg}}

\newcommand{\tc}{\hat T}

\newcommand{\GL}{\mathrm{GL}}
\newcommand{\sym}{\mathrm{Sym}}

\newcommand{\Hilb}{\mathrm{Hilb}}
\newcommand{\CHilb}{\mathrm{CHilb}}
\newcommand{\GHilb}{\mathrm{GHilb}}
\newcommand{\RHilb}{\mathrm{RHilb}}
\newcommand{\BHilb}{\mathrm{BHilb}}

\newcommand{\hch}{\mathrm{HC}}
\newcommand{\hcpt}{H^*_{\mathrm{cpt}}}
\newcommand{\C}{\mathbb{C}}

\def\a{\alpha}
\def\b{\beta}
\def\g{\gamma}
\def\d{\delta}
\def\e{\epsilon}

\def\l{\lambda}

\def\s{\sigma}

\def\vp{\varphi}

\def\L{\Lambda}

\title{Multiple-point residue formulas for holomorphic maps} 

\author{Gergely B\'erczi and Andr\'as Szenes}
\address{Aarhus University}
\email{gergely.berczi@math.au.dk}
\address{University of Geneva}
\email{andras.szenes@unige.ch}

\date{}
\begin{document}

\begin{abstract}
We develop a new approach to the study of the multipoint loci of holomorphic 
maps between complex manifolds. 
We relate the $k$-fold locus to the curvilinear component of the Hilbert scheme 
of $k$ points on the source space of the map, and using equivariant 
localisation, we derive a closed iterated multipoint residue formula. Our work 
is motivated by ideas and conjectures of M. Kazarian and R. Rim\'anyi on 
residual polynomials.
\end{abstract}

\maketitle
{\scriptsize
\tableofcontents}

\section{Introduction} 
Let $f: M \to N$ be a holomorphic map between two complex manifolds of 
dimensions $m \le n$. A point $p$ of $M$ is called a \textsl{$k$-fold point} of 
$f$ if the fiber $f^{-1}f(p)$ contains at least $k$ points. For sufficiently 
generic $f$ the closure of the $k$-fold point locus represents a cohomology 
class, and a \textit{multiple-point formula} or \textit{k-fold point formula} 
is a universal polynomial in the Chern classes of $f$ representing this cycle. 
Understanding the geometric origins of these universal polynomials, and finding 
explicit formulas for their coefficients is one of the central problems of 
enumerative geometry and global singularity theory. 

Another somewhat simpler problem is the identification of the class of
the locus $\Sigma_\a$ in $M$ where $f$ has a given singularity type $\a$. A 
(contact) singularity type is determined by a nilpotent algebra $A$, and the 
corresponding locus $\Sigma_A$ in $M$ is where $f$ has local algebra $A$, see 
\cite{arnold}. Assuming $M$
is compact and $f$ is sufficiently generic, the closure of $\Sigma_A$ is an 
analytic subvariety. Transplanting Thom's pioneering work \cite{thom} in the 
50's  into the complex/holomorphic world, one can say that the Poincar\'e dual 
class $[\Sigma_A]\in H^*(M,\mathbb{Z})$ is a polynomial 
$\Tp_A^{m \to n}(c_1,c_2,\ldots)$ in the Chern classes $c_i(f)$ of the difference 
bundle $TM-f^*TN$. These expressions are called \textit{the Thom polynomials of 
the 
singularity}; their full description has been a major open problem ever since 
then, see \cite{rimanyi,kazarian2,bsz,rf}. The theory of Thom polynomials provides a 
unified approach to many enumerative problems in algebraic geometry, which 
often can be reformulated as problems of counting singularities of certain maps 
of algebraic varieties.

In most of the enumerative problems one needs, however, an extension of Thom 
polynomial theory to the case of multisingularities. This natural, but 
significantly harder version studies the locus of points $p$ in $M$ where $f^{-1}f(p)$ has 
$k$ points with given singularity type at each point. 
In the present paper, we study the simplest multisingularity problem, the 
multipoint case, when $ p $ is a generic point of $ f $, i.e., when $ f $ is a 
local immersion of $ M $ into $ N $ at $ p $.

A key difference between the theories is that Thom polynomials are expressed in 
terms of the Chern classes $ c_i(f)=c_i(TM-f^*TN) $, which we will informally 
call \textit{local}, while the formulas expressing the classes of the 
multipoint loci involve more general classes of the form $ f^*f_*c_i(f) $, 
which we will call informally \textit{global}. The very first example is the 
double point formula:
\[     f^*f_*[1] - c_{n-m}(f).
 \]

Multisingularity theory is a classical subject in enumerative geometry, and its 
roots go back to the works of Pl\"ucker, and Salmon in the XIX. century (cf. for 
an account in Kleiman \cite{kleiman3}).
A major impetus to the subject was given by the work of Kleiman 
\cite{kleiman1,kleiman2}, who developed several methods to approach the 
subject: the 
\textit{iteration principle} and the method of \textit{Hilbert schemes}. His 
work however applied to the rather restricted set of corank-1 functions only. 
This assumption implies, in particular, the vanishing of a large 
set of characteristic classes of the map $f$, and thus the resulting 
formulas coincide with the conjectured general answer modulo an ideal only.

A breakthrough in the theory occurred following the work of Kazarian 
\cite{kaz97} and Rim\'anyi \cite{rimanyi2}.
In, \cite{kaz97,kazarian} Kazarian proposed a sieve-type multipoint formula (cf. 
\eqref{eq:sieve}), where the kernel of the sieve (which he named the 
\textit{residual polynomial}) is a polynomial depending only on the local 
classes. Experimenting with the restriction equations, Rim\'anyi in \cite{rimanyimarangell} then 
conjectured that the residual polynomials are simply the Thom polynomials for 
the $ A_k $ singularities with appropriate shifts in the depth and codimension. 
Assuming the sieve formula, 
Rim\'anyi and Marangell \cite{rimanyimarangell} also derived formulas for quadruple 
points of $ f $.

In \cite{kazarian_on_kleiman}, it was suggested that the right approach to the 
general 
case lies in a refined theory of intersection on singular spaces.
In the present paper we describe such an approach and, in fact give an exact 
meaning to Kazarian's sieve formula in this framework, and then prove Rim\'anyi's 
conjecture.

Let us describe our main result in some detail.
Let $M_k\subset M$ and $N_k=f(M_k) \subset N$ denote the $k$-fold locus, and $\bar{m}_k \in H^*(M), \bar{n}_k=k \cdot f_*m_k \in H^*(N)$ the multipoint classes. Note that the factor $k$ comes from the degree of the map $f$ on the $k$-fold locus. The $k$-fold locus in $M$ and $N$ have natural multiplicities, and we introduce 
\[m_k=(k-1)!\cdot \bar{m}_k\ \ \ \text{ and } n_k=k! \cdot \bar{n}_k,\]
where $n_k=f_*m_k$ holds. 
\begin{theorem}\label{mainthm1} Let $f:M \to N$ be a sufficiently generic holomorphic map between the complex manifolds $M$ and $N$ and $k>0$ and integer. The $k$-fold locus class in $N$ can be written as 
\begin{equation}\label{eq:sieve}
n_k=\sum_{i_1+\ldots +i_s=k} {k \choose i_1\ i_2\ \ldots \ i_s} S_{i_1} S_{i_2} \ldots S_{i_s}
\end{equation}
where $S_q=f_*R_q$ holds for the so-called residual polynomial $R_q$, which has the following closed expression for $1\le q\le k$:
\[R_q=\sires \frac{
(-1)^{q-1}Q_{q-1}(z_1,\ldots, z_{q-1})\,\prod_{1\le i<j \le q-1}(z_i-z_j)}{
\prod_{i+j \le l \le q-1}
(z_i+z_j-z_l)  (z_1\ldots z_{q-1})^{m-n+1}} \prod_{j=1}^{q-1} c_f\left(\frac{1}{z_j}\right)dz_i\]
where $c_f\left(\frac{1}{z_j}\right)=1+\frac{c_1(TM-f^*TN)}{z_j}+\frac{c_2(TM-f^*TN)}{z_j^2}+\ldots +$ is the total Chern class of $f$ at $1/z_j$. 
\end{theorem}
Explanations: 
\begin{itemize}
\item By sufficiently generic map we mean a stable map which Thom-Boardman transversal. 
\item The iterated residue $ \sires $ is equal to $(-1)^{q-1}$ times the 
coefficient of $(z_1\ldots z_{q-1})^{-1}$ in the Laurent expansion of the rational 
expression in the domain $z_1\ll \ldots \ll z_{q-1}$ and finally $Q_{q-1}$ is a 
polynomial invariant of Morin singularities.
\item $n_k$ is a polynomial in $c_i(f)$ and the Landweber-Novikov classes $s_J(f)$ on the target manifold $N$, where for $J=(j_1,j_2,\ldots)$  
 \[s_J(f)=f_*(c_1(f)^{j_1}c_2(f)^{j_2}\ldots )\]
\end{itemize}

In \cite{bsz} we proved that for a generic map $f: M\to N$ the Thom polynomial of the $A_{k-1}$ Morin singularity is
\[\Tp_{k-1}^{m \to n}=\sires \frac{
(-1)^{k-1}Q_{k-1}(z_1,\ldots, z_{k-1})\,\prod_{1\le i<j \le k-1}(z_i-z_j)}{
\prod_{i+j \le l \le k-1}
(z_i+z_j-z_l)  (z_1\ldots z_{k-1})^{m-n}} \prod_{j=1}^{k-1} c_f\left(\frac{1}{z_j}\right)dz_i\]

Note the only difference between the two expressions is the power of $z_1\ldots 
z_{k-1}$ in the denominator, which is equivalent to increasing the codimension of 
the map $f$ by $1$.This immediately implies Rim\'anyi's conjecture.

\begin{theorem}\label{mainthm2} The residual polynomial $R_q^{m \to n}$ for a generic map $f:M \to N$ is given by the Thom polynomial of $A_{q-1}$ singularity in codimension $n-m-1$, that is
\[R_q^{m\to n}=\Tp_{A_{q-1}}^{m \to n-1}.\]
\end{theorem}

Our approach relies on integration over Hilbert schemes of points, which goes 
back to our work on the Thom polynomials of $ A_k $-singularities. The 
fundamental steps are as follows:

\begin{itemize}
\item 	The description of the class of the $ k $-point loci as a generalized 
tautological integral on the Hilbert scheme of $k$ points on $ M $.
\item Introducing a K-theoretic sieve which, emulating Kazarian's sieve 
\eqref{eq:sieve}, 
localizes this integral in the neighborhood of the punctual locus of the 
Hilbert scheme.
\item A vanishing result, which eliminates the integral of all but the 
last element (the kernel) of the sieve.
\item The reduction of this integral to an integral over the curvilinear locus 
of the Hilbert scheme.
\item  An evaluation of this curvilinear integral using the techniques of 
\cite{bsz}.
\end{itemize}

The structure of the paper is as follows: we begin with a detailed overview of 
our strategy in \S\ref{sec:strategy}; after giving the necessary background on 
singularities in \S\ref{sec:background}, we explain the basic structural results on 
Thom polynomials and multipoint formulas in \S\ref{sec:thommulti}. Next, we 
turn to our main geometric object, the Hilbert scheme of $ k $ points on $ M $, 
for which 
we give the necessary background in \S\ref{sec:hilbertschemes}. In \S 
\ref{sec:tauintegral} 
we state the integral formula over the Hilbert scheme of points, and explain 
the proof for $k=2$ and $k=3$.
This is followed by the most technical part of the paper: the introduction of a local model of the Hilbert scheme in \S\ref{sec:model}
and the fully nested Hilbert scheme in \S \ref{sec:nested}. The last section, 
\S\ref{sec:equiv} contains the necessary technical results on equivariant 
localisation. 
The last three sections contain the proof of the main theorems: the integral 
formula on Hilbert schemes, first in a basic form and then in an enhanced 
push-forward version, and finally the multipoint formula.

\textbf{Acknowledgments}. The authors gratefully acknowledge useful discussions 
with Rahul Pandharipande and Rich\'ard Rim\'anyi.
This research was supported by SNF grant 175799, the NCCR SwissMAP and Aarhus University Research Foundation grant AUFF 29289.

\section{Overview of the strategy}\label{sec:strategy} 
In this section we give a brief outline of our strategy in order to explain the main ideas. The used concepts, terminology and their technical background is only briefly explained here, and the detailed definitions will be given in subsequent sections. This brief overview is intended to be a short guide to help the reader in following our argument.  
\subsection{Reformulation of the problem with Hilbert schemes} \label{subsec:reformulation}
Let $f: M\to N$ be a generic map between smooth compact complex manifolds of dimension $\dim(M)=m, \dim(N)=n$. For $m>2$ and $k\ge 8$ the Hilbert scheme $\Hilb^k(M)$ is singular and not irreducible. In fact, we will work with the ordered Hilbert scheme $\Hilb^{[k]}(M)$ which is a branched cover of the ordinary Hilbert scheme (see \S \ref{sec:nested}) with an associated ordered Hilbert-Chow morphism $HC: \Hilb^{[k]}(M) \to M^{\times k}$. The geometric component (also called the main component) of $\Hilb^{[k]}(M)$ 
\[\GHilb^k(M)=\overline{\{x_1 \sqcup \ldots \sqcup x_k: x_i \neq x_j\}} \subset \Hilb^{[k]}(M)\]
is the closure of the locus of reduced subschemes (a reduced point is a $k$-tuple of different points on $M$). We define
\[\GHilb^k(f)=\overline{\{x_1 \sqcup \ldots \sqcup x_k \in \Hilb^{[k]}(M): x_i \neq x_j, f(x_1)=\ldots =f(x_k) \in N\}}\subset \GHilb^k(M)\]
which is the closure in $\GHilb^k(M)$ of the $k$-fold set of $f$. This has a natural projection $\pi_1 \circ HC: \GHilb^k(f) \to M$ to the first factor, and the multipoint classes in $M,N$ are defined as 
\begin{equation}\label{strategyone}
m_k=[(\pi_1 \circ HC)(\GHilb^k(f)]\in H^*(M) \text{ and } n_k=[f(\GHilb^k(f))]\in H^*(N).
\end{equation}
The first key observation is that for sufficiently generic maps---which we call Thom-Boardman maps---$\GHilb^k(f)$ is a local topological complete intersection. We can summarize this in the following diagram: 
\begin{equation}\label{whatwewant} \xymatrix{& V \ar[d] & & \\
\GHilb^k(f) \ar@{^{(}->}[r]  & \GHilb^k(M) \ar[r]^-{f^{[k]}} & N^k & \nabla_N \ar@{^{(}->}[l] \ar[d]^{\pi_{\Delta N}}\\
& & & \Delta_N \ar@{^{(}->}[lu]^i}
\end{equation}
where 
\begin{itemize}
\item[(i)] $f^{[k]}=f^{\times k} \circ HC$ where $f^{\times k}: M^k \to N^k$ is the $k$th Cartesian product of $f$. 
\item[(ii)] $V=(f^*TN)^{[k]}/W$ is a to-be-defined vector bundle such that $\GHilb^k(f)$ is the zero set of a section. It is the quotient of the $k$th tautological bundle $(f^*TN)^{[k]}$ of the pull-back $f^*TN$.
\item[(iii)] The definiton of $W$ (and hence $V$) happens in two steps. Let $\nabla_N \subset N^k$ be (a tubular neighborhood of) the normal bundle of the diagonal $\Delta_N$. Then 
\[\GHilb^k(f) \subset (f^{[k]})^{-1}(\nabla_N)\]
and $\GHilb^k(f)$ is compact in $(f^{[k]})^{-1}(\nabla_N)$. We will first define $W$ over $(f^{[k]})^{-1}(\nabla_N)$, and take a topological extension over $\GHilb^k(M)$. Our calculations will take place in $\nabla_{N}$ and $(f^{[k]})^{-1}(\nabla_{N})$, and in the compactly supported cohomology of the open set $(f^{[k]})^{-1}(\nabla_{ N})$. 
\end{itemize}
We can reformulate the multipoint locus in \eqref{strategyone} as
\begin{equation}\label{strategytwo}
n_k=[f(\GHilb^{k}(f))]=\pi_*f^{[k]}_*\left[\Euler(V)\right]_{(f^{[k]})^{-1}(\nabla_{N}) \subset\GHilb^k(M)}
\end{equation}
This reduces the calculation to push forward of the tautological-type class 
\[\Euler((f^*TN)^{[k]}/W)\]
along $f$. Our push-forward formula will be an immediate extension of an integration formula over the Hilbert scheme. Hence, as a by-product, we develop a new approach to compute certain tautological integrals over $\GHilb^k(M)$. Tautological integration has attracted considerable attention in recent years. Over surfaces and some special threefolds there are old and new techniques and methods, including virtual integration over spaces endowed with so-called perfect obstruction theory which gives a $2$-term deformation complex \cite{virtual}. However, virtual integration has technical limitations and it only works in low dimensional geometries. 

In this paper we focus on specific tautological classes satisfying certain geometric conditions which we call properly supported classes (see Definition \ref{geocond}).

\subsection{Push forward along $f^{[k]}$} The core technical result of the present paper is a formula which reduces integration over $\GHilb^k(M)$ (and push-forward along $f^{[k]}$) to equivariant integration over the curvilinear component of the punctual Hilbert scheme over $\CC^m$ supported at the origin. 
\[\xymatrix{\int_{\GHilb^k(M)} \Phi(V^{[k]}) \ar@^{~>}[r] & \int_{\CHilb^k_0(\CC^m)}} \frac{\Phi(V^{[k]})}{\Euler(B)} \]
for some tautological bundle $B$. Here $\Phi$ is a Chern polynomial in the Chern roots of the tautological bundle $V^{[k]}$ for some bundle $V$ over $M$, satisfying a geometric condition, called properly supported, formulated in Definition \ref{geocond}. For integration over $\CHilb^k(M)$ we adapt results in \cite{bsz,kazarian,b0}. The key steps of our approach can be summarised as follows.  

\noindent \textbf{Step 1: The Sieve } We introduce a master blow-up $\pi_\Lambda: N^k(M) \to \GHilb^k(M)$ which we call the fully-nested Hilbert scheme. It admits a blow-up morphism $\pi_\a: N^k(M) \to \GHilb^\a(M)$ to approximating Hilbert spaces $\GHilb^\a(M)$ defined for partitions $\a=\a_1 \sqcup \ldots \sqcup \a_r =\{1,\ldots, k\}$ of the $k$ points into smaller groups. For a bundle $V$ over $M$ this defines approximating tautological bundles $V^\a$ over $\GHilb^\a(M)$, and the pull-backs $\pi_\a^*V^\a$ of tautological bundles from the different approximating sets provide enough approximating bundles over $N^k(M)$ to define a sieve formula, following ideas of Li and Rennemo. This sieve gives $K$-theoretic decomposition of $V^{[k]}$ as a sum 
\[V^{[k]}=\sum_{\a \in \Pi(k)}W^\a\]
of bundles indexed by partitions of $\{1, \ldots k\}$. For any partition $\{1,\ldots, k\}=\a_1 \sqcup \ldots \sqcup \a_s$ the bundle $W^\a$ is supported on the approximating punctual subset  
\[\supp(W^\a)=N^{\a}_0(X)=\HC^{-1}(\Delta_{\a})\]
and $W^\a$ is a linear combination of the classes $V^{[\b]}$ for those partitions $\b \in \Pi(k)$ which are refinements of $\a$. In particular, for $\a=\Lambda=\{1,\ldots, k\}$ the trivial partition
\begin{equation}\label{deepestterm}
W^\Lambda=\pi_\Lambda^* V^{[k]}+\sum_{\a \in \Pi(k)\setminus \Lambda} \gamma_\a \pi_\a^*V^\a   
\end{equation}
is a linear combination of tautological classes pulled-back from the tautological approximating classes. Let $\Phi$ be a Chern polynomial in the Chern roots of $V^{[k]}$. Then the sieve reduces the push-forward operation $f^{[k]}_*$ (and integration) over $\GHilb^k(M)$ to $N^k_\nabla(M)$:
\begin{equation}\label{step1} 
\xymatrixcolsep{5pc} \xymatrix{f^{[k]}_*\left[\Phi(V^{[k]}) \right]_{\GHilb^k(M)} \ar@^{~>}[r]^-{\mathrm{Sieve}} & f^{[k]}_*\pi_{\Lambda *} \left[\Phi(W^\Lambda) \right]_{N^k_\nabla(M)}+\sum_{\a \in \Pi_k \setminus \Lambda} \left[\Phi(W^\a) \right]_{N^\a_\nabla(M)}}
\end{equation}
where the terms in the sum are determined inductively, and $N^\a_\nabla(M)$ stands for a small, tubular neighborhood of the $\a$-punctual part $N^\a_0(M)$. This leaves us with calculating the leading term $f^{[k]}_*\pi_{\Lambda *} \left[\Phi(W^\Lambda) \right]_{N^k_\nabla(M)}$. 

\noindent \textbf{Step 2: From $M$ to $M=\CC^m$.} Since $\Phi(W^\Lambda)$ is compactly supported in a neighborhood $N^k_\nabla(M)$ of the punctual part $N^k_0(M)$, we can work locally, and pull-back our local neighborhood $N^k_\nabla(M)$ from a universal bundle. The Chern-Weil map then reduces push-forward (integration) to equivariant push-forward (integration) over $M=\CC^m$. The punctual nested Hilbert scheme $N^k_0(\CC^m)$ is not irreducible, but $\Phi(W^\Lambda)$ is supported in a tubular neighborhood which is the union of tubular neighborhoods of its components. 
\begin{equation}\label{step2} 
\xymatrixcolsep{5pc} \xymatrix{f^{[k]}_*\pi_{\Lambda *} \left[\Phi(W^\Lambda)\right]_{N^k_\nabla(M)} \ar@^{~>}[r]^-{\mathrm{Chern-Weil}} & f^{[k],T}_*\pi_{\Lambda *}^T \left[\Phi(W^\Lambda)\right]_{N^k_\nabla(\CC^m)}}
\end{equation}
where $T$ stands for torus-equivariant push-forward with $T\subset \GL(m,\CC)$ the maximal torus.

\noindent \textbf{Step 3: First Residue Vanishing Theorem.} The next crucial step in the strategy is the construction of a blow-up $\hat{N}^k_\nabla(\CC^m) \to N^k_\nabla(\CC^m)$ and $\widehat{\GHilb}^k_\nabla(\CC^m) \to \GHilb^k_\nabla(\CC^m)$ such that $\hat{N}^k_\nabla(\CC^m) \to \widehat{\GHilb}^k_\nabla \to \flag_{k-1}(T\CC^m)\to \CC^m$ fibers over the flag $\flag_{k-1}(T\CC^m)$ of $(1,2,\ldots, k-1)$ dimensional subspaces of the tangent bundle, which itself fibers over $\CC^m$. The fiber over the origin is the balanced Hilbert scheme $\BHilb^{k}_0(\CC^m)$ consisiting of subschemes with baricenter at the origin. The key observation is that equivariant localisation over $\BHilb^k_0(\CC^m)$ can be transformed into an iterated residue. This iterated residue formula shows some unexpected (and miraculous) symmetries which arranges fixed points into cycles and incarnated in the form of residue vanishing properties. Our first residue vanishing theorem tells that the contribution of the terms $\Phi(V^\a)$ in \eqref{deepestterm} to the push-forward (integral) formula is zero for $\a \neq \Lambda$. But the leading term is a pull-back from the Hilbert scheme, hence 
\begin{equation}\label{step3} 
\xymatrixcolsep{5pc} \xymatrix{ f^{[k],T}_*\pi_{\Lambda *}^T \left[\Phi(W^\Lambda)\right]_{N^k_\nabla(\CC^m)} \ar@^{~>}[r]^-{\mathrm{\text{Residue Vanishing I}}} & f^{[k],T}_*\pi_{\Lambda *}^T \left[\pi_\Lambda^* \Phi(V^{[k]}) \right]_{N^k_\nabla(T\CC^m)}=f^{[k],T}_*\left[\Phi(V^{[k]}) \right]_{\GHilb^k_\nabla(\CC^m)}}
\end{equation}
In short, the sieve and the first residue vanishing theorem allows us to go back to the Hilbert scheme and use equivariant techniques over $M=\CC^m$.
In this paper we will work with equivariant classes $\Phi(V^{[k]})$ which we call properly supported (see Definition \ref{geocond}): their support intersects the punctual Hilbert scheme in the curvilinear component:
\[\mathrm{supp}(\Phi(V^{[k]}))\cap \GHilb^k_0(\CC^m) \subset \CHilb^k(\CC^m),\]
and the support is locally irreducible at these points. 
In particular, in the next step we explain a local model for stable maps which shows that this assumption holds in the multipoint calculation. This reduces the push-forward (integration) to a tubular neighborhood $\CHilb^k_\nabla(\CC^m)$ of $\CHilb^k(\CC^m)$.
\begin{equation}\label{step3b} 
f^{[k],T}_*\left[\Phi(V^{[k]}) \right]_{\GHilb^k_\nabla(\CC^m)}=f^{[k],T}_*\left[\Phi(V^{[k]}) \right]_{\CHilb^k_\nabla(\CC^m)}
\end{equation}

\noindent \textbf{Step 4: The local model of stable maps} We prove that for sufficiently generic maps $f$ (we call these stable Thom-Boardman maps, whose complement is countable union of Zariski closed subsets in the space of maps) $\GHilb^k(f)$ intersects the punctual part $\GHilb^k_0(M)$ in the curvilinear component $\CHilb^k(M)$:
\[\GHilb^k_0(f):=\GHilb^k(f)\cap \GHilb^k_0(\CC^m) \subset \CHilb^k(\CC^m)\]
Moreover, the small neighborhood 
\[\GHilb^k_\nabla(f)=\GHilb^k(f) \cap \GHilb^k_\nabla(\CC^m)\]
 of $\GHilb^k_0(f)$ in $\GHilb^k(f)$ can be modeled as a bundle over $\GHilb^k_0(f)$. More precisely, we prove the existence of a bundle $B$ over the curvilinear component such that there is a diagram 
\begin{equation}\label{localmodel}
\xymatrix{\GHilb^k_\nabla(f) \ar[d]  \ar[r]^\nu &  B  \ar[d]  \\
  \GHilb^k_0(f) \ar@{^{(}->}[r] &  \CHilb^k(M)}
 \end{equation}
where $\nu$ is a topological isomorphism from $\CHilb^k_\nabla(f)$ to the total space of a bundle $B$. We have seen that the support of $\Euler(f^*TN^{[k]}/W)$ is the homology cycle represented by $\GHilb^k(f)$, hence the Thom isomorphism combined with Step 3 gives 
\[\xymatrixcolsep{5pc} \xymatrix{f^{[k],T}_*\left[\Euler(f^*TN^{[k]}/W)^{+}\right]_{\CHilb^k_\nabla(\CC^m)} \ar@{~>}[r]^-{\text{Local model}} &  f^{[k],T}_*\left[\frac{\Euler(f^*TN^{[k]}/W)}{\Euler(B)}\right]_{\CHilb^k(\CC^m)}}\]
More generally, this local model holds for classes $\Phi(V^{[k]})$ which are properly supported in the sense of Definition \ref{geocond}. Then  
\[\xymatrixcolsep{5pc} \xymatrix{f^{[k],T}_*\left[\Phi(V^{[k]}\right]_{\CHilb^k_\nabla(\CC^m)} \ar@{~>}[r]^-{\text{Local model}} &  f^{[k],T}_*\left[\frac{\Phi(V^{[k]})}{\Euler(B)}\right]_{\CHilb^k(\CC^m)}}\]
\noindent \textbf{Step 5: Second Residue Vanishing Theorem} The problem is now reduced to push-forward (integration) over the curvilinear component $\CHilb^k(\CC^m)$. We adapt and generalise the formula developed in \cite{bsz,kazarian} to perform this integration. The key step is to prove a  second residue vanishing theorem, which roughly says that only one fixed point in the residue formula has nonzero contribution. This allows us to ignore the complexity of the singular spaces involved , and reduces it to understand the geometry of a distinguished fixed point.
\[\xymatrixcolsep{5pc} \xymatrix{f^{[k],T}_*\left[\frac{\Phi(V^{[k]})}{\Euler(B)}\right]_{\CHilb^k(\CC^m)} \ar@{~>}[r]^-{\substack{\text{Localisation}\\ \text{Residue Vanishing II}}} & \text{Residue multipoint formula}}\]
The final residue multipoint formula gives the deepest term in the sieve, and can be identified by the residual polynomial in Theorem \ref{mainthm1}. It is a shifted version of the Thom polynomial formula derived in \cite{bsz}, which means a shift in the codimension of the map, hence we conclude Theorem \ref{mainthm2}. 

\subsection{Additional remarks}The core idea in Step 2 is to turn the localisation formula into an iterated reside by constructing partial blow-ups of the spaces involved. Here we give some details for the Hilbert spaces involved. The fibration for the master blow-up space will just combine this with the blow-up maps. 
The original set-up can be summarised as  
\begin{equation}
\xymatrix{&  &   (f^*TN)^{[k]}/W,\ B \ar[d]  \\
 \GHilb^k_0(f) \ar@{^{(}->}[r] &\CHilb^k(\CC^n) \ar@{^{(}->}[r] &  \GHilb^k(\CC^n)  }
\end{equation}
where the bundles have sections $s:\GHilb^k(\CC^m) \to (f^*TN)^{[k]}/W$, $t: \GHilb^k(\CC^m) \to B$ such that their zero cycles are:
\[s^{-1}(0)=\GHilb^k(f),\ \ t^{-1}(0)=\GHilb^k_0(\CC^m)\]
and for Thom-Boardman $f$, the intersection of these cycles sit in the curvilinear component:
\[\GHilb^k_0(f)=\GHilb^k(f)\cap \GHilb^k_0(\CC^m) \subset \CHilb^k(\CC^m)\]
We will construct partial blow-ups such that our diagram fibers over the flag bundle 
\[\flag_{k-1}(T\CC^m).\]
The idea roughly is the following: the first $j$ points $p_1,\ldots, p_j$ of a generic point $\{p_1,\ldots, p_k\} \in \GHilb^k(\CC^m)$ in the ordered Hilbert scheme span a $j-1$-dimensional subspace in the tangent space $T_p\CC^m$ where $p=\frac{1}{j}(p_1+\ldots. +p_j)$ is the center of gravity. 

These partial blow-up spaces are still singular, but they all sit in the smooth ambient space $\grass_{k-1}(S^\bullet T\CC^m)$, which is a partial blow-up of the Grassmannian bundle whose fiber over a flag $(V_1 \subset \ldots \subset V_{k-1}) \in \flag_{k-1}(T_p\CC^m)$ is the Grassmannian of $k$-codimensional ideals in the symmetric algebra $S^\bullet V_{k-1}$. A crucial fact for the residue vanishing properties is that the dimension of this fiber depends only on $k$, not on $m$, and hence for $k\ll m$ the dimension of the fiber is much smaller than that of the total space. Moreover, both $B$ and $V$ are restrictions of bundles over $\grass_{k-1}(S^\bullet T\CC^m)$. 
 
\begin{equation}
\xymatrix{&   &   & V,\ B \ar[d] \\
\widehat{\GHilb}^k_0(f) \ar@{^{(}->}[r]  \ar[rd]&  \widehat{\CHilb}^k(\CC^m) \ar@{^{(}->}[r] \ar[d]&  \widehat{\GHilb}^k(\CC^m) \ar[ld]  \ar@{^{(}->}[r] & \grass_{k-1}(S^\bullet T\CC^m) \ar[lld] \\
  & \flag_{k-1}(T\CC^m) \ar[d] & & \\
  & \CC^m &&}
\end{equation}

\section{Structure theorems for singularities of maps}\label{sec:background}

\subsection{Singularities and stability notions}
We introduce the notation $J(m)=J(m,1)$ for the space of germs of holomorphic functions $(\CC^m,0) \to (\CC,0)$; in local coordinates $(x_1,\ldots,x_m)$ at the origin,  
\[J(m,1)=\{h\in\CC[[x_1\ldots x_m]];\; h(0)=0\}\] 
is the algebra of power series without a constant term. Let $J_k(m,1)$
be the space of $k$-jets of holomorphic functions on $\CC^n$ near the
origin, i.e. the quotient of $J(m,1)$ by the ideal of those
power series whose lowest order term is of degree at least $k+1$.


Our basic object is $J_k(m,n)$, the space of $k$-jets of holomorphic
maps $(\CC^m,0)\to(\CC^n,0)$. This is a finite-dimensional complex
vector space, which one can identify as $J_k(m,1) \otimes \CC^n$; hence
$\dim J_k(m,n) =n\binom{m+k}{k}-n$.  We will call the elements of
$J_k(m,n)$ as map-jets of order $k$, or simply map-jets. 


One can compose map-jets via substitution and elimination of terms
of degree greater than $k$; this leads to the composition maps
\begin{equation}
  \label{comp}
J_k(m,n) \times J_k(n,p) \to J_k(m,p),\;\;  (\Psi_2,\Psi_1)\mapsto
\Psi_2\circ\Psi_1.
\end{equation}
When $k=1$, $J_1(m,n)$ may be identified with $n$-by-$m$
matrices, and \eqref{comp} reduces to multiplication of matrices.  By
taking the linear parts of jets, we obtain a map
\[\mathrm{Lin}: J_k(m,n)\to \Hom(\CC^m,\CC^n),\] 
which is compatible with the compositions \eqref{comp} and matrix
multiplication. The  set
\[
\diff_k(m) = \{\Delta\in \mapd mm;\; \mathrm{Lin}(\Delta) \text{ invertible}\}.
\]
is an algebraic group under the composition map \eqref{comp}, and this gives the 
so-called {\em left-right} action (also called $\mathcal{A}$-action) of the group $\diff_k(m) \times \diff_k(n)$ on $\mapd mn$:
\[ [(\Delta_L,\Delta_R),\Psi]
 \mapsto \Delta_L\circ\Psi\circ \Delta_R^{-1}
 \] 
 Singularity theory, in the sense that we are considering here, studies the left-right-invariant algebraic
 subsets of $J_k(m,n)$. A natural class of such  invariant subsets are given by jets with the same local algebra; the local algebra (or nilpotent algebra) of the jet $f=(f_1,\dots, f_n) \in J_k(m,n)$ is defined as 
 \[A_f=J_k(m,1)/(f_1,\dots,f_n)\]
 the algebra of function germs modulo the
 ideal generated by the components of $f$. This is a finitely generated nilpotent algebra, whose dimension over $\CC$ is called the multiplicity of $f$. We call the singularity isolated, if $\mu(f)<\infty)$. For isolated singularities the multiplicity counts the maximum number of points in a fiber of $f$ at a neighborhood of the singularity, that is 
 \[\dim(A_f)+1=\max_{0<|y| < \epsilon}(\# f^{-1}(y))\]

 Now let $A$ be a finitely generated nilpotent algebra, and consider the set
\begin{equation}
  \label{defconea}
\Theta_A^{m\to n} = \{(f_1,\dots,f_n) \in J_k(m,n);\; A_f \cong A\}
  \end{equation}
 Note that $\Theta_A$ is $\diff_k(m) \times \diff_k(n)$-invariant, and it serves as a prototype of singularity classes.  A key observation though is that two map-jets with
  the same nilpotent algebra may be in different $\diff_k(m) \times \diff_k(n)$-orbits. However, there is an algebraic group action on $J_k(m,n)$ whose orbits are
  exactly the sets $\Theta_A^{m\to n}$ for various nilpotent algebras
  $A$. The so-called contact equivalence or $\mathcal{K}$-equivalence was introduced by John Mather. Two germs $f,f' \in J_k(m,n)$  are contact equivalent if there is a diffeomorphism germ $\a \in \diff_k(m)$ and a map germ $M \in J_k(\CC^m,\GL(n))$ such that
\[f'(x)= M(x) \cdot f(\a(x))\] 
This can be also thought as a group action: the group $\mathcal{K}$ is
\[\mathcal{K} = \diff_k(m) \times J_k(\CC^m, \GL(n))\]
acts (from the left) on $J_k(m,n)$ by
\[((\a,M) \cdot f = M \cdot (f \circ \a^{-1})\]



\begin{theorem}[Mather, \cite{mather}] Two map germs are $\mathcal{K}$-equivalent if and only if their ideals are
taken into each other by a diffeomorphism in $\diff_k(m)$ for some $k$. Hence two finitely determined map germs are K-equivalent if and only if their
local algebras (or equivalently, their quotient algebras) are isomorphic.
\end{theorem}

Using the fact that  $\mathcal{K}_d$ is connected, 
it is not difficult to derive the following properties of $\Theta_A$.
  \begin{proposition}
    \label{propgena} Let $A$ be a nilpotent algebra such that
    $A^{d+1}=0$ and $m\geq\dim(A/A^2)$.  Then for $n$ sufficiently
    large, $\Theta_A^{m\to n}$ is a nonempty, $\diff(m) \times \diff(n)$-invariant, irreducible quasi-projective algebraic variety of
    codimension $(n-m+1)\dim(A)$ in $J_k(m,n)$.
\end{proposition}
Note that the codimension of $\Theta_A$ depends only on the codimension of the map
$n-m$ and does not depend on $k$.
In the present paper, we will study certain  rough topological
invariants of  contact singularities; these invariants depend only
on the closure of the singularity locus in $J_k(m,n)$. As it turns
out, in an asymptotic sense, the closures of contact orbits are also
closures of left-right orbits, hence, from our point of view, these
two types of singularity classes are closely related.

While we will not need this statement, we describe it in some
details for reference.  Roughly, we claim that for fixed $A$ and
$r$, and sufficiently large $m$, there is a dense left-right orbit
in $\Theta_A^{m\to m+r}$.

Let $r$ be a nonnegative integer. An {\em unfolding} of a map-jet
$\Psi \in J_k(m,n)$ is a map-jet $\widehat{\Psi}\in \mapd{m+r}{n+r}$ of
the form
\[(x_1,\ldots ,x_m,y_1,\ldots, y_r)\mapsto
(F(x_1,\ldots, x_m,y_1,\ldots, y_r),y_1,\ldots, y_r)
\]
where $F\in \mapd {m+r}n$ satisfies
\[F(x_1\ldots, x_m,0,\ldots,
0)=\Psi(x_1,\ldots, x_m).\] The {\em trivial unfolding} is the
map-jet
\[(x_1,\ldots, x_m, y_1,\ldots, y_r)\to (\Psi(x_1,\ldots
,x_m),y_1, \ldots ,y_r).\]

\begin{definition}
  A map-jet $\Psi \in J_k(m,n)$ is {\em stable} if all unfoldings of
  $\Psi$ are left-right equivalent to the trivial unfolding.
\end{definition}

Informally, a germ of a holomorphic map $f:M \rightarrow N$ of
complex manifolds at a point $p\in M$ is stable if for any
small deformation $\tilde{f}$ of $f$, there is a point in the
vicinity of $p$ at which the germ of $\tilde{f}$ is left-right
equivalent to the germ of $f$ at $p$.
Now we can formulate the relationship between contact and left-right
orbits precisely.

\begin{theorem}[Mather, \cite{mather}]
\label{propstable}
  \begin{enumerate}
  \item If $\,\widehat\Psi$ is an unfolding of $\Psi$, then
    $A_{\widehat\Psi}\cong A_\Psi$.
  \item Every map germ has a stable unfolding, these are also called stable representatives of the local algebra $A_\Psi$.
  \item If a map germ is stable, then its left-right orbit is dense in
    its contact orbit.
    \end{enumerate}
\end{theorem}

Given a finitely generated nilpotent algebra $A$, we can ask the following natural questions:
\begin{enumerate}[(i)]
\item Can we find a germ  $f:(\CC^m,0) \to 
(\CC^n,0)$ with $A_f=A$, and for what parameters $m,n$?
\item How can we construct a stable unfolding of $f$? 
\end{enumerate}

Let $\frakm$ be the maximal ideal of $A$, and write a minimal presentation 
\[A=\CC[x_1,\ldots,x_s]/(r_1,\ldots, r_p)\] 
with a smallest possible number $s=\dim(\frakm /\frakm^2)$ of generators and relations. Then the germ $f \in J(s,p)$ defined as
\[g_A:(x_1,\ldots, x_s) \mapsto (f_1,\ldots r_p)\] 
has local algebra $A_f=A$. Moreover, $s$ is the minimal possible source dimension and $l=p-s$ is the
minimal codimension for any map with local algebra $A$. Note that for this minimal presentation $r_i\in \frakm^2$, so $df=0$. This germ $g_A$ is
 called the \textsl{genotype} of $A$. 
 
 The genotype is not stable, but we can stabilize it by a nontrivial unfolding as follows. The quotient algebra 
\[Q=J(s,p)/\langle \frac{\partial f}{\partial x_i}, r_i^{(j)} \rangle\] 
is a finite-dimensional vector space, where $r_i^{(j)} \in J(s,p)$ stands for the germ we get by putting $r_i$ as the $j$th coordinate function and putting $0$ elsewhere. So one can write 
\[Q=\mathrm{Span}_\CC(e_1,\ldots, e_p, \phi_1,\ldots, \phi_r)\]
for some $\phi_i \in J(s,p)$. Then the unfolding
 \begin{eqnarray}
g:(\CC^s \times \CC^r,0)\rightarrow (\CC^p \times \CC^r,0)\\
(x,y)\rightarrow (\phi(x)+\sum_1^r y_i\phi_i(x), y)
\end{eqnarray}
is stable, also called the {\it miniversal unfolding} of $g$ and a \textsl{prototype} of $A$. To obtain a
stable germ with local algebra $A$ and codimension $l>p-s$ we take the
miniversal unfolding of $(x_1,\ldots,x_s)\rightarrow (r_1,\ldots, r_p,0,\ldots
,0)$ with $q$ $0$s where $q+p-s=l$.

\subsection{Stratifications of maps} Consider a holomorphic map $f:M \to N$. We give a short survey on stability, transversality and stratifications of maps, essential to state the structural theorems in the next section. We refer \cite{arnold} for details.  

Two maps $f,g: M\to N$ are said to be {\it $\mathcal{A}$-equivalent} if there exist diffeomorphisms $\phi \in \diff(M), \psi \in \diff(N)$ such that $f=\phi \circ g \circ \psi^{-1}$ holds. 
The map $f: M \to N$ is said to be {\it stable} if any map that is close enough (together with its derivatives) to $f$ is $\mathcal{A}$-equivalent to $f$.

The set $J_k(M,N)$ of all $k$-jets of smooth maps of a fixed manifold $M$ into $N$ is called the space of $k$-jets of maps. The $k$-jet extension $j^kf$ is a map $M \to J_k(M,N)$ which sends each point of $p\in M$ into the $k$-jet of $f$ at $p$ (and hence $j^kf(p)\in J_k(m,n)$).  

A {\it stratified submanifold} of a smooth manifold is a finite union of pairwise-disjoint smooth path-connected manifolds with the property that the closure of any stratum consists of that stratum and a finite union of strata of lower dimensions. A map is {\it transverse to a stratified submanifold} if it is transverse to each of its strata. Thom's Transversality Theorem (\cite{arnold}) implies that the set of transverse maps is the intersection of countable number of open, dense subsets. 

Let $S,U$ be strata of a stratification such that $S \subset \bar{U}$. {\it Whitney's condition A} requires that if a) $s\in S$ and a sequence $u_n \in U$ converging to $s$, and b) the tangent space $T_{u_n}U$ converges in some Grassmannian to a limit $V$ then $T_sS \subset V$. 
\begin{proposition}
Suppose that $Z \subset J_k(M,N)$ is a closed subset endowed with a stratification satisfying Whitney condition A. Then the set of maps $f:M \to N$ whose $k$-jet is transversal to all strata of the stratification is open, dense in $J(M,N)$.
\end{proposition} 

\subsection{Thom-Boardman classes} 

Classification of contact singularity classes and the hierarchy of these classes is out of reach, but there exist courser but full classifications of singular points which also satisfy nice geometric properties.  We will use the Thom-Boardman classification in this paper, introduced by Thom \cite{thom} and Boardman \cite{boardman}  (see also \cite{arnold}). 

Let $I=(i_1\ge i_2 \ge \ldots \ge i_k)$ be a finite nonincreasing sequence of nonnegative integers. 
\begin{definition}[Thom, \cite{thom}]\label{thomboardmandef1} For a smooth map $f: M \to N$, we define $\Sigma^I(f) \subset M$ inductively as follows. Let 
\[\Sigma^i(f)=\{x \in M : \dim(\ker(d_xf))=i\}\]
Suppose that $\Sigma^i(f) \subset M$ is a smooth submanifold and $0\le j\le i$; then we define $\Sigma^{ij}(f)$ to be $\Sigma^j(f|_\Sigma^i(f))$. Inductively, assume that $\Sigma^{i_1,\ldots, i_r}(f) \subset M$ is smooth and define
\[\Sigma^{i_1,\ldots, i_{r+1}}(f) = \Sigma^{i_{r+1}}((f|_{\Sigma^{i_1,\ldots, i_r}(f)})\]
\end{definition}

Boardman in \cite{boardman} proposed a different definition to avoid the condition on the smoothness of the sets $\Sigma^I(f)$. For $I=(i_1,\ldots, i_k)$ he defined a smooth, not necessarily closed submanifold $\Sigma^I \subset J_k(M,N)$ of the $k$-jet space, independent of the choice $f$. 

\begin{definition} Let $B$ be an ideal in $\mathcal{E}_m=J(m)\oplus \CC$, the ring at the origin of germs of holomorphic functions on $\CC^m$. The $k$th Jacobian extension $\Delta_k(B)$ is the ideal generated by $B$ together with the $k \times k$ determinants $\det(\partial \varphi_i/\partial x_j)$ formed by partial derivatives of functions in $B$ in some fixed local coordinates at the origin on $\CC^m$. It will be convenient to introduce the relabeling $\Delta^k(B):=\Delta_{m-k+1}(B)$. Note that $\Delta^k(B)$ actually does not depend on the choice of these local coordinates.
\end{definition}
We get a chain of ideals
\[B = \Delta^0(B) \subseteq  \Delta^1(B) \subseteq \ldots  \subseteq \Delta^{m+1}(B) = \mathcal{E}_m\]
The largest $\Delta^k$ such that $\Delta^k(B) \neq \mathcal{E}_m$ is called the critical Jacobian extension. Note that the critical extension of an ideal $B$ is $\Delta^{m-r}=\Delta_{r+1}$, where $r=\dim_\CC(\mathfrak{m}^2+B/\mathfrak{m}^2)$.
\begin{definition}[Boardman, \cite{boardman}] Let $I=(i_1,\ldots, i_k)$ be a nonincreasing set of nonnegative integers. The map jet $f=(f_1,\ldots, f_n) \in J_k(\CC^m,\CC^n)$ belongs to $\Sigma^I$ (and said to have Boardman symbol $I$) if the ideal $B=(f_1,\ldots, f_n) \subset \mathcal{E}_m$ has successive critical extensions 
\[\Delta^{i_1}B,\ \  \Delta^{i_2}\Delta^{i_1}B,\ \  \Delta^{i_3}\Delta^{i_2}\Delta^{i_1}B ,\ldots \]
The Boardman class $\Sigma^I$ is a smooth submanifold of $J_k(m,n)$.
\end{definition}
\begin{proposition}[\cite{arnold}] Let $I=(i_1,\ldots, i_k)$ be a Boardman symbol. 
\begin{enumerate}
\item The codimension of the submanifold $\Sigma^I$ in $J_k(m,n)$ is given by the formula
\[\mathrm{codim}(\Sigma^I)=(n-m+i_1)\mu(i_1,\ldots, i_k)-(i_1-i_2)\mu(i_2,\ldots, i_k)-\ldots -(i_{k-1}-i_k)\mu(i_k)\]
where $\mu(i_1,\ldots, i_s)$ is the number of sequences $(j_1,\ldots, j_s)$ satisfying $j_1>0$, $j_1 \ge \ldots j_s$ and $i_s \ge j_s$.
\item The multiplicity of $f\in \Sigma^I(m,n)$ is
\[\dim(A_f)=i_1+\ldots +i_k+1\]
\end{enumerate}
\end{proposition}

\begin{definition}[Boardman, \cite{boardman}] For a map $f: M \to N$ between manifolds and $x \in M$ let $f=(f_1, \ldots ,f_n)$ be the coordinate functions of $f$ in some local coordinate system. Then $x \in \Sigma^I(f)$ if $B=(f_1,\ldots, f_n) \in \Sigma^I$.
\end{definition} 

\begin{remark} Thom-Boardman singularities of order $k$ are $k$-determined, that is, it is enough to
look at the first $k$ differentials of a map to decide whether it belongs to the given Thom-
Boardman class (this is clear from Boardman's definition). They are also stable in
the sense that if $f:\CC^m \to \CC^n$ belongs to $\Sigma^I$,then so does $f \oplus id_\CC: \CC^{m+1} \to \CC^{n+1}$.
\end{remark}

Let $J_k(M,N)=J_k(M) \otimes \CC^n \to M$ denote the $k$-jet bundle whose fiber over $x \in M$ is the set $J_k(m,n)$ of $k$-jets of germs at $x$. For a map $f: M \to N$ its $k$-jet extension defines a section $s_f$ of $J_k(M,N)$. 

\begin{theorem}[Boardman, \cite{boardman}] \begin{enumerate}
\item Assume that the $k$-jet extension $j^kf$ of the map $f: M  \to N$ is transverse to the manifolds $\Sigma^I$ (we call these maps Thom-Boardman maps). Then 
\[\Sigma^I(f)=(j^kf)^{-1}(\Sigma^I)\]
for any $I$ of length $k$. 
\item Any smooth map $f:M \to N$ can be arbitrarily well approximated , together with any number of its derivatives, by a Thom-Boardman map.
\end{enumerate}
\end{theorem}

In fact, we have the following stronger result, which follows from Thom's transversality theorem.
\begin{proposition}[(see also Wilson \cite{wilson})] The set of Thom-Boardman maps is residual, that is, the intersection of countable number of open dense subsets in $J_k(M,N)$. In this sense, a generic map is Thom-Boardman. 
\end{proposition}

\begin{remark}
Thom-Boardman singularities do not give a stratification of the jet space $J_k(M,N)$: they provide a partition of $M$ of a generic mapping into locally closed submanifolds $\Sigma^I(f)$, but the closure of a submanifold is not necessarily a union of similar submanifolds. Indeed, it can be shown (see Lander \cite{lander}) that $\Sigma^2 \cap \overline{\Sigma^{1,\ldots, 1}} \neq \emptyset$ for any number of $1$'s. But if this number of $1$'s is sufficiently large, then $\dim(\Sigma^{1,\ldots, 1})<\dim(\Sigma^2)$.
\end{remark}

\begin{remark} Existence of good approximation of smooth maps does not imply that Thom-Boardman maps form a dense open subset of the space $J_k(M,N)$ of $k$-jets of maps from $M$ to $N$. Thom's Transversality Theorem (see e.g. \cite{arnold})  implies that this set is residual. But Wilson \cite{wilson}] showed that the Thom-Boardman maps form an open subset of $J(M,N)$, if and only if either $m\le n$ and $3m-4<2n$ or $m>n$ and $2n<m+4$. These numerical conditions hold precisely if $\codim(\Sigma^2)>m$, that is, there are no corank-$2$ singularities of $f$ (such maps are also called corank $1$ maps or Morin maps). Moreover, If these numerical conditions hold, then a map is Thom-Boardman if and only if its germ is stable at each point. 

\end{remark}

We now  prove some simple structural statements about the Thom-Boardman classification. Recall that the Thom-Boardman classes $\Sigma^I$ do not give a stratification of the jet space $J_k(M,N)$; Porteous pointed out the existence of generic smooth maps $f:\RR^5 \to \RR^5$ with singularities of type $\Sigma^2$ and $\Sigma^{1,1,1,1}$, both of dimension $1$, such that $\Sigma^2(f)$ contains some isolated points in the closure of $\Sigma^{1,1,1,1}(f)$. In general, $\Sigma^i \nsubset \overline{\Sigma^{1,\ldots, 1}}$ holds, but if the number of $1$'s is not more then $i$, then one contains the other according to the next Proposition.

\begin{proposition}\label{prop:boardman}
For a Boardman symbol $I=(i_1,\ldots, i_k)$ and a Thom-Boardman map $f: M \to N$ we have 
\[\Sigma^{i_1+\ldots +i_k}(f) \subset \overline{\Sigma^{i_1,\ldots i_k}(f)}\]
\end{proposition}
\begin{proof}
First we prove the statement for $(i_1,\ldots, i_k)=(1,\ldots, 1)$. The Thom-Boardman class $\Sigma^{1,1,\ldots 1}$ is equal to the Morin contact singularity class:
\[\Sigma^{1,1,\ldots 1}=\{f\in J_k(m,n): A_f \simeq \CC[t]/t^k\} \subset J_k(m,n).\]  
The test curve model introduced in \cite{bsz,b3} says that 
\[f \in \Sigma^{1,1,\ldots 1} \Leftrightarrow \exists \g \in J_k^\reg(1,m) \text{ such that } f \circ \g =0\]
that is, $\Sigma^{1,\ldots, 1}$ is the set of those $k$-jets of germs on $\CC^m$ at the origin which vanish on some regular curve $\g$; we call such a $\g$ a test
curve. Test curves of germs are generally not unique: if $\g$ is a test
curve for $f \in \Sigma^{1,1,\ldots 1}$, and $\vp \in J_k^\reg(1,1)$ is a 
polynomial reparametrisation of $\CC$ at the origin, then $\g \circ \vp$ is an other test curve of $\Psi$:
\begin{displaymath}
\label{basicidea}
\xymatrix{
  \CC \ar[r]^\vp & \CC \ar[r]^\g & \CC^m \ar[r]^{f} & \CC^n}
\end{displaymath}
\[f \circ \g=0\ \ \Rightarrow \ \ \ f \circ (\g \circ \vp)=0\]
In fact, we get all test curves of $f$ in this way if the
following property, open and dense in $\Sigma^{1,1,\ldots 1}$, holds: the linear part of $f$ has 
$1$-dimensional kernel. For a given $\g \in \jetreg 1m$ let $\mathcal{S}^n_{\g}$ denote the set of 
solutions of the equation $f \circ \g=0$, that is,
\begin{equation}\label{solutionspace}
\mathcal{S}^n_\g=\left\{f \in \jetk mn, f \circ \g=0 \text{ up to order } i \right\}.
\end{equation}
The equation $f \circ \g=0$ reads as follows for $k=3$:
\begin{eqnarray} \label{eqn3}
& f'(\g')=0 \\ \nonumber & f'(\g'')+f''(\g',\g')=0 \\
\nonumber
& f'(\g''')+2f''(\g',\g'')+f'''(\g',\g',\g')=0 \\
\end{eqnarray}
These are linear in $f$, hence $\mathcal{S}^n_\g \subset \jetk mn$
is a linear subspace of codimension $kn$, i.e a point of $\grass_{\mathrm{codim}=kn}(J_k(m,n))$, whose orthogonal, $(\mathcal{S}^n_{\g})^\perp$, is an $kn$-dimensional subspace of $J_k(m,n)^*$. These subspaces are invariant under the reparametrization of $\g$. In fact, $f \circ \gamma$ has $n$ vanishing coordinates and therefore
\[(\mathcal{S}^n_{\g})^\perp=(\mathcal{S}^1_{\g})^\perp \otimes \CC^n\]
holds. The map
\[\phi: \jetreg 1m \rightarrow \grass_k(J_k(m,1)^*)\]
defined as  $\gamma  \mapsto (\mathcal{S}^1_\g)^\perp$
is $J_k^\reg(1,1)$-invariant and induces an injective map on the $J_k^\reg(1,1)$-orbits into the Grassmannian 
\[\phi^\grass: \jetreg 1m /J_k^\reg(1,1) \hookrightarrow \grass_k(J_k(m,1)^*).\]
This is given by 
\[\phi(\g',\ldots, \g^{(k)})=\Span(\g',\g''+(\g')^2,\ldots, \g^{(k)}+\sum_{i_1+\ldots +i_k=k}(\g')^{i_1}\cdots (\g^{(k)})^{i_k})\]
and the closure of the image is $\overline{\Sigma^{1,\ldots,1}} \subset \grass_k(J_k(m,1)^*)$.
Hence a generic point in $\Sigma^{1,\ldots, 1}$ has the form $\phi(\g',\ldots, \g^{(k)})$ with $\det(\g',\ldots, \g^{(k)})\neq 0$. A sequence satisfying $|\g^{(i)}|<\e_i |\g^{(i+1)}|$ and $\e_i \to 0$ for $1\le i \le k-1$ converges to $\Span(\g',\ldots, \g^{(k)})$, which is a point in $\Sigma^k$. Conversely, any germ $f \in \Sigma^k \subset J_k(m,n)$ satisfies $\dim(\ker(f))=k$, and is reprtesented by a point $\Span(\g',\ldots, \g^{(k)}) \in \grass_k(J_k(m,1)^*)$, and hence can be approximated by points of the form $\phi(\g',\ldots, \g^{(k)})$.

Now let $f \in \Sigma^{i_1,\ldots, i_k}$. We apply the just proved special case for the restriction $f|_{\Sigma^{i_1,\ldots, i_{k-1}}(f)}$ to get
\[\Sigma^{i_1,\ldots, i_k}(f) \subset \Sigma^{i_1,\ldots, i_{k-1},\overbrace{1\ldots, 1}^{i_k}}(f),\]
then do the same inductively.
\end{proof}

\begin{remark}
We will use the test curve model introduced in this proof for the description of the curvilinear Hilbert scheme, see Theorem \ref{bszmodel}.
\end{remark}


We finish this overview on Thom-Boardman classes with a general structure theorem on singularities of maps. Let $f: M \to N$ be a Thom-Boardman map and let $A$ be a local algebra of dimension $k$,  and assume that the local algebra of $f$ at some $p\in M$ is $A_f\simeq A$. By Prop. \ref{prop:boardman} $p\in \overline{\Sigma^{1,1,\ldots, 1}(f)}$ and hence the point $\mathrm{Spec}(A)\in \Hilb^k(\CC^m)$ sits in the closure of the points with Morin algebra, that is, the curvilinear locus. We obtain 

\begin{corollary}
Assume that $f: M \to N$ is a Thom-Boardman map and $A$ is a local algebra of dimension $k$ such that $\Tp_A^{M \to N} \neq 0$, which implies that there are points on $M$ where the local algebra of $f$ is isomorphic to $A$. Then $\mathrm{Spec}(A)$ must sit in the the curvilinear component $\CHilb^k(\CC^m)$. 
\end{corollary}

\section{Thom polynomials and multisingularity classes}\label{sec:thommulti}
 A natural generalisation of the multipoint problem addresses cohomological counts of multisingularity loci for generic maps. In this paper we only deal with the multipoint situation, but we give here a short survey on multisingularities to provide a wider context of the problem. In a forthcoming paper we address the multisingularity question. Thom polynomials are multisingularity classes for monosingularity types, and the conjecture of Rim\'anyi formulates the nontrivial fact that Thom polynomials are the building blocks of multisingularity classes. 

\subsection{Thom polynomials}

Let $f: M \to N$ be a complex analytic map between two complex manifolds, $M$ and $N$, of dimensions $m\le n$. We say that $p\in M$ is a {\em singular} point of $f$ if the rank of the differential $df_p:T_pM\to T_{f(p)}N$ is less than $m$. Recall that a singularity class $O \subset \mapd mn$ is a subset invariant under left-right equivalence. 
For a singularity $O$ and generic 
holomorphic map $f:M\to N$, we can define the set
\[\Sigma_O(f) = \{p\in M;\; f_p\in O\}, \] which is independent of 
coordinate choices on $M,N$. Assuming $M$
is compact and $f$ is sufficiently generic, $\Sigma_O(f)$ is an analytic subvariety of $M$ giving a Poincar\'e dual class $[\Sigma_O(f)]\in H^*(M,\mathbb{Z})$. This problem was first formulated by Ren\'e Thom (cf. \cite{thom, haef}) in the category of smooth varieties and smooth maps; in this case cohomology with $\ZZ/2\ZZ$-coefficients is used. Thom discovered
that to every singularity $O$ one can associate a bivariant characteristic class $\tau_O$, which, when evaluated on the pair $(TM,f^*TN)$ produces the Poincar\'e dual class $[\Sigma_O(f)]$.  One of the consequences of this result is that the class $[\Sigma_O(f)]$ depends only on the homotopy class of $f$.

The structure of the $\gnk$-action on $\mapd mn$ is rather complicated; even the parametrization of the orbits is difficult. However, as we saw in the oprevious section, there is a simple invariant on the space of orbits: to each map-germ $f=(f_1,\ldots, f_n):(\C^m,0)\to(\C^n,0)$, we can associate its local algebra $A_f=(x_1,\ldots, x_m)/(f_1,\ldots, f_n)$ defined as the quotient of the algebra of power series with no constant term by the ideal generated by the pull-back subalgebra $f^*(y_1\ldots,y_n)$. This algebra $A_{f}$ is trivial if the map-germ $f$ is nonsingular, and it $\mathcal{A,K}$-invariant. Hence, for a fixed nilpotent algebra $A$ with $n$ generators one can associate the invariant subset 
$\Sigma_A=\{f \in J(m,n):A_f \simeq A\}$,
the set of jets with local algebra isomorphic to $A$. The {\em Thom principle} in this holomorphic setting (cf \cite{kazariantalk,bsz}) for $\calo=\Sigma_A$  tells that there is a jet order $k$ depending only on $A$ such that  
\[ [\Sigma_A(f)]=\epd{\Sigma_A,J_k(m,n)}(f^*TN,TM),\]
is obtained by substituting the Chern roots of $TM,f^*TN$ into the $\GL(m) \times \GL(n)$-equivariant dual of $\Sigma_A$ in $J_k(m,n)$. This equivariant dual sits in the ring $\CC[x_1,\ldots, x_m,y_1,\ldots, y_n]^{S_m\times S_n}$ of bi-symmetric polynomials on $m+n$ variables. Damon and Ronga \cite{damon,sigma11} proved that $[\Sigma_A(f)]$ depends on $TM,f^*TN$ only through the Chern classes of the difference bundle $TM-f^*TN$, and hence $[\Sigma_A(f)]=\Tp_A(c_1,c_2,\ldots)$ is a polynomial in these classes, which we call the {\em Thom polynomial of $A$}. Calculation of these polynomials remains a major open problem ever since then, see e.g \cite{arnold, rimanyi, kazariantalk, rimanyifeher}.

\textbf{Toy example: the singularity locus.} The very first example should be the cohomological locus of singulariites, i.e the locus of points where $f$ is not of maximal rank.
\[\mathrm{Sing}(f)=\{p\in M: \mathrm{rank}(d_pf)<m\}\]
This locus corresponds to the algebra $A_1=\CC[t]/t^2$ and $f$ has the same singularity locus as its first jet, so we can take $k=1$. Then  
\begin{multline*}
\Sigma_{A_1}=\{A \in J_1(m,n):A_f \simeq \CC[t]/t^2\}=\{A\in \Hom(\CC^m,\CC^n) \dim \ker A=1\}=\\
=\{A\in \Hom(\CC^m,\CC^n): \exists! v\in\CC^n,\,v\neq 0:\,Av=0\}
\end{multline*}

Then $\Hom(\CC^m,\CC^n)=(\CC^n)^* \otimes \CC^m$ is acted on by $\GL(n)\times \GL(m)$. Let $\l_1,\ldots, \l_m$ and $\theta_1,\ldots, \theta_n$ denote the torus weights for the maximal tori $T^m \subset \GL(m)$ and $T^n \subset \GL(n)$, respectively. 
Then the $T^m \times T^n$ torus weight under the left-right action on the $(i,j)$ entry is $\theta_j-\lambda_i$.  
We have a natural $T^m$-equivariant fibration $\pi:\overline{\Sigma_{A_1}}
\rightarrow \PP^{m-1}$ and the fiber over a point $[v]\in \PP^{m-1}$ is the
linear subspace $W_v=\{A;\;Av=0\}\subset\Sigma_{A_1}$.

The torus fixed-points on $\PP^{m-1}$ correspond to the coordinate axes, let $p_1\ldots p_m$ denote these. The tangent weights at $p_i$ are 
$\{\lambda_s-\lambda_i;\;s\neq i\}$. The fiber at $p_i$ is the set of matrices $A$ with all entries in the
$i$th column vanishing. This is a linear subspace in $\Hom(m,n)$, whose equivariant dual is the product of the normal weights:
\[\epd{W_i,\Hom(m,n)}=\prod_{j=1}^n(\theta_j-\lambda_i)\]
Hence the $T^m \times T^n$-equivariant dual of $\Sigma_{A_1}$ is given by the Atiyah-Bott formula:
\[\Tp(\Sigma_{A_1})=\sum_{i=1}^m\frac{\prod_{j=1}^n(\theta_j-\lambda_i)}{\prod_{s\neq
i}(\lambda_s-\lambda_i)}=\res_{q=0}\frac{\prod_{j=1}^{n}(1+q\theta_j)}
{\prod_{i=1}^{m}(1+q\lambda_i)}\,\frac{dq}{q^{n-m+2}}=c_{n-m+1}(TM-f^*TN)
\]
Here we used the simple but crucial observation that the residues of the rational differential form
\[
-\frac{\prod_{j=1}^{m}(\theta_j-z)}
{\prod_{i=1}^{n}(\lambda_i-z)}\,dz.\] 
at finite poles: $\{z=\lambda_i;\;i=1\ldots m\}$ exactly recover the
terms of the Atiyah-Bott sum. Apply the residue theorem, and change variables $z=-1/q$.

\subsection{Multisingularity classes}

To every multisingularity class $\a=(\a_1,\ldots, \a_s)$ one can assign the corresponding multisingularity cycle $M(\a)$ of a given map $f:M \to N$ of relative dimension $l=\dim(N)-\dim(M)$. This cycle is defined as a subvariety in the Cartesian product $M^s=M\times \ldots \times M$, given as the closure of the set of tuples $(x_1,\ldots, x_s) \in  M^s$ of pairwise distinct points, $x_i \neq x_j$ for $i\neq j$, such that $f(x_1)=\ldots =f(x_s)$ and $f$ has a singularity of the given type $\a_i$ at $x_i$ for $i=1,\ldots, s$. Under some genericity condition $M(\a)$ has the expected dimension equal to $\dim(M)-\codim \a_i-(r-l)l$, where $\codim \a_i$ is the expected codimension of the cycle of the local singularity $\a_i$ in the source manifold $M$ of the map and is equal to the true codimension of the variety of this singularity in the space of jets of the maps.
   
Let us consider the natural projections $p_M :M(\a)\to M$ and $p_N: M(\a)\to N$ defined by the projection on the first factor of the Cartesian product $M^s$ and by the map $f$, respectively. If $f$ is proper, then the images of these projections are closed singular subvarieties in $M$ and $N$, respectively. The theory of characteristic classes of multisingularities describes the cohomology classes $\overline{m}_\a \in H^*(M)$ and $
\overline{n}_\a \in H^*(N)$, dual to the images $p_M(M(\a)) \subset M$ and $p_N(M (\a))$, respectively (regarded as singular reduced subvarieties).

We often consider the classes $m_\a$ and $n_\a$ with their natural multiplicities, namely, we set
\[m_\a = \#\mathrm{Aut}(\a_2,\ldots ,\a_s)\bar{m}_\a, n_\a = \# \mathrm{Aut}(\a_1, \ldots, \a_s)\bar{n}_\a\]
where $\#\mathrm{Aut}(\a_1,\ldots, \a_s) = \# \mathrm{Aut}(\a)$ is the number of permutations $\s \in S_s$ such that $\a\s(i) = \a(i)$ for $i = 1,\ldots ,s$. If $\a$ contains $k_j$ singularities of the same type $\b_j, j = 1,2,\ldots $ with pairwise distinct $\b_j$ , then $\# \mathrm{Aut}(\a)=k_1! \cdot k_2! \cdots$. These multiplicities are also equal to the degrees of the map of the multisingularity cycle $M(\a)$ onto its image under the action of the projections $p_M$ and $p_N$, respectively.
We note that the map $f$ induces a map $p_M(M(\a))\to p_N(M(\a))$. The degree of this map is equal to the number of singularities $\a_1$ in the tuple $\a = (\a)_1, \ldots , \a_r)$, hence $n_\a = f_* m_\a \in H^*(N)$,
where $f_*$ is the Gysin homomorphism (which is well defined if the map f is proper). 

Kazarian \cite{kazarian} formulates a conjecture which expresses the classes $m_\a$ and $n_\a$ in terms of some characteristic classes related to the map $f$. These classes include, in particular, the Chern classes $c_i(TM) \in H^*(M)$ and $c_i(TN) \in H^*(N)$ themselves of the manifolds. One can obtain more classes from these by applying the pull-back and push-forward homomorphisms $f^* :H^*(N) \to H^*(M)$ and $f_* : H^*(M) \to H^*(N)$. In fact, the final expressions include only certain combinations of these classes, defied as follows.
\begin{itemize} 
\item The relative Chern classes $c_i(f) = c_i(f^*TN -TM)$ of the map $f$ are defined on the source manifold $M$. 
\item The Landweber-Novikov classes on the target manifold $N$:
\[s_I(f)=f_*(c_1(f)^{i_1}c_2(f)^{i_2} \ldots) \in H^{2\deg s_I}(N) \text{ with } \deg s_I =l+\deg c_I =l+i_1+2i_2+\ldots \]
\end{itemize}


\begin{definition} For a multisingularity $\a = (\a_1,\ldots ,\a_s)$ the residual polynomial $R_\a$ is the polynomial in the classes $c_i$ equal to the constant term of the polynomial $m_\a$ in the variables $s_I$. In other words, $R_\a$ is obtained from $m_\a$ if one equates all the generators $s_I$ to zero in the latter polynomial.
\end{definition}
The degree of $R_\a$ is equal to the codimension of the corresponding cycle of multisingularities in $M$ , that is, 
\[\deg R_\a = \sum \codim \a_i+(s-1)l\]
We denote by $S_\a = s(R_\a)$ the linear combination of Landweber-Novikov classes that corresponds to the polynomial $R_\a$. If $f:M\to N$ is a holomorphic proper map, then, as explained above, the polynomials $R_\a$ and $S_\a$ determine the corresponding characteristic cohomology classes on $M$ and $N$, respectively.
\begin{conjecture}[Kazarian] The polynomials $m_\a$ and $n_\a$ are uniquely determined by the residual polynomials $R_{\a_J} = R_{\a_{j_1},\ldots ,\a_{j_t}}$ for all possible subsets $J = \{j_1, \ldots , j_t\} \subset \{1, \ldots , s\}$. More explicitly, the following relations hold:
\[m_\a = \sum_{J_1 \sqcup \ldots \sqcup J_m=\{1,\ldots ,s \}} R_{\a_{J_1}} S_{\a_{J_2}} \ldots S_{\a_{J_m}}.\]
\[n_\a=s(m_\a)=\sum_{J_1 \sqcup \ldots \sqcup J_m=\{1,\ldots ,s \}} S_{\a_{J_1}} \ldots S_{\a_{J_m}}\]
The subset containing the element $1 \in \{1,\ldots, s\}$ is denoted by $J_1$. 
\end{conjecture}
In particular, this defines $R_\a$ recursively as the unique polynomial satisfying the relation
\[R_{\a}=\Sigma_{J_1 \sqcup \ldots \sqcup J_s=\{1,\ldots ,k\}} (-1)^{s-1}(s-1)!m_{\a_{J_1}} n_{\a_{J_2}} \ldots n_{\a_{J_s}}.\]

\section{Hilbert schemes on smooth varieties}\label{sec:hilbertschemes}

\subsection{Components and punctual components}
Let $X$ be a smooth projective variety of dimension $m$ and let  
\[\Hilb^k(X)=\{\xi \subset X:\dim(\xi)=0,\mathrm{length}(\xi)=\dim H^0(\xi,\calo_\xi)=k\}\]
denote the Hilbert scheme of $k$ points on $X$ parametrizing all length-$k$ subschemes of $X$. For $p\in X$ let 
\[\Hilb^{k}_p(X)=\{ \xi \in \Hilb^k(X): \mathrm{supp}(\xi)=p\}\]
denote the \textit{punctual Hilbert scheme} consisting of subschemes supported at $p$. Let 
\[\rho: \Hilb^k(X) \to \Sym^kX,\  \xi \mapsto \Sigma_{p\in X}\mathrm{length}(\calo_{\xi,p})p\]
denote the Hilbert-Chow morphism. Then $\Hilb^k(X)_p=\rho^{-1}(kp)$.

The Hilbert scheme of $k$ points on smooth surfaces is a $2k$ dimensional nonsingular variety \cite{fogarty}. It occupies a central position in many important branches of mathematics and mathematical physics. In particular, it plays a pivotal role in enumerative geometry,  see e.g \cite{nakajima}. The cohomological intersection theory of the Hilbert scheme of points on surfaces can be approached from several different directions: 1) via the inductive recursions set up in \cite{egl, esa,esa2}; 2) using Nakajima calculus \cite{nakajima,groj,lehn} or more recently 3) virtual localisation on Quot schemes \cite{virtual, mop2}. \textit{Lehn's conjecture} \cite{lehn} on top Segre numbers of tautological line bundles incapsulates the complexity of this theory.

While Hilbert schemes on surfaces are well-understood and broadly studied, the Hilbert scheme of points on manifolds of dimension three or higher are extremely wild objects with several unknown irreducible components and bad singularities \cite{vakil}, and our undertanding of them is very limited. 

In enumerative geometry applications we are mostly interested in the \textit{geometric component} (also called the main component, or smoothable component) of the Hilbert scheme $\Hilb^k(X)$. This component, which we denote by $\GHilb^k(X)$ is the closure of the open locus formed by the reduced points in $\Hilb^k(X)$, that is, those supported at $k$ distinct points on $X$. This is an irreducible, but highly singular component of dimension $km$. Ekedahl and Skjelnes \cite{ekedahl} constructs the geometric component as a certain blow-up along an ideal of $\Sym^k X$--this is the generalisation of the classical result of Haiman \cite{haiman} on surfaces. 

In this paper we develop iterated residue formulas for certain \textit{tautological integrals} on the geometric component. The idea is to reduce integration to the \textit{punctual curvilinear component}, which, in turn, is a projective compactification of a non-reductive quotient: the moduli of $k$-jets in $X$ up to polynomial reparametrisatons. 
A subscheme $\xi \in \Hilb^k_p(X)$ is called curvilinear if $\xi$ is contained in some smooth jet of a curve $\calc_p \subset X$ at $p$.  Equivalently, $\xi$ is curvilinear if $\calo_\xi$ is isomorphic  to the $\CC$-algebra $\CC[z]/z^{k}$.
The \textit{punctual curvilinear locus} at $p\in X$ is the set of curvilinear subschemes supported at $p$: 
\begin{multline}\nonumber 
\mathrm{Curv}^k_p(X)=\{\xi \in \Hilb^k_p(X): \xi \subset \mathcal{C}_p \text{ for some smooth curve } \mathcal{C} \subset X\}=\\
=\{\xi \in \Hilb^k_p(X):\calo_\xi \simeq \CC[z]/z^{k}\}
\end{multline}

We call the closure $\CHilb^k_p(X)=\overline{\mathrm{Curv}^{k}_p(X)}$ the \textit{punctual curvilinear component} supported at $p$. $\CHilb_p^k(X)$ is an irreducible component of the punctual Hilbert scheme $\Hilb^k_p(X)$ of dimension $(k-1)(m-1)$. On a smooth surface $S$, $\CHilb^k_p(S)=\Hilb^k_p(S)$ according to Brianchon. Moreover, 
\[\GHilb^k_p(S)=\Hilb_p^{[k]}(S) \cap \GHilb^k(S)\]
is irreducible and equal to the full punctual Hilbert scheme.  However, for $\dim(X)\ge 3$ $\GHilb^k_p(S)$ is typically not irreducible; its components are called the \textit{smoothable components} of $\Hilb_p^{k}(X)$. The punctual curvilinear component $\CHilb^k_p(X)$ is a smoothable component. On the other hand, the  Iarrobino-type punctual components \cite{iarrobino} are not necessarily smoothable: their dimension can be higher than $km$. Description of the smoothable components is a hard problem and is unknown in general.  It was an open question until recently whether there exist divisorial components of $\GHilb^k_p(S)$, that is, components $C_p \subset R_p^{[k]}$ whose sweep $\cup_{p\in X} C_p$ over $X$ is a divisor of the good component. The dimension of this $C_p$ is necessarily $m(k-1)-1$.  Erman and Velasco in \cite{erman} give affirmative answer to this question when $d\ge 4, k\ge 11$. A beautiful illustration of the components and structure of $\Hilb^k(X)$ can be found in the PhD thesis of Jelisiejew \cite{jel}, called the \textit{'Bellis Hilbertis'}.

To sum up, when $\dim(X)\ge 3$, and the number of points $k$ is large enough, the punctual Hilbert scheme $\Hilb^k_p(X)$ is not necessarily irreducible or reduced, and it has smoothable and non-smoothable components. Among the smoothable components, there is a distinguished one, the curvilinear component, we will call all other smoothable components  \textit{exotic components}.

\subsection{Tautological bundles and integrals} Let $X$ be a smooth projective variety of dimension $m$, and let $V$ be a rank $r$ algebraic vector bundle on $X$. Let $\Hilb^k(X)$ denote  the Hilbert scheme of length $k$ subschemes of $X$ and let $V^{[k]}$ be the corresponding tautological rank $rk$ bundle on $\Hilb^k(X)$ whose fibre at $\xi \in \Hilb^k(X)$ is $H^0(\xi,V|_\xi)$. 

 Let $\Phi(c_1,\ldots, c_{rk})$ be a Chern polynomial in $c_i=c_i(V^{[k]})$ of weighted degree equal to $\dim \GHilb^k(X)=km$. The Chern numbers 
\[\int_{\GHilb^k(X)} \Phi(V^{[k]})\]
are called tautological integrals of $V^{[k]}$ on the geometric component. Rennemo \cite{rennemo} shows that for any $\Phi$ there is a universal polynomial $P_{\Phi}$ (independent of $X$) such that the integral $\int_{\GHilb^k(X)} \Phi$ is obtained by substituting the Chern roots of $X$ and $V$ into $P_{\Phi}$. Over the curvilinear component $\CHilb^k(X)$ closed formula for $P_\Phi$ is developed in \cite{b0}.

\section{Residue formula for tautological integrals on Hilbert schemes}\label{sec:tauintegral}

A central result of this paper is an iterated residue formula for certain tautological integrals on the main component $\GHilb^k(M)$. We state and explain the formula in this section, and prove it in   \S \ref{sec:prooftauintegrals}. This argument is then extended to push-forwards along holomorphic maps in \S \ref{push-forwards}, which leads us to the proof of our main theorems. 
Before presenting the proof of Theorem \ref{mainthm1} and Theorem \ref{mainthm2} in full generality in \S \ref{sec:proofmain}, we explain the formula first when the number of points is $2$ and $3$. For these low number of points (and more generally, when $k\le 13$ and hence the punctual Hilbert scheme $\GHilb^k(\CC^m)=\CHilb^k(\CC^m)$ is irreducible) we do not have to deal with all technical difficulties, but they show they already exhibit most of the key steps of our strategy. 

\subsection{Hilbert scheme of $2$ points}

For $k=2$ the geometric component $\GHilb^2(M)=\Hilb^2(M)$ is equal to the full Hilbert scheme, and the punctual part is irreducible, consisting only of the curvilinear component $\CHilb^2_0(\CC^n) \simeq \PP^{n-1}$. Hence Condition \ref{geocond} automatically holds, and Theorem \ref{tauintegral} reads as follows. 
\begin{proposition}[\textbf{Integration over $\GHilb^2(M)$}]
	Let $V$ be a rank-$r$ bundle over $M$, and $\Phi$ be a symmetric 
	polynomial of degree $2m$ in $2r$ variables. Then we have
	\[\int_{\GHilb^2(M)}\Phi\left(V^{[2]}\right)=\int_M \res_{z=\infty} \Phi\left(V\oplus V(z)\right)\,
	s_M\left(\frac{1}{z}\right)\,\frac{dz}{z^{m+1}}+\int_{M \times M} \Phi\left(V_1 \oplus V_2\right)\]
where 
\begin{itemize}
	\item $V_i$ is the bundle $V$ pulled back from the $i$th copy of $M$ for $i=1,2$,
	\item $V(z)$ stands for the bundle $V$ tensored by the line $\mathbb{C}_z$, which is the representation of a torus $T$ with weight $z$. Then 
	\[ \Phi(V\oplus V(z)) =  \Phi(\theta_1,z+\theta_1,\theta_2,z+\theta_2,\dots,\theta_r,z+\theta_r)
	\]
	\item $s_M\left(\frac{1}{z}\right)$ is the Segre series of $M$.
\end{itemize}
\end{proposition}

\begin{proof}
Let $\hch:\GHilb^2(M)\to M\times M$ be the Hilbert-Chow morphism, and introduce the notation $V^{[1,1]} = \hch^*(V_1 \oplus V_2)$.	We begin by writing 
	\[\int_{\GHilb^2(M)}\Phi(V^{[2]})=
	\int_{\GHilb^2(M)}\left[ \Phi(V^{[2]})-\Phi(V^{[1,1]})\right] +
	\int_{\GHilb^2(M)}\Phi(V^{[1,1]}),
	\]
	and observing that we can rewrite the second term as follows:
	\[ \int_{\GHilb^2(M)}\Phi(V^{[1,1]}) = \int_{M \times M} \Phi\left(V_1 \oplus V_2\right). \]
	As to the first term, we note that the integrand may be represented as a form supported in a neighborhood of $\CHilb^2_0$ which is, in turn, isomorphic to the normal bundle $\nabla_0\to\CHilb^2_0$ (a line bundle in this case) of the punctual Hilbert scheme in the entire Hilbert scheme. According to the Thom isomorphism theorem, we have $\hcpt(\nabla_0)\simeq H^*(\CHilb^2_0)\cdot u$, where $u$ is the compactly supported Thom class. We thus have
	\begin{equation*}
		\int_{\GHilb^2(M)}\left[ \Phi(V^{[2]})-\Phi(V^{[1,1]})\right] =
		\int_{\nabla_0}\left[ \Phi(V^{[2]})-\Phi(V^{[1,1]}\right] =
		\int_{\CHilb^2_0} \frac{\Phi(V^{[2]})-\Phi(V^{[1,1]})}{u}
	\end{equation*} 
	
	Now we use the family version of Kalkman's formula: we represent $\CHilb^2_0$ as a $T=\C^*$ quotient of the normal bundle $\nabla_{ M}$ of $\Delta M\subset M\times M$. We lift the bundles $V^{[2]}$ and $V^{[1,1]}$ to equivariant bundles $V_z^{[2]}$ and $V_z^{[1,1]}$ on $\nabla_{ M}$, the Euler class $u$ to $u_z$, and then the formula reads
	\[ \int_{\CHilb^2_0} \frac{\Phi(V^{[2]})-\Phi(V^{[1,1]})}{u}
	=  \int_{\Delta M} \res_{z=\infty} \frac{\Phi(V_z^{[2]})-\Phi(V_z^{[1,1]})}{u_z\cdot c_m(\nabla^z_{\Delta M})},
	\]
	where all the objects under the integral sign are restricted to $\Delta M$, the zero-section of $\nabla_{ M}$; identifying $\Delta M$ with $M$, these restrictions are as follows:
\begin{itemize}
	\item $V_z^{[2]}=V\oplus V(z)$
	\item  $V_z^{[1,1]}=V^{[1,1]}$
	\item $u_z=z$
	\item $c_m(\nabla^z_{ M})=\prod_{i=1}^{m}(z+\lambda_i)$, where the $\lambda$s are the Chern roots of $\nabla_{ M}$.
\end{itemize}
	Thus we have
	\[  \int_{\CHilb^2_0} \frac{\Phi(V^{[2]})-\Phi(V^{[1,1]})}{u}
	= \int_{\Delta M} \res_{z=\infty} \frac{\Phi(V\oplus V(z))-\Phi(V^{[1,1]})}{z\cdot \prod_{i=1}^{m}(z+\lambda_i)} \]
	Finally, we observe that the second term of the numerator is independent of $z$, and thus does not contribute to the residue. We  thus arrive at the final formula
\[	\int_{\GHilb^2(M)}\left[ \Phi(V^{[2]})-\Phi(V^{[1,1]})\right] =	\int_M \res_{z=\infty} \Phi\left(V\oplus V(z)\right)\,
	s_M\left(\frac{1}{z}\right)\,\frac{dz}{z^{m+1}}, \]
and this completes the proof.
\end{proof}	
	
Now we extend this formula to the case when the push-forward to the 
point is replaced by a pushforward along a smooth map between varieties, 
$f:M\to N$. 
Notation:
\begin{itemize}
	\item We will use the standard notation $f_*[\Psi]$ for the push-forward 
	of a cohomology class $\Psi\in H^*(M)$. Occasionally, we will indicate the 
	source space in the index for clarity, as in  $f_*[\Psi]_M$. 
	\item We will write $ f^{\times2}: M \times M\to N \times N$, and $ f^{[2]} =  f^{\times2}\circ\hch$
	\item The embeddings of the diagonals will be denoted by $\iota_{\Delta M}:\Delta M\hookrightarrow M\times M$ and $i_\Delta: \Delta N\hookrightarrow N \times N$.
\end{itemize}

Then, with the same proof as in Step 1, we arrive at the following formula:
\begin{equation}\nonumber
  f^{[2]}_*[\Phi(V^{[2]})]_{\GHilb^2(M)} =
f^{\times2}_* [\Phi(V_1 \oplus V_2)]_{M \times M} + \iota_{\Delta N*} f_* \left[\res_{z=\infty} \Phi\left(V\oplus V(z)\right)\,
s_M\left(\frac{1}{z}\right)\,\frac{dz}{z^{m+1}}\right] 
\end{equation}


A key observation is that this generalized setup has an added flexibility: the coefficient ring of the push-forward become the cohomology of $N$. One way this could appear is the following situation: Suppose that $V^{[2]}$ has a subbundle of the form $f^{[2]*}W\hookrightarrow V^{[2]}$ for some bundle $W$ over $N\times N$ and this is consistent with an embedding $f^{\times2*}W\hookrightarrow V_1\oplus V_2$. For this case we can write
\begin{equation}\label{relative2pt}
f^{[2]}_*\left[\Phi\left(\frac{V^{[2]}}{f^{[2]*}W}\right)\right]_{\GHilb^2(M)} =
f^{\times2}_* \left[\Phi\left(\frac{V_1 \oplus V_2}{f^{\times2*}W}\right)\right]_{M \times M} + \iota_{\Delta N*} f_* \left[\res_{z=\infty} \Phi\left(\frac{V\oplus V(z)}{f^*\iota_{\Delta N}^*W}\right)\,
s_M\left(\frac{1}{z}\right)\,\frac{dz}{z^{m+1}}\right] 
\end{equation}

Step 3. Now we apply this to our setup. We define the 2-point locus in $N$ as 
$f_*[\Hilb^2(f)]$ and our goal is to evaluate this cohomology class in $N$.
As $f(\Hilb^2(f))$ is a subset of the diagonal $\Delta N\subset N\times N$, our calculations will take place in a neighborhood of $\Delta N$, which we will identify with the normal bundle $\nabla_{ N}$ of $\Delta N$ in $N\times N$, and in compactly supported cohomology in the open set $U=\left(f^{[2]}\right)^{-1}\nabla_{ N}$.

The fundamental geometric statement is as follows: 
\begin{proposition}\label{prop:sectionk=2}
	Let  $T_v\to\nabla_N$ be the tautological pull-back bundle on $\nabla_{N}$  and $ T_h=\pi^*T\Delta M\to\nabla_N $ be the pull-back from the base. Then there is a bundle embedding  $f^{[2]*}T_h\hookrightarrow f^*TN^{[2]}$  consistent with an embedding $f^{\times2*}T_h\hookrightarrow f^*TN_1\oplus f^*TN_2$. In addition, the quotient bundle
	$f^*TN^{[2]}/f^{[2]*}T_h$ has a transversal section presenting  $\Hilb^2(f)$ as a complete intersection in $U$.
\end{proposition}

 Recall first the Thom isomorphism for the space $\nabla_N$. Since $c_n(T_v)$ is exactly the Thom class of the bundle $\nabla_N$, we have  
\[  \hcpt(\nabla_N)\simeq H^*(\Delta_N)\cdot c_n( T_v).
 \] 
We also, clearly have $T\nabla_N\simeq TN_1\oplus TN_2\simeq T_h\oplus T_v$. 
Now, we apply \eqref{relative2pt} as follows:
 we set $V=f^*TN$, $W=T_h$ and $\Phi=c_n$. Then
\begin{multline}\label{finalk=2}
	[f(\Hilb^2(f))]= 
	\pi_*f^{[2]}_*\left[c_n\left(\frac{f^*TN^{[2]}}{f^{[2]*}T_h}\right)\right]_{U\subset\GHilb^2(M)}=\\  
\frac1{c_n(T_v)}	f^{\times2}_* \left[c_n\left(\frac{f^*TN_1 \oplus f^*TN_2}{f^{\times2*}T_h}\right)\right]_{M \times M} +
 \frac1{c_n(T_v)}\iota_{\Delta N*} f_* \left[\res_{z=\infty} c_n\left(\frac{f^*TN\oplus f^*TN(z)}{f^*\iota_{\Delta N}^* T_h}\right)\,
	s_M\left(\frac{1}{z}\right)\,\frac{dz}{z^{m+1}}\right]=\\ 
	\iota_{\Delta N}^*f^{\times2}_*[1]+ f_*\left[\res_{z=\infty} c_n\left(f^*TN(z)\right)\,
	s_M\left(\frac{1}{z}\right)\,\frac{dz}{z^{m+1}}\right]=f_*[1]\cdot f_*[1] + f_*c_{n-m}(f).
\end{multline}

\subsection{The Hilbert scheme of three points}\label{subsec:k=3}

In this section we follow our strategy to prove the formula for $3$ points. Most of the key ingredients of the general argument already appears here. 

For $k=3$ the geometric component $\GHilb^3(M)=\Hilb^3(M)$ is, again, equal to the full Hilbert scheme, and the punctual part is irreducible, consisting only of the curvilinear component $\CHilb^3_0(\CC^n)$, there are no exotic components. Hence Definition \ref{geocond} automatically holds for the forms we deal with, and Theorem \ref{tauintegral} reads as follows. 

\begin{proposition}\label{prop:k=3}
	Let $V$ be a rank-$r$ bundle over the $m$-dimensional manifold $M$ with $r\ge m$, and $\Phi=c_{3m}(V^{[3]})$ be the top Chern class, which is a symmetric 
	polynomial of degree $3m$ in $3r$ Chern roots. Then we have
	\begin{eqnarray*} \int_{\GHilb^3(M)}\Phi\left(V^{[3]}\right)&=&\int_M \res_{z_1=\infty} \res_{z_2=\infty} \frac{\Phi\left(V\oplus V(z_1)\oplus V(z_2) \right)(z_1-z_2)dz_1dz_2}{(z_1z_2)^{m+1}(2z_1-z_2)}\,
		s_M\left(\frac{1}{z_1}\right) s_M\left(\frac{1}{z_2}\right)\,+ \\
		& +& \sum_{i=1}^3 \int_{M\times M} \res_{z=\infty} \frac{\Phi\left(V\oplus V(z) \oplus V_i \right)dz}{z^{m+1}}\,
		s_M\left(\frac{1}{z}\right)
		+\\
		&+& \int_{M \times M \times M} \Phi\left(V \oplus V \oplus V \right)
	\end{eqnarray*}
	
\end{proposition}
\begin{proof}
	We let $\GHilb^{[1,2,3]}(M)$ denote the ordered Hilbert scheme of $3$ points, which is a branched 6-fold cover of $\GHilb^3(M)$. Let $\GHilb^{[1,2],[3]}(M)=\overline{\im(\phi^{[1,2],[3]})}$ denote the closure of the image of the rational morphism 
	\[\phi^{[1,2],[3]}: \GHilb^{[1,2,3]}(M) \dasharrow \GHilb^{[1,2]}(M) \times M\]
	which is defined on the locus where point number $3$ does not collide with points number $1$ and $2$. We define $\GHilb^{[2,3],[1]}(M), \GHilb^{[1,3],[2]}(M)$ and $\GHilb^{[1],[2],[3]}(M)$ similarly, and we call these birational models \textit{ approximating Hilbert schemes} for short. Note that $\GHilb^{[1],[2],[3]}(M)=M\times M \times M$. We form a rational map 
\[\phi=(id, \phi^{[1,2],[3]}, \phi^{[1,3],[2]},  \phi^{[2,3],[1]}, \phi^{[1],[2],[3]})\]
which sends a generic point of the ordered Hilbert scheme to the product of the approximating sets:
\[\xymatrix{ \GHilb^{[1,2,3]}(M) \ar@{-->}[r]^-{\phi} & \GHilb^{[1,2,3]} \times \GHilb^{[1,2],[3]} \times \GHilb^{[1,3],[2]} \times \GHilb^{[2,3],[1]} \times \GHilb^{[1],[2],[3]}}\]
We call $N^3(M)=\overline{\im(\phi)}$ the \textit{(master) nested Hilbert scheme}, which is endowed  
with projections $\pi^{[i,j],[k]}$ and $\pi^{[1],[2],[3]}$ to the corresponding approximating sets. Pull-back along these projections provide \textit{approximating bundles}
\[V^{[i,j],[k]}=\pi^{[i,j],[k]*}(V^{[i,j]} \oplus V) \text{ and } V^{[1],[2],[3]}=\pi^*(V \oplus V \oplus V)\]
where $V^{[ij]} = V^{[2]}$ stands for the tautological bundle over $\GHilb^{[i,j]}(M)$.

Let $\hch:\GHilb^{[1],[2],[3]}(M)\to M\times M \times M$ be the ordered Hilbert-Chow morphism, and for any partition $\l$ of $\{1,2,3\}$ let 
$N^\l_0(M)=(\pi^{[1],[2],[3]})^{-1} \Delta_\l$
denote the punctual part sitting over the corresponding diagonal of $M \times M \times M$. 
The sieve formula in Step 1 of the strategy has the following form:
\begin{align} \nonumber
	         \int_{\GHilb^3(M)} \Phi(V^{[3]})&=\int_{\GHilb^{[1],[2],[3]}(M)} \Phi(V^{[1,2,3]})=\int_{N^3(M)} \pi^{[1,2,3]*}\Phi(V^{[1,2,3]})=\\ \nonumber
		&=\int_{N^3(M)}\left[ \Phi(V^{[1,2,3]})-\Phi(V^{[1,2],[3]]})-\Phi(V^{[1,3],[2]]})-\Phi(V^{[2,3],[1]]})+2\Phi(V^{[1],[2],[3]]})\right] +\\ \nonumber
		&+\int_{N^3(M)}\left[ \Phi(V^{[1,2],[3]]})-\Phi(V^{[1],[2],[3]})\right]\\ \nonumber
		&+\int_{N^3(M)}\left[ \Phi(V^{[1,3],[2]]})-\Phi(V^{[1],[2],[3]})\right]\\ \nonumber
		&+\int_{N^3(M)}\left[ \Phi(V^{[2,3],[1]]})-\Phi(V^{[1],[2],[3]})\right]\\ \nonumber
		&+\int_{N^3(M)} \Phi(V^{[1],[2],[3]}) \nonumber
		\end{align}
The second term can be written as 
\[ \int_{N^3(M)}\Phi(V^{[1,2],[3]]})-\Phi(V^{[1],[2],[3]}) = \int_{\Hilb^{[1,2]}(M) \times M} (\Phi(V^{[1,2]})-\Phi(V^{[1],[2]})) \oplus V. \]
and we can apply the residue formula we presented for $k=2$ points in the previous section. This gives the second term in Proposition \ref{prop:k=3}. The last term can be written as 
\[\int_{N^3(M)} 2\Phi(V^{[1],[2],[3]})=\int_{M\times M\times M}V_1 \oplus V_2 \oplus V_3\]
It remains to prove the residue formula for the first term, where the integrand 
\[\Phi^+=\Phi(V^{[1,2,3]})-\Phi(V^{[1,2],[3]]})-\Phi(V^{[1,3],[2]]})-\Phi(V^{[2,3],[1]]})+2\Phi(V^{[1],[2],[3]]})\] 
is supported in a neighborhood $N^3_\nabla(M)$ of the punctual part $N^3_0(M)$. Hence, as explained in \S \ref{subsec:reduction}, we can reduce the calculation to equivariant  integration assuming $M=\CC^m$. In practice this means that we apply equivariant integration, and then substitute the Chern roots of $M$ into the torus weights (this is the Chern-Weil map).

The next crucial step is to define partial blow-ups $\widehat{\GHilb}^3(\CC^m) \to \GHilb^3(\CC^m)$ and $\hat{N}^3(\CC^m) \to N^3(\CC^m)$ with fibration 
\begin{equation}\label{fibration}
\xymatrix{\hat{N}^3(\CC^m) \ar[r]^-\pi & \widehat{\GHilb}^3(\CC^m) \ar[r]^-\rho & \flag_2(T\CC^m) \ar[r] & \CC^m}
\end{equation}  
We define $\widehat{\GHilb}^3(\CC^m)$ as the closure of the graph of the rational map $\rho: \GHilb^3(\CC^m) \dasharrow \flag_{2}(T\CC^m)$ defined on the open locus formed by $3$ different points as
\[(p_1,p_2,p_3) \mapsto (\frac{1}{3}(p_1+p_2+p_3), \mathrm{Span}(p_1,p_2),\mathrm{Span}(p_1,p_2,p_3)).\]
In particular, the fiber over a flag over the origin $\bW=(0,W_1 \subset W_2 \subset \CC^m)$ is the \textit{balanced Hilbert scheme}
\[\rho^{-1}(\bW)=\BHilb^3(W_2)=\{\xi \in \GHilb^3(W_2): \text{baricenter of }\xi \text{ is the origin}\}\] 
Similarly, we define $\hat{N}^3(\CC^m)$ as the closure of the graph of the rational composition 
\[\rho \circ \pi: N^3(\CC^m) \to \GHilb^3(\CC^m) \dasharrow \flag_{2}(T\CC^m) \to \CC^m\]

We apply torus equivariant localisation for the action of the maximal torus $T \subset \GL(m)$ on the fiber $\hat{N}^3_0(\CC^m)$ over the origin in $\CC^m$:
\[\hat{N}^3_0(\CC^m) \to \flag_{2}(T_0\CC^m)=\flag_2(\CC^m).\]
This action is induced from the standard diagonal action on $\CC^m$ with weights $\l_1,\ldots, \l_m$ in the basis $\{e_1,\ldots, e_m\}$. The fibration \eqref{fibration} is equivariant, and Proposition \ref{ABtoresidue} turns our integral over $\flag_k(\CC^m)$ into an iterated residue
\begin{equation}\label{integralfirst}
\int_{\hat{N}^3(\CC^m)} \Phi^+=\int_{\CC^m}\res_{z_1=\infty} \res_{z_2=\infty}\frac{(z_1-z_2)\Phi^+_{\bff}dz_2dz_1}{\prod_{i=1}^{m}(z_1-\lambda_i)\prod_{i=1}^{m}(z_2-\lambda_i) }
\end{equation}
where integration over $\CC^m$ is a shorthand notation for substituting the Chern roots of $M$ into $\l_1,\ldots, \l_m$, followed by integration over $M$ (see \S \ref{subsec:reduction}).
Here 
\[\Phi^+_\bff=\int_{\hat{N}^3_\bff}\Phi^+\]
is the integral of the form on the fiber $\hat{N}^3_\bff=(\rho \circ \pi)^{-1}(\bff) \simeq BN^3(\CC_{[2]})$ over a fixed reference flag $\bff=(\mathrm{Span}(e_1)\subset \mathrm{Span}(e_1,e_2) \subset \CC^m)\in \flag_2(\CC^m)$. Here $\CC_{[2]}=\Span(e_1,e_2)$, and $z_1,z_2$ stand for the weights of the residual $T_\bz^2=(\CC^*)^2$ action on $\hat{N}^3_\bff$, and $BN^3(\CC_{[2]})$ is the balanced nested Hilbert scheme formed by schemes with baricenter at the origin.

We will apply a second localisation with respect to this 2-dimensional torus acting on $\hat{N}^3_\bff$. The geometry of $\hat{N}^3_\bff$ is sketched in Figure 1. The dimension is $\dim(\hat{N}^3_\bff)=3$, whose generic point is a triple $(p_1,p_2,p_3) \in (\CC^2)^{\times 3}$ with $p_1+p_2+p_3=0$, such that the line $p_1p_2$ has fixed direction $e_1$.  The punctual part 
\[N^3_0(\CC_{[2]})=\pi^{-1}(\widehat{\GHilb}^3_0)\cap \hat{N}^3_\bff\]
consists of triples $(\xi,[v_{12}],[v_{13}],[v_{23}]) \in \GHilb^3_0(\CC^2) \times \PP^1 \times \PP^1$ where $\xi$ is a punctual scheme of length $3$, and $v_{ij}$ are directions corresponding to the length-2 subschemes formed by points $i,j$, but $[v_{12}]=[e_1]$ is fixed. Note that the punctual part $\GHilb^3_0(\CC^2)=\CHilb^3(\CC^2)$ has two types of points: curvilinear subschemes which form a 1-dimensional set, and the Porteous point $\xi_{por}$ corresponding to the ideal $\mathfrak{m}^2=(x^2,xy,y^2)$. 

The punctual part $N^3_0(\CC_{[2]})$ has two components: $\PP^1=CN^3(\CC_{[2]})$ is the curvilinear component which is mapped by $\pi$ dominantly to $\GHilb^3_0(\CC_{[2]})$, and a second component $\PP^1 \times \PP^1$ which sits over the Porteous point. A point in $\PP^1 \times \PP^1$ is given by a tuple $(\xi_{por},[v_{12}],[v_{13}],[v_{23}])$ with $[v_{12}]=[e_1]$. So in fact, this is determined by the two directions $[v_{13}],[v_{23}]$, indicated in Figure 1 by yellow resp. green arrows.

\begin{figure}[ht!]
\centering
\includegraphics[width=90mm]{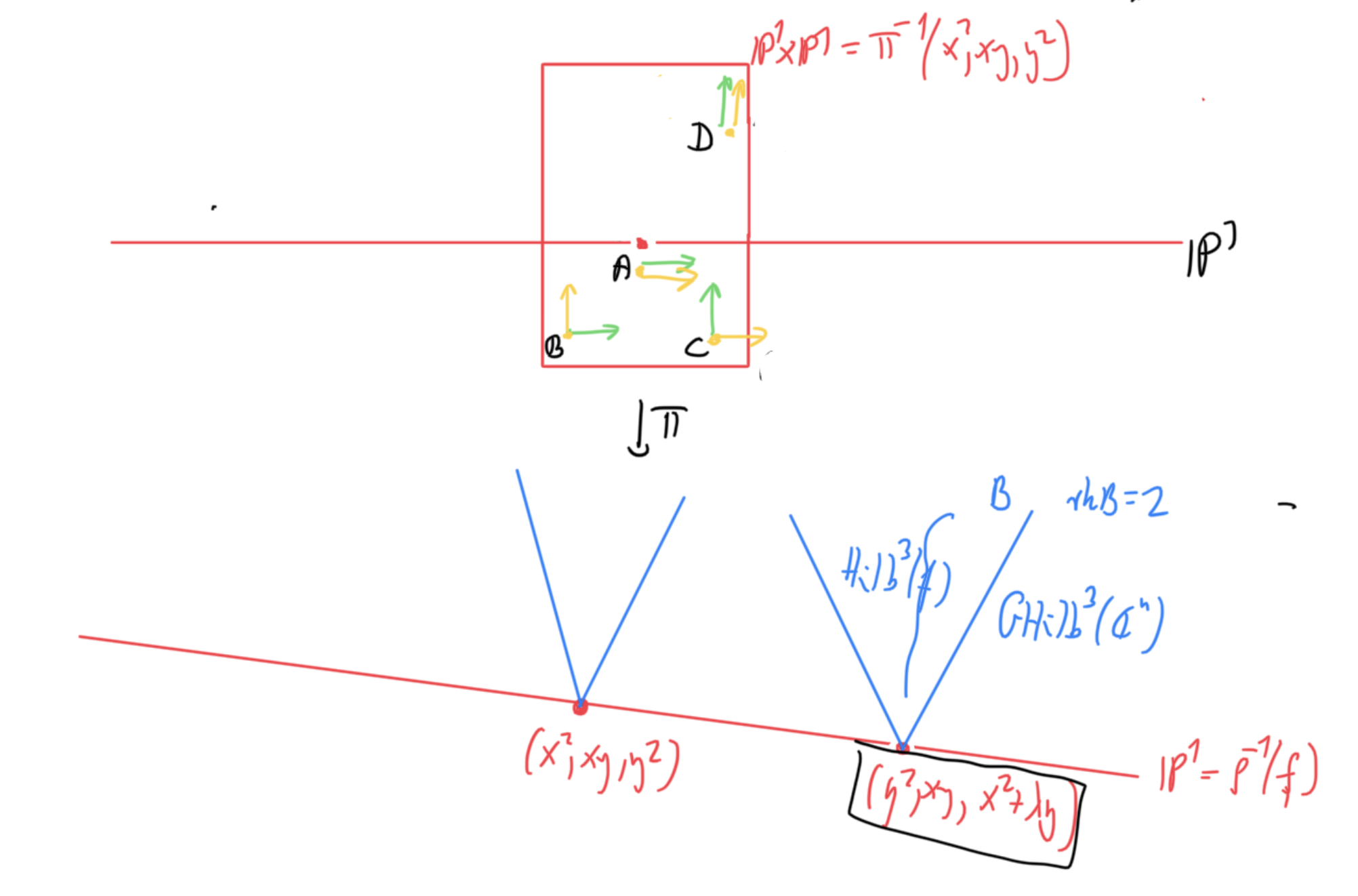}
\caption{$\pi: N^3_\bff(\CC_{[2]})\to \GHilb^3(\CC_{[2]})$}
\end{figure}

For equivariant localisation over $N^k_\bff$, in general, we will use a smooth ambient space 
\[\xymatrix{\hat{N}^k_0(\CC^m) \ar@{^{(}->}[r] \ar[d] & \Conf^k(\CC^m) \ar[ld] \\ \flag_{k-1}(\CC^m) & }\]
where $\Conf^k(\CC^m)$ stands for \textit{configuration space}, consisting of certain flag bundle of the tautological rank $k-1$ bundle over $\grass_{k-1}(T\CC^m)$, where $\GHilb^k(\CC^m)$ sits. For $k=3$ we have $\hat{N}^3(\CC^m)\simeq \Conf^3(\CC^m)$, and in particular the fibers over $\bff$ are isomorphic: $\hat{N}^3_\bff \simeq \Conf^3_\bff$.
Hence for $k=3$ Rossmann formula is not needed, but we can apply Atiyah-Bott localisation on $\hat{N}^3_\bff=N^3(\CC^2_{[2]})$ directly. We introduce the notations
\begin{equation}\label{equivbundles}
V_{x,y}^{[1,2,3]}=V\oplus V(x) \oplus V(y),\ \ V_{x}^{[i,j],[k]}=V \oplus V(x) \oplus V,\ \ V^{[1],[2],[3]}=V \oplus V \oplus V
\end{equation}
The following table contains the restrictions of the terms of $\Phi^+_\bff$ at different fixed points $F \in (\hat{N}^3_\bff)^{T_\bz^2}$.
\begin{center}
\begin{tabular}{|c|c|c|c|c|c|}
\hline
$F$ & $V_F^{[1,2,3]}$ & $V_F^{[1,2],[3]}$ & $V_F^{[1,3],[2]}$ & $V_F^{[2,3],[1]}$ & $V_F^{[1],[2],[3]}$ \\
\hline
$(\xi_{por},e_1,e_1,e_2)$ & $V_{z_1,z_2}^{[1,2,3]}$ & $V_{z_1}^{[1,2],[3]}$ & $V_{z_1}^{[1,3],[2]}$ & $V_{z_2}^{[2,3],[1]}$ & $V^{[1],[2],[3]}$ \\
\hline
$(\xi_{por},e_1,e_2,e_1)$ & $V_{z_1,z_2}^{[1,2,3]}$ & $V_{z_1}^{[1,2],[3]}$ & $V_{z_2}^{[1,3],[2]}$ & $V_{z_1}^{[2,3],[1]}$ & $V^{[1],[2],[3]}$ \\
\hline
$(\xi_{por},e_1,e_1,e_1)$ & $V_{z_1,z_2}^{[1,2,3]}$ & $V_{z_1}^{[1,2],[3]}$ & $V_{z_1}^{[1,3],[2]}$ & $V_{z_1}^{[2,3],[1]}$ & $V^{[1],[2],[3]}$ \\
\hline
$(\xi_{por},e_1,e_2,e_2)$ & $V_{z_1,z_2}^{[1,2,3]}$ & $V_{z_1}^{[1,2],[3]}$ & $V_{z_2}^{[1,3],[2]}$ & $V_{z_2}^{[2,3],[1]}$ & $V^{[1],[2],[3]}$ \\
\hline
$(\xi_{curv},e_1,e_1,e_1)$ & $V_{z_1,z_1}^{[1,2,3]}$ & $V_{z_1}^{[1,2],[3]}$ & $V_{z_1}^{[1,3],[2]}$ & $V_{z_1}^{[2,3],[1]}$ & $V^{[1],[2],[3]}$ \\
\hline
\end{tabular}
\end {center}
Note that here $\xi_{por}=(x^2,xy,y^2)$ stands for the Porteous point and $\xi_{curv}=(y,x^3)$ stands for the curvilinear fixed point in $\Hilb^3(\CC^2)$. Substitute these into the Atiyah-Bott formula:
\small{\[
\int\displaylimits_{\hat{N}^3(\CC^m)}\Phi^+=\res_{z_1,z_2=\infty}\sum\limits_{F \in (\hat{N}^3_\bff)^{T_\bz^2}} \frac{(z_1-z_2)\left(\Phi(V_F^{[1,2,3]})-\Phi(V_F^{[1,2],[3]})-\Phi(V_F^{[1,3],[2]})-\Phi(V_F^{[2,3],[1]})+2\Phi(V_F^{[1],[2],[3]})\right)}{\Euler^z(T_FN^3(\CC^2)) \prod_{i=1}^{m}(z_2-\lambda_i)\prod_{i=1}^{m}(z_1-\lambda_i) }\]}
The crucial observation is that only $V_F^{[1,2,3]}$ depends on both weights $z_1,z_2$. For all other terms, one of $z_1,z_2$ or both are missing, and hence the degree of the denominator in this missing variable (which is $\sim m$) is much bigger than the degree of the numerator, which only depends on $k$ in general. This is the essence of our first residue vanishing theorem (Theorem \ref{vanishing1}). Hence we are left with
\begin{equation}
\int_{\hat{N}^3(\CC^m)} \Phi^+= \res_{z_1=\infty} \res_{z_2=\infty}\sum\limits_{F \in (\hat{N}^3_\bff)^{T_\bz^2}} \frac{(z_1-z_2)\Phi(V\oplus V(z_1) \oplus V(z_2))dz_2dz_1}{\Euler^z(T_FN^3(\CC^2))\prod_{i=1}^{m}(z_2-\lambda_i)\prod_{i=1}^{m}(z_1-\lambda_i) },
 \end{equation}
However, the form $V_{z_1,z_2}^{[1,2,3]}=V\oplus V(z_1) \oplus V(z_2)=\pi^*V_{z_1,z_2}^{[1,2,3]}$ is a pull-back form from $\pi: N^3(\CC^m) \to \Hilb^3(\CC^m)$. Hence we can integrate over the Hilbert scheme:
\[\small{
\int\displaylimits_{\hat{N}^3(\CC^m)} \Phi^+= \res_{z_1,z_2=\infty}\sum\limits_{F \in \Hilb^3(\CC^2)^{T_\bz^2}}\frac{(z_1-z_2)\Phi(V\oplus V(z_1) \oplus V(z_2))dz_2dz_1}{\Euler^z(T_F(\Hilb^3(\CC^2)) \prod_{i=1}^{m}(z_2-\lambda_i)\prod_{i=1}^{m}(z_1-\lambda_i) }=\int\displaylimits_{\Hilb^3_\nabla(\CC^m)} \Phi(V_{z_1,z_2}^{[1,2,3]})
}\]
The next crucial observation is that the small neighborhood $\Hilb^3_\nabla(\CC^m)$ of $\Hilb^3_0(\CC^m)=\CHilb^3(\CC^m)$ can be modeled as a bundle over $\CHilb^3(\CC^m)$. More precisely, we will identify the normal bundle as a rank $2$ bundle $B$ over $\CHilb^3(\CC^m)$ such that the total space of $B$ is topologically isomorphic to $\Hilb^3_\nabla(\CC^m)$. We will see that $B$ is a tautological bundle over the punctual part, namely 
\[B|_{\CHilb^k(\CC^m)}=\calo_M^{[k]}/\calo_M\]
and hence the equivariant bundle over the fiber $\CHilb^3(\CC^2)$ over the reference flag $\bff$ has equivariant Euler class 
\[\Euler(B_{z_1.z_2})=z_1z_2\]
The Thom isomorphism and the iterated residue formula of \cite{b0,bsz} for integrals over the curvilinear component then gives 
\[\int_{\Hilb^3_\nabla(\CC^m)} \Phi=\int_{\CHilb^3(\CC^m)} \frac{\Phi}{\Euler(B)}=\res_{z_1=\infty} \res_{z_2=\infty}\frac{(z_1-z_2)\Phi(V\oplus V(z_1) \oplus V(z_2))dz_1dz_2}{(z_1z_2)(2z_1-z_2)\prod_{i=1}^{m}(z_2-\lambda_i)\prod_{i=1}^{m}(z_1-\lambda_i) }\]
After the substitution $z_i \to -z_i$ and the identity 
\[\frac{1}{\prod_{i=1}^{m}(z_1+\lambda_i)\prod_{i=1}^{m}(z_2+\lambda_i)}(\l_i \to \text{Chern roots of } M)=\frac{s_M(1/z_1)s_M(1/z_2)}{(z_1z_2)^m}\]
completes the proof of Proposition \ref{prop:k=3}.
\end{proof}

The next step is, again, to extend this formula to the case when the push-forward to the 
point is replaced by a push-forward along a smooth map between varieties, 
$f:M\to N$. \\
Notation:
\begin{itemize}
	\item We will use the standard notation $f_*[\Psi]$ for the push-forward 
	of a cohomology class $\Psi\in H^*(M)$. As before, we will indicate the 
	source space in the index for clarity, as in  $f_*[\Psi]_M$. 
	\item We will write $ f^{\times3}: M \times M \times M \to N \times N \times N$, and $ f^{[3]} =  f^{\times3}\circ\hch$. For the big diagonals we write the compositioin
	\[\xymatrix{f^{\times i,j}: \Delta_{i,j}M \times M \ar[r]^-{(f,f)} & \Delta_{i,j}N \times N} \] 
	\item The embeddings of the small diagonals will be denoted by $\iota_{\Delta M}:\Delta M\hookrightarrow M\times M \times M$ and $i_{\Delta N}: \Delta N\hookrightarrow N \times N \times N$. For the big diagonals we use $\iota_{\Delta_{i,j}N}: \Delta_{i,j}N \times N \to N \times N \times N$. 
\end{itemize}
First, the proof of Proposition \ref{prop:k=3} gives 
\begin{align*}
  f^{[3]}_*[\Phi(V^{[3]})]_{\GHilb^3(M)}=& \iota_{\Delta N*} f_* \left[\res_{z_1,z_2=\infty} \frac{\Phi\left(V\oplus V(z_1)\oplus V(z_2) \right)(z_1-z_2)dz_1dz_2}{(z_1z_2)^{m+1}(2z_1-z_2)}\, s_M\left(\frac{1}{z_1}\right) s_M\left(\frac{1}{z_2}\right)\right]_M+\\
 & +\sum_{i\neq j \in \{1,2,3\}} \iota_{\Delta_{i,j} N*} f_*^{\times i,j}\left[\res_{z=\infty} \frac{\Phi\left(V\oplus V(z) \oplus V \right)dz}{z^{m+1}}\,
		s_M\left(\frac{1}{z}\right)\right]_{\Delta_{i,j}M \times M}\\
&+f^{\times3}_* [\Phi(V_1 \oplus V_2 \oplus V_3)]_{M \times M \times M}  
\end{align*}


Suppose that $V^{[3]}$ has a subbundle of the form $f^{[3]*}W\hookrightarrow V^{[3]}$ for some bundle $W$ over $N\times N \times N$ and this is consistent with an embedding 
\[f^{\times3*}W\hookrightarrow V_1\oplus V_2 \oplus V_3 \text{\ \   and\ \   } HC_{i,j}^*f^{\times 3*}W\hookrightarrow V^{[i,j],[k]}.\]
where
\[\HC_{i,j}:\Hilb^{[i,j]}(M) \times M\to M \times M \times M\]
is the Hilbert-Chow for the points labeled by $i,j$. Using the identifications in \eqref{equivbundles}, we write
\begin{align}\label{relative3pt}
f^{[3]}_*\left[\Phi\left(\frac{V^{[3]}}{f^{[3]*}W}\right)\right]_{\GHilb^3(M)}  =&\iota_{\Delta N*} f_* \left[\res_{z_1,z_2=\infty} \Phi\left(\frac{V\oplus V(z_1)\oplus V(z_2)}{f^*\iota_{\Delta N}^*W} \right)\frac{(z_1-z_2)dz_1dz_2}{(z_1z_2)^{m+1}(2z_1-z_2)}\, s_M\left(\frac{1}{z_1}\right) s_M\left(\frac{1}{z_2}\right)\right]_M+\\ \nonumber
&+\sum_{i\neq j \in \{1,2,3\}} \iota_{\Delta_{i,j} N*} f_*^{\times i,j}\left[\res_{z=\infty} \Phi\left(\frac{V\oplus V(z) \oplus V }{f^{\times i,j*} \iota_{\Delta_{i,j}}^*W}\right)\frac{dz}{z^{m+1}}\,
		s_M\left(\frac{1}{z}\right)\right]_{\Delta_{i,j}M \times M}\\ \nonumber
& +f^{\times3}_* \left[\Phi\left(\frac{V_1 \oplus V_2 \oplus V_3}{f^{\times3*}W}\right)\right]_{M \times M \times M}. \nonumber
\end{align}

As $f(\Hilb^3(f))$ is a subset of the diagonal $\Delta N\subset N\times N \times N$, our calculations will take place in a neighborhood of $\Delta N$, which we will identify with the normal bundle $\nabla_{ N}$ of $\Delta N$ in $N\times N \times N$, and in compactly supported cohomology in the open set $U=\left(f^{[3]}\right)^{-1}\nabla_{ N}$. 

Proposition \ref{prop:sectionk=2} has the following extension: 
\begin{proposition}\label{prop:sectionk=3}
	Let  $T_v\to\nabla_N$ be the tautological pull-back bundle on $\nabla_{N}$  and $ T_h=\pi^*T\Delta N\to\nabla_N $ be the pull-back from the base. Then there is a bundle embedding  $f^{[3]*}T_h\hookrightarrow f^*TN^{[3]}$  consistent with an embedding 
\[f^{\times3*}T_h\hookrightarrow f^*TN_1\oplus f^*TN_2 \oplus f^*TN_3 \text{. and } (f^{\times 3)} \circ HC_{i,j})^*T_h \hookrightarrow (f^*TN)^{[i,j],[k]}.\] 
In addition, the quotient bundle
	$f^*TN^{[3]}/f^{[3]*}T_h$ has a transversal section presenting $\Hilb^3_0(f)$ as a complete intersection in $U$.
\end{proposition} 

We are left with repeating the final calculation \eqref{finalk=2}.  The Thom isomorphism for the space $\nabla_N$ gives again  
\[  \hcpt(\nabla_N)\simeq H^*(\Delta_N)\cdot c_n( T_v).
 \] 
We also, clearly have $T\nabla_N\simeq TN_1\oplus TN_2 \oplus TN_3\simeq T_h\oplus T_v$, and we can apply \eqref{relative3pt} with $V=f^*TN$, $W=T_h$ and $\Phi=c_{2n}$. Let $\pi_{\Delta N}: \nabla N \to \Delta N$ denote the projection (see e.g. Diagram \ref{whatwewant})
\begin{align*}
	&[f(\Hilb^3(f))]=
	\pi_{\Delta N*}f^{[3]}_*\left[c_{2n}\left(\frac{f^*TN^{[3]}}{f^{[3]*}T_h}\right)\right]_{U\subset\GHilb^3(M)}= \\ 
	&= \frac1{c_{2n}(T_v)} \iota_{\Delta N*} f_* \left[\res_{z_1,z_2=\infty} c_{2n}\left(\frac{f^*TN\oplus f^*TN(z_1)\oplus f^*TN(z_2)}{f^*\iota_{\Delta N}^*T_h} \right)\frac{(z_1-z_2)dz_1dz_2}{(z_1z_2)^{m+1}(2z_1-z_2)}\, s_M\left(\frac{1}{z_1}\right) s_M\left(\frac{1}{z_2}\right)\right]_M+ \\ \nonumber
&+\sum_{i\neq j \in \{1,2,3\}} \frac1{c_{2n}(T_v)}\iota_{\Delta_{i,j} N*} f_*^{\times i,j}\left[\res_{z=\infty} c_{2n}\left(\frac{f^*TN_i\oplus f^*TN_j(z) \oplus f^*TN_k }{f^{\times i,j*} \iota_{\Delta_{i,j}}^*T_h}\right)\frac{dz}{z^{m+1}}\, s_M\left(\frac{1}{z}\right)\right]_{\Delta_{i,j}M \times M}+  \\ \nonumber
& +\frac1{c_{2n}(T_v)}	f^{\times3}_* \left[c_{2n}\left(\frac{f^*TN_1 \oplus f^*TN_2 \oplus f^*TN_3}{f^{\times3*}T_h}\right)\right]_{M \times M \times M}= \\
& 	=f_* \left[\res_{z_1,z_2=\infty} c_{2n}\left(f^*TN(z_1)\oplus f^*TN(z_2)\right)\frac{(z_1-z_2)dz_1dz_2}{(z_1z_2)^{m+1}(2z_1-z_2)}\, s_M\left(\frac{1}{z_1}\right) s_M\left(\frac{1}{z_2}\right)\right]_M+ \\
&+\sum_{i\neq j \in \{1,2,3\}} f_*\left[\res_{z=\infty} c_n\left(f^*TN(z)\right)\frac{dz}{z^{m+1}}\, s_M\left(\frac{1}{z}\right) \right]_{\Delta_{i,j}M \times M}+\iota_{\Delta N}^*f^{\times3}_*[1]= \\
&=f_*\Tp^{m \to n-1}_{A_2}(f)+f_*\Tp^{m \to n-1}_{A_1}f_*[1]+f_*[1]\cdot f_*[1] \cdot f_*[1]
\end{align*}
which completes the proof of Theorem \ref{mainthm1} and Theorem \ref{mainthm2} for $k=3$. We used the identity
\begin{multline}
\frac{c_{2n}\left(f^*TN(z_1)\oplus f^*TN(z_2)\right)(z_1-z_2)}{(z_1z_2)^{m+1}(2z_1-z_2)}\, s_M\left(\frac{1}{z_1}\right) s_M\left(\frac{1}{z_2}\right)=\\ \nonumber
\frac{z_1-z_2}{(z_1z_2)^{m+1}(2z_1-z_2)}\cdot \frac{\prod_{j=1}^n (\theta_j+z_1)(\theta_j+z_2)}{\prod_{i=1}^m (\l_i+z_1)(\l_i+z_2)}=
\frac{z_1-z_2}{(z_1z_2)^{m-n+1}(2z_1-z_2)} \cdot c_f(1/z_1)c_f(1/z_2)
\end{multline}
which is the residue formula of \cite{bsz} for Thom polynomials (see the Introduction).


\subsection{The integral formula in general}\label{subsec:integralformula}

We formulate the geometric conditions under which our theorem holds. 
\begin{definition}\label{geocond} Let $V$ be a rank $r$ vector bundle over the smooth projective variety $M$ of dimension $m$. We say that the Chern class $c_d(V^{[k]})$ of its $k$th tautological bundle over $\GHilb^k(M)$ is properly supported if $\supp(c_d(V^{[k]}))$ is irreducible, and intersects the punctual Hilbert scheme only in the curvilinear component:
\begin{equation}\label{supportcond}
\mathrm{supp}(c_d(V^{[k]})) \cap \GHilb^k_0(M) \subset \CHilb^k_0(M).
\end{equation}
\end{definition}
\begin{remark} \begin{enumerate}
\item In particular, if there is a section $\s$ of $V^{[k]}$ whose zero set satisfies $\s^{-1}(0)\cap \GHilb^k_0(M) \subset \CHilb^k_0(M)$, and it is locally irreducible, then we say that the Euler class $\Euler(V^{[k]})$ is properly supported.
\item Being properly supported is a strong condition for tautological classes, and typically happens when $V=f^*E$ is a pull-back along some sufficiently nice $f:M \to N$ whose properties impose geometric conditions on the support of the sections of $V^{[k]}$. We will see that our candidate $\Euler(V)$ is indeed properly supported.  
\end{enumerate}
\end{remark}


We denote by $\Pi(k)$ the set of partitions of $\{1,\ldots, k\}$ into nonempty subsets; an element $\a=\{\a_1,\ldots, \a_s\} \in \Pi(k)$ consists of subsets $\a_i \subset \{1,\ldots, k\}$ such that $\a_i\cap \a_j=\emptyset$ for $1\le i<j\le s$ and $\{1,\ldots, k\}=\cup_{i=1}^s \a_i$. We introduce a set of variables 
\[\bz^{\a_i}=\{z^{\a_i}_1,\ldots, z^{\a_i}_{|\a_i|-1}\}\]
for each element of the partition.

\begin{theorem}[\textbf{Integrals over $\GHilb^{k+1}(M)$}]\label{tauintegral}
Let $V$ be a rank $r$ vector bundle over the smooth projective variety $M$ of dimension $m$. Assume that $r,k,m$ satisfy the inequalities
\[k\le m \le r \le (m-1)\frac{k}{k-1},\]
and $c_{(k+1)m}(V^{[k+1]})$ is properly supported on $\GHilb^{k+1}(M)$. Then    
\[\int_{G\Hilb^{k+1}(M)}c_{(k+1)m}(V^{[k+1]}) =\sum_{(\a_1,\ldots, \a_s) \in \Pi(k+1)} \int_{M^s} \res_{\bz^{\a_1}=\infty}\ldots \res_{\bz^{\a_s}=\infty} \mathcal{R}^\a(\theta_i, \bz) d\bz^{\a_s}\ldots d\bz^{\a_1}\]
where $\mathcal{R}^\a(\theta_i, \bz)$ stands for the rational form 
\[c_{(k+1)m}(V(\bz^{\a_1}) \oplus \ldots \oplus V(\bz^{\a_s})) 
\prod_{l=1}^s \frac{(-1)^{|\a_l|-1} \prod_{1\le i<j \le |\a_l|-1}(z_i^{\a_l}-z_j^{\a_l})Q_{|\a_l|-1}d\bz^{\a_l}}
{\prod_{i+j\le l\le 
|\l_l|-1}(z_i^{\a_l}+z_j^{\a_l}-z_l^{\a_l})(z_1^{\a_l}\ldots z_{|\a_l|-1}^{\a_l})^{m+1}}\prod_{i=1}^{|\a_l|-1} 
s_M\left(\frac{1}{z_i^{\a_l}}\right).\]
Here we use the following notations:
\begin{itemize}
	\item $\res_{\bz^{\a_l}=\infty}=\res_{z^{\a_l}_1=\infty} \ldots \res_{z^{\a_l}_{|\a_l|-1}=\infty}$ is the iterated residue at infinity.
	\item $s_M\left(\frac{1}{z}\right)$ is the Segre series of $M$ and $Q_k$ are the universal Borel-multidegrees explained in \cite{bsz}.
	\item $V(z)$ stands for the bundle $V$ tensored by the line $\mathbb{C}_z$, which is the representation of a torus $T$ with weight $z$. Hence its Chern roots are $z+\theta_1,\ldots ,z+\theta_r$. Moreover 
	\[V(\bz^{\a_l})=V \oplus V(z^{\a_l}_1) \oplus \ldots \oplus V(z^{\a_l}_{|\a_l|-1})\]
	has rank $r|\a_l|$, and 
	\begin{equation}\label{residueeuler} \Phi(V(\bz^{\a_1}) \oplus \ldots \oplus V(\bz^{\a_s})) =  \Phi(\theta_i^1,z^{\a_1}_1+\theta_i,\ldots z^{\a_1}_{|\a_1|-1}+\theta_i, \ldots, \theta_i^s,z^{\a_r}_1+\theta_i,\ldots z^{\a_r}_{|\a_r|-1}+\theta_i)
	\end{equation}
	Note that we have $s$ copies of the roots $\theta_1,\ldots , \theta_r$ of $V$, and we think of the $i$th copy $\theta_1^i,\ldots , \theta_r^i$ as the Chern roots of $V=V^i$ sitting over the $i$th copy of $M$ in $M^s$.
	\item The iterated residue is a homogeneous symmetric polynomial of degree $ns$ in the Chern roots $\theta^i_j$, that is, the Chern roots of $V^1 \oplus \ldots \oplus V^s$ over $M^s$, and integration over $M^s$ evaluates the homogeneous $(n,n,\ldots,n)$ part with at fundamental class of $M^s$.   
\end{itemize}

\end{theorem}

\noindent \textbf{Remarks on the formula}  Let us add a few remarks and comments on the residue formula.
\begin{enumerate}
\item First, if we set the degree of every variable $z_i^{\a_l}$ and $\theta_i^j$ to be $1$, then the total degree of the rational expression $\mathcal{R}^\a(\theta_i, \bz)$ is $mk-(m+1)(k-s)=(m+1)s-k$.
The iterated residue has $k-s$ variables, and hence 
\[\res_{\bz^{\a_1}=\infty}\ldots \res_{\bz^{\a_s}=\infty} \mathcal{R}^\a(\theta_i, \bz) d\bz^{\a_s}\ldots d\bz^{\a_1}\]
is a symmetric polynomial in the Chern roots $\theta_1^i,\ldots, \theta_r^i$ of $V^i$ for $1\le i \le s$, of degree 
\[(m+1)s-k+(k-s)=ms=\dim(M^s)\]
Here $V^i$ is the pull-back of $V$ from the $i$th factor in $M^s$, and integration over $M^s$ gives the result.  
This shows that the dependence on Chern classes of $M$ in fact can be expressed via the Segre classes of $M$. 
For fixed $k$ the rational expression in $\mathcal{R}^\a$ in the formula is independent of the dimension $m$ of $M$, and the iterated residue depends on $m$ only through the total Segre class $s_M$ of $M$. The iterated residue is then some linear combination of the coefficients of (the expansion of) $\mathcal{R}^\a$'s multiplied by Segre classes of $M$. By increasing the dimension, the iterated residue involves new terms of the expansion of $\mathcal{R}^\a$'s, and we can think of the our residue formula as a generating function of tautological integrals for fixed $k$ but varying $m$. 
\item The formula covers the domain where $k\ll m<r\ll 2m$. These are pretty restrictive for the parameters involved, and in a forthcoming paper we will study how to relax on these conditions.   
\end{enumerate}

Let $M=\CC^m$ be the affine space with the standard $\GL(m)$ action which induces a $\GL(m)$ action on the Hilbert scheme $\GHilb^{k+1}(\CC^m)$. We formulate the equivariant version of Theorem \ref{tauintegral}, which is very similar, but a key difference is that we can integrate (push-forward) equivariant Chern classes $c_d(V^{[k+1]})$ with $d\ge \dim(\GHilb^{k+1}(\CC^m)$.   Let $\l_1,\ldots, \l_m$ denote the weights of the maximal torus $T \subset \GL(m)$ on $\CC^m$. 

\begin{theorem}[\textbf{Equivariant integrals over $\GHilb^{k+1}(\CC^m)$}]\label{equivariantintegral}
Let $\l_1,\ldots, \l_m$ denote the weights for the maximal torus $T \subset \GL(m)$ acting on $\CC^m$. Let $V$ be an equivariant bundle over $\CC^m$ with weights (i.e equivariant Chern root) $\theta_1, \ldots, \theta_r$. Assume that $r,k,m$ satisfy the inequalities $k \le m$ and
\[m \le r \le (m-1)\frac{k}{k-1},\]
and suppose that $c_{d}(V^{[k+1]})$ is properly supported on $\GHilb^{k+1}(M)$ for some 
\[\mathrm{rank}(V^{k+1})=(k+1)r \ge d\ge (k+1)m=\dim(\GHilb^{k+1}(\CC^m)).\] 
Then    
\[\int_{G\Hilb^{k+1}(\CC^m)}c_{d}(V^{[k+1]}) =\sum_{(\a_1,\ldots, \a_s) \in \Pi(k)} \int_{M^s} \res_{\bz^{\a_1}=\infty}\ldots \res_{\bz^{\a_s}=\infty} \mathcal{R}^\a(\theta_i, \bz) d\bz^{\a_s}\ldots d\bz^{\a_1}\]
where $\mathcal{R}^\a(\theta_i, \bz)$ is the rational form 
\[c_d(V(\bz^{\a_1}) \oplus \ldots \oplus V(\bz^{\a_s})) 
\prod_{l=1}^s \frac{(-1)^{|\a_l|-1} \prod_{1\le i<j \le |\a_l|-1}(z_i^{\a_l}-z_j^{\a_l})Q_{|\a_l|-1}d\bz^{\a_l}}
{\prod_{i+j\le l\le 
|\l_l|-1}(z_i^{\a_l}+z_j^{\a_l}-z_l^{\a_l})(z_1^{\a_l}\ldots z_{|\a_l|-1}^{\a_l})\prod_{i=1}^{|\a_l|-1} \prod_{j=1}^m
\left(\l_j-z_i^{\a_l}\right)}\]
with the same notations as in Theorem \ref{tauintegral}.
\end{theorem}

\section{Local models for Hilbert schemes}\label{sec:model}


In this section let $X$ be a smooth projective variety of dimension $m$, and let $J_kX \to X$ be the bundle of $k$-jets of germs of parametrized curves in $X$; that is, $J_kX$ is the of equivalence classes of germs of holomorphic
maps $f:(\CC,0) \to (X,p)$, with the equivalence relation $f\sim g$
if and only if the derivatives $f^{(j)}(0)=g^{(j)}(0)$ are equal for
$0\le j \le k$ when computed in some local coordinate system of $X$ near $p\in X$. The projection map $J_kX \to X$ is simply $f \mapsto f(0)$. If we choose local holomorphic coordinates on an open neighbourhood $\Omega \subset X$
around $p$, the elements of the fibre $J_kX_p$ can be represented by Taylor expansions 
\[f(t)=p+tf'(0)+\frac{t^2}{2!}f''(0)+\ldots +\frac{t^k}{k!}f^{(k)}(0)+O(t^{k+1}) \]
 up to order $k$ at $t=0$ of $\CC^m$-valued maps $f=(f_1,f_2,\ldots, f_m)$
on open neighbourhoods of 0 in $\CC$. Locally in these coordinates the fibre $J_kX_p$ can be identified with the set of $k$-tuples of vectors $(f'(0),\ldots, f^{(k)}(0)/k!)=(\CC^m)^k$ 
which further can be identified with $J_k(1,m)$.

Note that $J_kX$ is not
a vector bundle over $X$ since the transition functions are polynomial but
not linear. In fact, let $\diff_X$ denote the principal $\diff_k(m)$-bundle over $X$ formed by all local polynomial coordinate systems on $X$. Then 
\[J_kX=\diff_X \times_{\diff_k(m)} J_k(1,m).\] 
is the associated bundle whose structure group is $\diff_k(m)$.

Let $J_k^{\reg}X$ denote the bundle of $k$-jets of germs of parametrized regular curves in $X$, that is, where the first derivative $f'\neq 0$ is nonzero. After fixing local coordinates near $p\in X$ the fibre $J_k^{\reg}X_p$ can be identified with $\jetreg 1m$ and 
\[J_k^{\reg}X=\diff_X \times_{\diff_k(m)} J_k^\reg(1,m).\]

Let $\cald_X^{\le k}$ denote the dual bundle formed by at most $k$th order differential operators over $X$. Then $\cald_X^{\le 0}=\calo_X$, and we let $\cald_X^k=\cald_X^{\le k}/\cald_X^{\le 0}$. We have a filtration of bundles 
\begin{equation}\label{dfiltration}
\calo_X=\cald_X^{\le 0} \subset \cald_X^{\le 1} \subset \ldots \subset \cald_X^{\le k}
\end{equation}
where the graded component $\cald_X^{\le i}/\cald_X^{\le i-1}\simeq \Sym^iTX$ but this filtration is not split in general, so $\cald_X^{\le k} \nsimeq \Sym^{\le k}TX$. After choosing local coordinates on $X$ near $p$, the fibre $\cald^{k}_{X,p}$ can be identified with the space $J_k(m,1)^*\simeq \symdot$ of $k$th order differential operators on $\CC^n$.
Therefore $\cald_X^k$ is the associated bundle 
\[\cald_X^k=\diff_X \times_{\diff_k(m)} J_k(m,1)^*=\diff_X \times_{\diff_k(m)} \symdot.\]
Given a regular $k$-jet $f:(\CC,0) \to (X,p)$ we may push forward the differential operators of order k on $\CC$ to $X$ along $f$ and obtain a $k$-dimensional subspace of $\cald^{\le k}_{X,p}$. This gives the bundle map 
\begin{equation}\label{diffoppushforward}
J_k^\reg X \to \grass_k(\cald_X^{k})
\end{equation} 
Note that $\diff_k(1)=\jetreg 11$ acts fibrewise on the jet bundle $J_k^{\reg}X$ and the map \eqref{diffoppushforward} is $\diff_k(1)$-invariant resulting an embedding 
\begin{equation}\label{invdiffoppush}
\tilde{\phi}: J_k^{\reg}X/\diff_k(1) \hookrightarrow \grass_k(\cald_X^{k})
\end{equation}
This is the fibered version of the map 
\begin{equation}\label{embedgrass1}
\phi:J_k(1,m)/\diff_k(1) \hookrightarrow \grass(k,J_k(m,1)^*)
\end{equation}
\begin{theorem}[\textbf{The test curve model of Hilbert schemes}, \cite{bsz}, \cite{b}]\label{bszmodel}
\begin{enumerate}
\item For any $k,m$ we have 
\[\CHilb^{k+1}_0(\CC^m)=\overline{\mathrm{im}(\phi)} \subset \grass_k(J_k(m,1)^*)=\grass_k(\mathfrak{m}),\]
where $\mathfrak{m}$ is the maximal ideal at the origin. 
\item Let $\{e_1,\ldots, e_m\}$ be a basis of $\CC^m$. For $k\le m$ the $\GL(m)$-orbit of 
\[p_{k,m}=\phi(e_1,\ldots, e_k)=\mathrm{Span}_\CC (e_1,e_2+e_1^2,\ldots, \sum_{i_1+\ldots +i_r=k}e_{i_1}\ldots e_{i_r})\] 
forms a dense subset of the image $\jetreg 1m$ and
\[\CHilb^{k+1}_0(\CC^m)=\overline{\GL_m \cdot p_{k,m}}.\] 
\item There are two canonically defined bundles on $\CHilb^{k+1}_0(\CC^m)$: 
\begin{enumerate}
\item The tautological rank $k+1$ bundle $\calo_{\CC^m}^{[k+1]}$. 
\item The tautological rank $k$ bundle $\cale$, which is the restriction of the tautological bundle over $\grass_k(\wedge^k \symdot)$. 
\end{enumerate}
These fit into the short exact sequence 
\[\xymatrix{0 \ar[r] & \calo_\grass \ar[r] & \calo_{\CC^m}^{[k+1]} \ar[r] & \cale \ar[r] & 0}.\]
\item Similarly, we get an embedding of the curvilinear Hilbert scheme 
\[\CHilb^{k+1}(X)=\overline{\mathrm{im}(\tilde{\phi})} \subset \grass_k(\cald_X^{k}),\] 
and for a bundle $V$ over $X$ we obtain the short exact sequence of bundles over $\CHilb^{k+1}(X)$:
\[\xymatrix{0 \ar[r] & V \ar[r] & V^{[k+1]}=\calo_{X}^{[k+1]}\otimes V \ar[r] & \cale_X \ar[r] & 0}\]
where $\cale_X$ is the tautological bundle over $\grass_k(\cald_X^{k})$.
\end{enumerate}
\end{theorem}

\subsection{Fibration over the flags in $TX$.}

Let $k\le m$ and let $P_{m,k} \subset \GL_m$ denote the parabolic subgroup which preserves the flag 
\[\mathbf{f}=(\mathrm{Span}(e_1)   \subset \mathrm{Span}(e_1,e_2) \subset \ldots \subset \mathrm{Span}(e_1,\ldots, e_k) \subset \CC^m).\] 
\begin{definition}\label{def:xktilde}
Define the partial desingularization 
\[\widetilde{\CHilb}^{k+1}_0(\CC^m)=\GL_m \times_{P_{m,k}} \overline{P_{m,k} \cdot p_k}\]
with the resolution map $\rho: \widetilde{\CHilb}^{k+1}_0(\CC^m) \to \CHilb^{k+1}_0(\CC^m)$ given by $\rho(g,x)=g\cdot x$.
\end{definition} 

Equivalently, let $\jetnondeg 1m \subset \jetreg 1m$ be the set of test curves with $\g',\ldots, \g^{(k)}$ linearly independent. These correspond to the nonsingular $m \times k$ matrices in $\Hom(\CC^k,\CC^m)$, and they fibre over the set of complete flags in $\CC^m$:
\begin{equation}\label{proj2}
\jetnondeg 1m/\diff_k(1) \to \Hom(\CC^k,\CC^m)/B_k=\flag_k(\CC^m)
\end{equation}
where $B_k \subset \GL(k)$ is the upper Borel. The image of the fibres under $\phi$ are isomorphic to $P_{m,k} \cdot p_k$, and therefore $\widetilde{\CHilb}^{k+1}_0(\CC^m)$ is the fibrewise compactification of $\jetnondeg 1m$ over $\flag_k(\CC^m)$. 

We construct a fibred version of $\widetilde{\CHilb}^{k+1}_0(\CC^m)$, a fibrewise partial desingularization 
\begin{equation}\label{desing}
\rho:\widetilde{\CHilb}^{k+1}(X) \to \CHilb^{k+1}(X)
\end{equation}
over $X$, where $\widetilde{\CHilb}^{k+1}(X)$ is a locally trivial bundle over $X$ with fibres isomorphic to $\widetilde{\CHilb}^{k+1}_0(\CC^m)$. 
The immediate problem is that $J_k^{\reg}X/\diff_k(1)$ is embedded into $\flag_k(\cald_X^k)$ but there is no projection $\cald_X^k \to TX$ to define a fibration of $J_k^{\reg}X/\diff_k(1)$ over $\flag_k(TX)$. To resolve this problem we will work with a linearised bundle $\CHilb^{k+1}(X)^{\GL}$ associated to a principal $\GL_m$-bundle over $X$, rather than $\CHilb^{k+1}(X)$ which is associated to the principal $\diff_k(1)$-bundle $\diff_X$. This way we reduce the structure group of $\CHilb^{k+1}(X)$ to $\GL_m$, but the new space has the same topological intersection numbers. We explain the details in the next subsection. 

\subsection{Linearisation of the Hilbert scheme}

Recall from Step 1 the notation $\diff_k(m)=J_k^\reg(m,m)$ for the group of $k$th order diffeomorphism germs of $\CC^m$ at the origin. In Step 1 we introduced the bundle of differential operators
\[\cald_X^k=\diff_X \times_{\diff_k(m)} \symdot.\]
where $\diff_X$ stands for the principal $\diff_k(m)$-bundle over $X$ formed by all local polynomial coordinate systems on $X$. This is not a vector bundle---the structure group is $\diff_k(m)$---but we can linearise it. 
The set $\GL(m)$ of linear coordinate changes forms a subgroup of $\diff_k(m)$. Let $\GL_X$ denote the principal $\GL(m)$-bundle over $X$ formed by all local linear coordinate systems on $X$. Then the vector bundle $\sym^{\le k} TX=\oplus_{i=1}^k \sym^i TX$ is associated to the same $\symdot$ considered as a $\GL(m)$-module:
\[\sym^{\le k} TX=\GL_X \times_{\GL(m)} \symdot.\]
Note that $\cald_X^{k}$ and  $\Sym^{\le k}TX$ are not isomorphic bundles, and in particular the filtration defined in \eqref{dfiltration} does not split. Hence there is no projection map $\cald_X^k \to TX$ but there is a natural projection $\sym^{\le k} TX \to TX$.

There is an induced $\diff_k(m)$ action on $\grass_k(\symdot)$ and the image $\CHilb^{k+1}_0(\CC^m)=\overline{\im(\phi^\grass)}$ is $\diff_k(m)$-invariant subvariety of $\grass_k(\symdot)$. Then $\CHilb^{k+1}(X)$ is the associated bundle
\[\CHilb^{k+1}(X)=\diff_X \times_{\diff_k(m)} \CHilb^{k+1}_0(\CC^m) \subset \diff_X \times_{\diff_k(m)} \grass_k(\symdot)=\grass_{k}(\cald^k_X).\]
We define the corresponding linearised bundle 
\[\CHilb^{k+1}(X)^{\GL}=\GL_X \times_{\GL_m} \CHilb^{k+1}_0(\CC^m) \subset \GL_X \times_{\GL_m}\grass_k(\symdot)=\grass_k(\sym^{\le k} TX)\]
which is the linearised version of $\CHilb^{k+1}(X)$ remembering the linear action on the fibres. 
For torus localisation purposes we can replace $\CHilb^{k+1}(X)$ with its linearised version $\CHilb^{k+1}(X)^{\GL}$.   

Let $J^\nondeg_k X= \GL_X \times_{\GL_m} J_k^\nondeg(1,m)$. We can form the bundle 
\[\calx^{\nondeg}_k=\GL_X \times_{\GL(m)} (\phi(\jetnondeg 1m)) \subset \GL_X \times_{\GL(m)} \grass_k(\symdot)=\grass_k(\sym^{\le k} TX)\]
which is a dense subbundle of $\CHilb^{k+1}(X)^\GL$ but not a subbundle of $\CHilb^{k+1}(X)$.
The projection $\sym^{\le k}TX \to TX$ then induces the following diagram whose restriction to the fibres over $X$ was given in \eqref{proj2}:
\begin{equation}\label{proj3}
\xymatrix{J_k^\nondeg X/\diff_k(1) \ar[r]^-{\phi^\grass} \ar[rd]^\pi & \grass_k(\sym^{\le k}TX) \ar@{-->}[d]\\
& \flag_k(TX) } 
\end{equation}
The image of $\phi^\grass$ sits in the domain of the vertical rational map and therefore we have a fibration
\[\mu:\calx^\nondeg_k \to \flag_k(TX)\]
of the bundles. 
\begin{definition}\label{def:widetilde} Let $\widehat{\CHilb}^{k+1}(X) \to \flag_k(TX)$ denote the fibrewise compactification of the bundle $\pi: \calx^{\nondeg}_k \to \flag_k(TX)$. In other words, if $P_{m,k} \subset \GL_m$ denotes the parabolic subgroup which preserves the flag 
\[\mathbf{f}=(\mathrm{Span}(e_1)   \subset \mathrm{Span}(e_1,e_2) \subset \ldots \subset \mathrm{Span}(e_1,\ldots, e_k) \subset \CC^m).\] 
 and $\mathfrak{p}_{m,k}=\tilde{\phi}(e_1,\ldots ,e_k)$ denotes the base point in $\grass_k(\symdot)$ then 
\[\widehat{\CHilb}^{k+1}(X)=\GL_X \times_{\GL(m)} \left(\GL(n) \times_{P_{m,k}} \overline{P_{m,k}\cdot \mathfrak{p}_{m,k}}\right)\] 
 and we have a partial resolution map
\[\rho: \widehat{\CHilb}^{k+1}(X)=\GL_X \times_{\GL(m)} \left(\GL(m) \times_{P_{m,k}} \overline{P_{m,k}\cdot \mathfrak{p}_{m,k}}\right) \to \GL_X \times_{\GL(m)} \left(\overline{\GL(m)\cdot \mathfrak{p}_{m,k}}\right)=\CHilb^{k+1}(X)^{\GL}.\]
\end{definition}

\subsection{Neighborhood of the $\CHilb^{k+1}(X)$ in $\GHilb^{k+1}(X)$} We start with $X=\CC^m$, and how to construct the model of a small neighborhood in $\GHilb^{k+1}(\CC^m)$ of the curvilinear locus $\CHilb^{k+1}(\CC^m)$ supported at some point on $\CC^m$. 
\begin{definition} Let $\xi=\xi_1 \sqcup \xi_2 \sqcup \ldots \sqcup \xi_s \in \GHilb^{k+1}(\CC^m)$ be a subscheme whose support consists of the $s$ points $p_1,\ldots, p_s \in \CC^m$ with $\mathrm{length}(\xi_i)=l_i$. The baricenter of $\xi$ is 
\[\bbb(\xi)=l_1p_1+\ldots +l_sp_s \in \CC^m\]
The balanced Hilbert scheme centered/balanced at $p\in \CC^m$ is 
\[\BHilb^{k+1}_p(\CC^m)=\bbb^{-1}(p) \subset \GHilb^{k+1}(\CC^m)\]
consists of all subschemes $\xi$ whose baricenter is $p$.
\end{definition}
Let $\{0\} \in U \subset \CC^m$ be a small open neighborhood of the origin and let 
\[\BHilb^{k+1}(U)=\{\xi \in \BHilb^{k+1}_0(\CC^m): \mathrm{supp}(\xi) \subset U\}\]
be a small neighborhood of $\CHilb^{k+1}_0(\CC^m)$ in $\BHilb^{k+1}_0(\CC^m)$. The fiberwise embedding gives 
\begin{equation}\label{diagramone}
\xymatrix{
\widetilde{\CHilb}^{k+1}_0(\CC^m)  \ar@{^{(}->}[r] \ar[d]^\rho & \BHilb^{k+1}_0(U)  \ar@{^{(}->}[r]^-{\iota} & \grass_{k+1}(S^k\CC^m)\\
\flag_k(\CC^m)  & & } 
\end{equation}
We construct a blow-up $\widetilde{\BHilb}^{k+1}_0(U)$ of $\BHilb^{k+1}_0(U)$ which fibers above the flag $\flag_k(\CC^m)$. We follow a similar path to the construction of $\widetilde{\CHilb}^{k+1}_0(\CC^m)=\GL_m \times_{P_{m,k}} \overline{P_{m,k} \cdot p_k}$. We use the following simple observation.
\begin{lemma}\label{keylemma}
Every point $\xi \in \BHilb^{k+1}_0(\CC^m)$ is contained in a $k$-dimensional linear subspace of $\CC^m$.
\end{lemma}
\begin{proof} Any subscheme $\xi \in \BHilb^{k+1}_0(\CC^m)$ is the limit $\xi=\lim_{l \to \infty} \xi^l$ of reduced subschemes 
\[\xi^l=p_1^l \sqcup \ldots \sqcup p^l_{k+1}\]
with $p_1+\ldots +p_{k+1}=0$. Hence $V^l=\mathrm{Span}(p_1,\ldots, p_{k+1})$ is a $k$-dimensional subspace in $\CC^m$, and $\xi \in  \lim_{l\to \infty} V^l$.


\end{proof}

Now let 
\[\mathbf{f}=(\mathrm{Span}(e_1)   \subset \mathrm{Span}(e_1,e_2) \subset \ldots \subset \mathrm{Span}(e_1,\ldots, e_k)=\CC_{[k]} \subset \CC^m).\] 
be our base flag as before and consider 
\[\CHilb^{k+1}_\ff(\CC^m)=\rho^{-1}(\ff)=\overline{P_{k,n}\cdot \mathfrak{p}_{k,n}}  \subset \grass_k(\sym^{\le k}\CC_{[k]})\]
Let $\BHilb^{k+1}_0(U \cap \CC_{[k]})$
be the space of those subschemes in the neighborhood which sit in the $k$-dimensional base space $\CC_{[k]}$. We choose $U$ to be $\GL(m)$-invariant then $U \cap \CC_{[k]}$, and hence $\BHilb^{k+1}(U \cap \CC_{[k]})$ is $P_{m,k}$-invariant and we can define 
\[\widetilde{\BHilb}^{k+1}_0(\CC^m)=\GL(m) \times_{P_{m,k}} \BHilb^{k+1}(U \cap \CC_{[k]})\]
which fibers over $\flag_k(\CC^m)=\GL(m) / {P_{m,k}}$. So we arrive at the diagram
\begin{equation}\label{diagramtwo}
\xymatrix{
\widetilde{\CHilb}^{k+1}_0(\CC^m) \ar@{^{(}->}[r] \ar[d]^\rho &  \widetilde{\BHilb}^{k+1}_0(U) \ar[ld] \ar@{^{(}->}[r]^-{\iota} & \grass_{k+1}(S^k\CC^m)\\
\flag_k(\CC^m)  & } 
\end{equation}
and by Lemma \ref{keylemma} $\rho: \widetilde{\BHilb}^{k+1}_0(U) \to \BHilb^{k+1}_0(U)$ is surjective rational map (generically 1-1). 

Next, we apply this construction fiberwise over $\CC^m$: the origin can be replaced by any point $p\in \CC^m$ and $U$ by a $\GL(m)$-invariant open neighborhood $U_p \subset \CC^m$ of $p$ such that 
\[\cup_{p\in \CC^m} \BHilb^{k+1}_p(U_p)=\GHilb^{k+1}(U)\]
where $U=\cup_{p\in \CC^m}U_p$ is a tubular neighborhood of the zero section of $T\CC^m$. We get 
\begin{equation}\label{diagramthree}
\xymatrix{
\widehat{\CHilb}^{k+1}(\CC^m) \ar@{^{(}->}[r] \ar[d]^\rho &  \widehat{\GHilb}^{k+1}(U) \ar[ld] \ar@{^{(}->}[r]^-{\iota} & \grass_{k+1}(S^k\CC^m)\\
\flag_k(T\CC^m) \ar[d]^\mu & & \\
\CC^m &&} 
\end{equation}

Finally, the same argument works over a smooth manifold $X$: we let $\bU \subset TX$ be a small $\GL_X$-invariant tubular neighborhood of the zero section, with the exponential map $\exp: \bU \to X$, which identifies $U_x$ with $\exp(\bU_x) \subset X$. We may assume that $\bU=\GL_X \times_{\GL(m)} U$ for some $U \subset \CC^m$. We define 
\[\widehat{\GHilb}^{k+1}(\bU)=\GL_X \times_{\GL(m)} \widetilde{\BHilb}^{k+1}_0(\CC^m)\]
We get the diagram
\begin{equation}\label{diagramfour}
\xymatrix{
\widehat{\CHilb}^{k+1}(X) \ar@{^{(}->}[r] \ar[d]^\rho &  \widehat{\GHilb}^{k+1}(\bU) \ar[ld] \ar@{^{(}->}[r]^-{\iota} & \grass_{k+1}(S^kTX)\\
\flag_k(TX) \ar[d]^\mu & &\\
X &&} 
\end{equation}
and $\widehat{\GHilb}^{k+1}(\bU) \to \GHilb^{k+1}(\bU)$ is a birational morphism. We obtain the following relation among the tautological bundles
 \[\calo_X^{[k+1]}=\iota^* \cale^{k+1}\]
 and for a bundle $V$ on $X$ the restriction of the tautological bundle to the curvilinear locus is 
 \begin{equation}\label{taurestrictedtocurvi}
 V^{[k+1]}|_{\widehat{\CHilb}^{k+1}(X)}=V \otimes \calo_X^{[k+1]}
 \end{equation}
 

\section{Fully nested Hilbert schemes}\label{sec:nested}

The central object in our proof of the main theorems is a generalisation of the standard nested Hilbert schemes of points. Recall the space  
\[\Hilb^{[1,2,\ldots, k]}(X)=\{(\xi_1 \subset \xi_2 \subset \ldots \subset \xi_k)| \xi \subset X, \dim H^0(\xi_i)=i\} \subset X \times \Hilb^2(X) \times \ldots \times \Hilb^k(X)\] 
which is formed by flags of subschemes is called the nested Hilbert scheme on $X$. Our fully nested Hilbert scheme defined in this section is associated to the ordered Hilbert scheme of $k$ labeled points, and it keeps track degenerations of any subset of the $k$ labeled points.
\subsection{Approximating bundles}
The ordered Hilbert scheme of points $\Hilb^{[k]}(X)$ is a branched cover of the ordinary Hilbert scheme, defined by the commutative diagram
\begin{equation*}
\xymatrix{\Hilb^{[k]}(X) \ar[r] \ar[d]^{\HC} &  \Hilb^k(X) \ar[d]^{\HC} \\
 X^k \ar[r]  & \Sym^k(X) } 
\end{equation*}
where the vertical arrows are the Hilbert-Chow morphisms taking a subscheme/ordered subscheme $Z$ to its support cycle. This is a ramified cover of $\Hilb^k(X)$ with a stratified ramification locus sitting over the diagonals of $\Sym^k(X)$. Define the bundle $F^{[k]}$ on $\Hilb^{[k]}(X)$ as the pullback of $F^{[k]}$ along $\Hilb^{[k]}(X) \to \Hilb^k(X)$. Note that we loosely use the same notation for $F^{[k]}$ on both $\Hilb^k(X)$ and $\Hilb^{[k]}(X)$.

Following Li \cite{junli} and Rennemo \cite{rennemo}, we define for every partition $\alpha \in \Pi(k)$ of $\{1,\ldots, k\}$ the scheme $\Hilb^{[\alpha]}(X)$ which is a certain approximation of  $\Hilb^{[k]}(X)$. We fix some notation and conventions first. For simplicity, let $\RHilb^{[k]}(X) \subset \GHilb^{[k]}(X)$ denote the open set of reduced subschemes of the form $\xi=x_1 \sqcup \ldots \sqcup x_k$, formed by $k$ different points on $X$.

\begin{definition}\label{def:approximatingsets} Let $\a=(\a_1,\ldots, \a_s) \in \Pi(k)$ be a partition of $\{1,\ldots, k\}$. 
\begin{enumerate}
\item We let $\sim_\alpha$ be the equivalence relation on $\{1,\ldots, k\}$ given by letting the
elements of $\alpha$ form the equivalence classes.
We introduce the partial order by setting $\alpha \le \beta$ if $\sim_\a$ is a refinement of $\sim_\beta$.
Let 
\[\Delta_\alpha=\{(x_1,\ldots, x_k) \in X^k | x_i=x_j \text{ if } i\sim_\a j\}.\]
denote the closed diagonal, then $\Delta_\b \subseteq \Delta_\a$ whenever $\a \le \b$.
\item Let 
\[\Hilb^{[\a]}(X)=\prod_{i=1}^s \Hilb^{[\a_i]}(X),\]
where for a subset $A \subset \{1,\ldots, k\}$, $\Hilb^{[A]}(X)$ denotes the ordered Hilbert scheme of $|A|$ points labeled by $A$. The punctual part sits over the corresponding diagonal:
\[\Hilb^{[\a]}_0(X)=\prod_{i=1}^s \Hilb^{[\a_i]}_0(X)=\HC^{-1}(\Delta_\a).\]
\item We define
\[\GHilb^{[\a]}(X) = \prod_{i=1}^s \GHilb^{[\a_i]}(X)\]
which is the the closure of $\RHilb^{[k]}(X)$ in $\Hilb^{[\a]}(X)$. Like above, the punctual part is $\GHilb^{[\a]}_0(X)=\HC^{-1}(\Delta_\a)$. 
\end{enumerate}
\end{definition}
The fully nested Hilbert scheme $N^k(X)$ parametrises ordered collections of $k$ points in $X$, with the additional data that when $l$ points with labels in the same set in the partition $\a$ come together at $x$, one must specify a length $l$ subscheme supported at $x$. The master blow-up space encodes information on how different subsets of points collide. 

\begin{definition}  Let $X$ be a complex nonsingular manifold and $k\ge 1$. \begin{enumerate} 
\item The fully nested Hilbert scheme is $N^k(X)=\overline{im(h)}$, the closure of the image of the natural map
\[h: \RHilb^{[k]}(X) \to \prod_{\a \in \Pi(k)} \GHilb^{[\a]}(X).\]
The Hilbert-Chow morphism extends to $N^k(X)$ and gives a morphism $\HC: N^k(X) \to X^k$, hence obtain the $\a$-punctual locus $N^\a_0(X)=\HC^{-1}(\Delta_\a)$. The projection $\pi_\a:N^k(X) \to \GHilb^{[\a]}(X)$ fits into the diagram
\begin{equation}\label{commdiagram}
\xymatrix{N^k(X) \ar[r]^-{\pi_\a} \ar[d]^{\HC} & \GHilb^{[\a]}(X) \ar[dl]^{\HC_\a} \\
X^k 
}
\end{equation}
\item Let $\a=(\a_1,\ldots, \a_s) \in \Pi(k)$ be a partition. Pull-back along the projection map $\pi_\a: N^k(X)  \to \GHilb^{[\a]}(X)$ defines an approximating bundle $\pi_\a^*(F^{[\a]})$ over $N^k(X)$, which satisfies 
 \[\pi_\a^*F^{[\a]}=\pi_{\a_1}^*F^{[[\a_1]]} \oplus \ldots \oplus \pi_{\a_s}^*F^{[[\a_s]]}.\]
We will loosely use the shorthand notation $F^{[\a]}$ for the bundle $\pi_\a^*(F^{[\a]})$. 
\item The punctual part, which sits over the $\a$-diagonal, fits into the diagram 
\begin{equation}\label{commdiagram2}
\xymatrix{N^\a_0(X) \ar[r]^{\pi_\a} \ar[d]^{\HC} & \GHilb^{[\a]}_0(X) \ar[dl]^{\HC_\a} \\
\Delta_\a 
}
\end{equation}
and  $F^{[\a]}_0=\pi_\a^*(F^{[\a]}|_{\GHilb^{[\a]}_0(X)})$ its restriction to the punctual part $N^\a_0(X)$. 
\end{enumerate}
\end{definition}

The terminology is self-explanatory: $N^k(X)$ can be considered as a nested Hilbert scheme, but nested with respect to the full partially ordered net of subsets of $\{1,\ldots, k\}$. Note that $N^k(X)$ is irreducible, being the closure of the image under $h$ of an irreducible variety. It is a blow-up of $\GHilb^k(X)$ via the natural projection map 
\[\pi_\L: N^k(X) \to \GHilb^k(X)\] 
where $\L=\{1,\ldots, k\}$ is the trivial partition. However, the geometry of the restriction $\pi_\L: N^k_0(X) \to \GHilb^k_0(X)$
to the punctual part is more subtle. In particular, the preimage of the curvilinear component $\CHilb^k(X)$ is not necessarily irreducible which is a delicate part of our integration argument, addressed in the next section.  

\begin{definition}\label{def:curvgeometricsubsetQ} The curvilinear part of $N^k(X)$ is $CN^k(X)=\pi_\L^{-1}(\CHilb^{[k]}(X))$ where $\L=\{1,\ldots, k\}$ is the trivial partition. The punctual curvilinear component supported at $p\in X$ is 
\[CN^k_p(X)=\pi_\L^{-1}(\CHilb^{[k]}_p(X))=CN^k(X) \cap \HC^{-1}(\{p,\ldots, p\}).\]
\end{definition}

\begin{example} We have seen in \S \ref{subsec:k=3} that $CN^3_0(\CC^2)$ is not irreducible, and  it has $2$ components: the curvilinear component $CN^{main}_0(\CC^2)$ of dimension $2$, and an other $\PP^1$ sitting over the boundary point $I=\mathfrak{m}^2 \in \CHilb^3_0(\CC^2)$. 
\end{example}

Despite these extra components, the projection $\pi^\L$ is isomorphism over the curvilinear locus in $\CHilb^{k}(X)$. Recall this is defined as
\[\mathrm{Curv}^k(X)=\{\xi \in \Hilb^k_0(X): \calo_\xi \simeq \CC[t]/t^k\},\]
and it is a dense open subset of $\CHilb^k(X)$. 


\subsection{Fibration of the fully nested Hilbert scheme over the flag manifold}

Recall the curvilinear part of $N^{k+1}(X)$, defined as $CN^{k+1}(X)=\pi_\L^{-1}(\CHilb^{k+1}(X))$ where $\L=\{0,1,\ldots, k\}$ is the trivial partition. The punctual curvilinear component supported at $p\in X$ is 
\[CN^{k+1}_p(X)=\pi_\L^{-1}(\CHilb^{k+1}_p(X))=CN^{k+1}(X) \cap \HC^{-1}(\{p,\ldots, p\}).\]

We follow the same argument which resulted in diagram \eqref{diagramfour}. Let $\bU \subset TX$ be a small $\GL_X$-invariant tubular neighborhood of the zero section, with the exponential map $\exp: \bU \to X$, which identifies $U_x$ with $\exp(\bU_x) \subset X$. We may assume that $\bU=\GL_X \times_{\GL(m)} U$ for some $U \subset \CC^m$. We define 
\[\widehat{N}^{k+1}(\bU)=\pi_\Lambda^{-1}(\widehat{\GHilb}^{k+1}(\bU)).\]
We get the following extension of diagram \eqref{diagramfour}:
\begin{equation}\label{diagramfive}
\xymatrix{
\widehat{CN}^{k+1}(X) \ar@{^{(}->}[r]  \ar[d]^{\pi_\Lambda} & \widehat{N}^{k+1}(\bU) \ar[d]^{\pi_\Lambda}\\ 
\widehat{\CHilb}^{k+1}(X) \ar@{^{(}->}[r] \ar[d]^\rho &  \widehat{\GHilb}^{k+1}(\bU) \ar[ld] \\
\flag_k(TX) \ar[d]^\mu & \\
X &} 
\end{equation}
and $\widehat{N}^{k+1}(\bU) \to \widehat{\GHilb}^{k+1}(\bU)$ is a birational morphism.

When $X=\CC^m$, all vertical maps in diagram \eqref{diagramfive} are $\GL(n)$-equivariant, and we can take $\bU=T\CC^m$, and an equivariant nonsingular resolution of $\widehat{CN}^{k+1}(X)$ and $\widehat{N}^{k+1}(T\CC^m)$ fitting into our complete diagram
\begin{equation}\label{diagramsix}
\xymatrix{\mathbf{CN}^{k+1}(\CC^m) \ar@{^{(}->}[r]  \ar[d] & \mathbf{N}^{k+1}(T\CC^m) \ar[d] \\ 
\widehat{CN}^{k+1}(\CC^m) \ar@{^{(}->}[r]  \ar[d]^{\pi_\Lambda} & \widehat{N}^{k+1}(T\CC^m) \ar[d]^{\pi_\Lambda}  \\ 
\widehat{\CHilb}^{k+1}(\CC^m) \ar@{^{(}->}[r] \ar[d]^\rho &  \widehat{\GHilb}^{k+1}(T\CC^m) \ar[ld]   \\
\flag_k(T\CC^m) \ar[d]^\mu &  \\
\CC^m &  } 
\end{equation}
where $\mathbf{N}^{k+1}(T\CC^m)$ is a nonsingular $\GL(m)$-equvivariant blow-up corresponding to $\bU=T\CC^m$. We can pull-back all approximating bundles $F^{[\a]}$ to $\widehat{N}^{k+1}(\bU)$ and $\mathbf{N}^{k+1}(\bU)$, and we loosely use the same notation for these pull-back bundles.

\subsection{Embedding into configuration spaces}\label{subsec:configuration}  $CN^{k+1}(X)$ and a small neighborhood of it has a natural embedding into a Bott-Samelson type configuration space as follows. 
\begin{definition} \begin{enumerate} 
\item Let $V$ be a complex vector space of dimension $k$. Its configuration space is 
\[\mathbf{V}=\{(V_A|A \subset \{1,\ldots ,k\}): V_A \subset V \text{ linear }, dim(V_A)=|A|, A \subset B \Rightarrow V_A \subset V_B\}\]
which is smooth an irreducible variety. 
\item Let $\pmb{\mathscr{E}}$ be the configuration bundle of the tautological bundle $\cale$ over $\grass_k(W)$. 
\item The shifted configuration space of $V$ is 
\[\mathbf{V}^{-1}=\{(V_A|A \subset \{1,\ldots ,k\}): V_A \subset V \text{ linear},dim(V_A)=|A|-1, A \subset B \Rightarrow V_A \subset V_B\}\]
which is, again, a smooth and irreducible variety. 

\end{enumerate}
\end{definition}

Then $CN^{k+1}(X)$ sits naturally in the shifted configuration bundle $\pmb{\mathscr{E}}^{-1}$ over $\grass_k(\sym^{\le k}TX) \subset \grass_{k+1}(\calo_X\oplus \ldots \oplus \sym^k TX)$, compatible with the fibration:
\[\xymatrix{CN^{k+1}(X) \ar@{^{(}->}[r]^{\bi} \ar[d] & \pmb{\mathscr{E}}^{-1} \ar[d] \\
\CHilb^{[k+1]}(X) \ar@{^{(}->}[r]^{\phi} & \grass_k(\sym^{\le k}TX)
}\]

Here $\bi$ is defined as follows. A point $\xi \in CN^{k+1}(X)$ is determined by the vector $(\xi_A:A \subset \{0,1,...,k\})$ where $\xi_A \subset \xi$ is the subscheme defined by points in $A$. Then 
\[\bi(\xi_A:A\subset \{0,\ldots, k\})=(\phi^\grass(\xi_A):A\subset \{0,\ldots, k\})\]
We need the shifted configuration bundle here because $\dim(\phi^\grass(\xi_A))=|A|-1$.

The embedding into the configuration bundle happens fiberwise over the curvilinear Hilbert scheme, and hence the construction of a neighborhood of $CN^{k+1}(X)$ in $N^{k+1}(X)$ can be extended. First, \eqref{diagramone} has the following extension for a tubular neighborhood $U$ of the zero section in $TX$
\begin{equation}\label{diagram8}
\xymatrix{& N^{k+1}(\bU) \ar@{^{(}->}[dr]  &\\
CN^{k+1}(X) \ar@{^{(}->}[r]^{\bi} \ar@{^{(}->}[ur] \ar[d]^{\pi_\Lambda} & \pmb{\mathscr{E}}^{-1} \ar[d] \ar@{^{(}->}[r]& \pmb{\mathscr{E}} \ar[d]\\
\CHilb^{k+1}(X)   \ar[r]^-{\phi^\grass} &   \grass_k(\sym^{\le k}TX) \ar@{^{(}->}[r]^-{\iota} & \grass_{k+1}(\calo_X \oplus \ldots \oplus \sym^kTX)} 
\end{equation}
Finally, the embedding into the configuration space can be constructed fiberwise over the flag $\flag_k(TX)$.



\section{Equivariant localisation and multidegrees}\label{sec:equiv}

This section is a brief of equivariant cohomology and localisation. For
more details, we refer the reader to Berline--Getzler--Vergne \cite{bgv} and B\'erczi--Szenes \cite{bsz}. 

Let $\kt\cong U(1)^m$ be the maximal compact subgroup of
$T\cong(\CC^*)^m$, and denote by $\mathfrak{t}$ the Lie algebra of $\kt$.  
Identifying $T$ with the group $\CC^m$, we obtain a canonical basis of the weights of $T$:
$\lambda_1,\ldots ,\lambda_m\in\mathfrak{t}^*$. 

For a manifold $M$ endowed with the action of $\kt$, one can define a
differential $d_\kt$ on the space $S^\bullet \mathfrak{t}^*\otimes
\Omega^\bullet(M)^\kt$ of polynomial functions on $\mathfrak{t}$ with values
in $\kt$-invariant differential forms by the formula:
\[   
[d_\kt\alpha](X) = d(\alpha(X))-\iota(X_M)[\alpha(X)],
\]
where $X\in\mathfrak{t}$, and $\iota(X_M)$ is contraction by the corresponding
vector field on $M$. A homogeneous polynomial of degree $d$ with
values in $r$-forms is placed in degree $2d+r$, and then $d_\kt$ is an
operator of degree 1.  The cohomology of this complex--the so-called equivariant de Rham complex, denoted by $H^\bullet_T(M)$, is called the $T$-equivariant cohomology of $M$. Elements of $H_T^\bullet (M)$ are therefore polynomial functions $\mathfrak{t} \to \Omega^\bullet(M)^K$ and there is an integration (or push-forward map) $\int: H_T^\bullet(M) \to H_T^\bullet(\mathrm{point})=S^\bullet \mathfrak{t}^*$ defined as  
\[(\int_M \alpha)(X)=\int_M \alpha^{[\mathrm{dim}(M)]}(X) \text{ for all } X\in \mathfrak{t}\]
where $\alpha^{[\mathrm{dim}(M)]}$ is the differential-form-top-degree part of $\alpha$. The following proposition is the Atiyah-Bott-Berline-Vergne localisation theorem in the form of \cite{bgv}, Theorem 7.11. 
\begin{theorem}[(Atiyah-Bott \cite{atiyahbott}, Berline-Vergne \cite{berlinevergne})]\label{abbv} Suppose that $M$ is a compact complex manifold and $T$ is a complex torus acting smoothly on $M$, and the fixed point set $M^T$ of the $T$-action on M is finite. Then for any cohomology class $\a \in H_T^\bullet(M)$
\[\int_M \alpha=\sum_{f\in M^T}\frac{\a^{[0]}(f)}{\mathrm{Euler}^T(T_fM)}.\]
Here $\mathrm{Euler}^T(T_fM)$ is the $T$-equivariant Euler class of the tangent space $T_fM$, and $\alpha^{[0]}$ is the differential-form-degree-0 part of $\alpha$. 
\end{theorem}

The right hand side in the localisation formula considered in the fraction field of the polynomial ring of $H_T^\bullet (\mathrm{point})=H^\bullet(BT)=S^\bullet \mathfrak{t}^*$ (see more on details in Atiyah--Bott \cite{atiyahbott} and \cite{bgv}). Part of the statement is that the denominators cancel when the sum is simplified.

\subsection{Equivariant Poincar\'e duals and multidegrees}
\label{subsec:epdmult} 

Restricting the equivariant de Rham complex to compactly supported (or quickly
decreasing at infinity) differential forms, one obtains the compactly
supported equivariant cohomology groups $ H^\bullet_{\kt,\mathrm{cpt}}(M)
$. Clearly $H^\bullet_{\kt,\mathrm{cpt}}(M) $ is a module over
$H^\bullet_\kt(M)$. For the case when $M=W$ is an $N$-dimensional
complex vector space, and the action is linear, one has
$H^\bullet_\kt(W)= S^\bullet\mathfrak{t}^*$ and $ H^\bullet_{\kt,\mathrm{cpt}}(W) $ is
a free module over $H^\bullet_\kt(W)$ generated by a single element of
degree $2N$:
\begin{equation}
  \label{thomg}
   H^\bullet_{\kt,\mathrm{cpt}}(W) = H^\bullet_{\kt}(W)\cdot\mathrm{Thom}_{\kt}(W)
\end{equation}

Fixing coordinates $y_1,\dots,y_N$ on $W$, in which the $T$-action is
diagonal with weights $\eta_1,\ldots,  \eta_N$, one can write an explicit
representative of  $\mathrm{Thom}_{\kt}(W)$ as follows:
\[   \mathrm{Thom}_{\kt}(W) = 
e^{-\sum_{i=1}^N|y_i|^2}\sum_{\sigma\subset\{1,\ldots , N\}}
\prod_{i\in\sigma}\eta_i/2\cdot\prod_{i\notin \sigma}dy_i\,d\bar y_i
\]

We will say that an algebraic variety has dimension $d$ if its
maximal-dimensional irreducible components are of dimension $d$.  A
$T$-invariant algebraic subvariety $\Sigma$ of dimension $d$ in $W$
represents $\kt$-equivariant $2d$-cycle in the sense that
\begin{itemize}
\item a compactly-supported equivariant form $\mu$ of degree $2d$ is
  absolutely integrable over the components of maximal dimension of
  $\Sigma$, and $\int_\Sigma\mu\in S^\bullet \mathfrak{t}$;
\item if $d_\kt\mu=0$, then $\int_\Sigma\mu$ depends only on the class
  of $\mu$ in $ H^\bullet_{\kt,\mathrm{cpt}}(W) $,
\item and $\int_\Sigma\mu=0$  if $\mu=d_\kt\nu$ for a
  compactly-supported equivariant form $\nu$.
\end{itemize}

\begin{definition} \label{defepd} Let $\Sigma$ be an $T$-invariant algebraic
  subvariety of dimension $d$ in the vector space $W$. Then the
  equivariant Poincar\'e dual of $\Sigma$ is the polynomial on $\mathfrak{t}$
  defined by the integral
\begin{equation}
 \label{vergneepd}
 \epd\Sigma = \frac1{(2\pi)^d}\int_\Sigma\mathrm{Thom}_{\kt}(W).
\end{equation}  
\end{definition}
\begin{remark}
  \begin{enumerate}
  \item An immediate consequence of the definition is that for an equivariantly
closed differential form $\mu$ with compact support, we have
\[  \int_\Sigma\mu = \int_W \epd\Sigma\cdot\mu.
\]
This formula serves as the motivation for the term {\em equivariant
  Poincar\'e dual.}
\item This definition naturally extends to the case of an analytic
  subvariety of $\CC^n$  defined in the neighborhood of the origin, or
  more generally, to any $T$-invariant cycle in $\CC^n$.
  \end{enumerate}
\end{remark}

Another terminology for the equivariant Poincar\'e dual is {\em multidegree}, which is close in spirit to the original
construction of Joseph \cite{joseph}. Let  $\Sigma \subset W$ be a $T$-invariant
subvariety. Then we have
\[       \epd{\Sigma,W}_T=\mdeg{I(\Sigma),\CC[y_1,\ldots , y_N]}.
\] 

Some basic properties of the equivariant Poincar\'e dual are listed in \cite{bsz}, these are: Positivity, Additivity, Deformation invariance, Symmetry and a formula for complete intersections. Using these properties one can easily describe an algorithm for
computing $\mdeg{I,S}$ as follows (see Miller--Sturmfels \cite[\S8.5]{milsturm}, Vergne \cite{voj} and \cite{bsz} for details). %

An ideal $M\subset S$ generated by a set of monomials in
$y_1,\ldots, y_N$ is called a \emph{monomial ideal}. Since
$\mathrm{in}_<(I)$ is such an ideal, by the deformation invariance
it is enough to compute $\mdeg{M}$ for monomial ideals $M$. If the
codimension of $\Sigma(M)$ in $W$ is $s$, then the maximal
dimensional components of $\Sigma(M)$ are codimension-$s$ coordinate
subspaces of $W$. Such subspaces are indexed by subsets
$\mathbf{i}\in\{1\ldots  N\}$ of cardinality $s$; the corresponding
associated primes are $\mathfrak{p}[\mathbf{i}]=\langle y_i:i\in
\mathbf{i} \rangle$. Then 
\begin{equation*}
\label{primemon} \mult(\mathfrak{p}[\mathbf{i}],M)=
\left|\left\{\mathbf{a}\in\ZZ_+^{[{\mathbf{i}}]};\;
\mathbf{y}^{\mathbf{a}+\mathbf{b}}\notin M\text{ for all }
\mathbf{b}\in\ZZ_+^{[\hat{\mathbf{i}}]}\right\}\right|,
\end{equation*}
where $\ZZ_+^{[\mathbf{i}]}=\{\mathbf{a}\in \ZZ_+^N;a_i=0 \text{ for }
i\notin \mathbf{i}\}$, $\hat{\mathbf{i}}=\{1\ldots 
N\}\setminus\mathbf{i}$, and $|\cdot|$, as usual, stands for the
number of elements of a finite set. By the normalization and additivity axiom we have
\begin{equation}\label{mdegformula}
\mdeg{M,S} =
\sum_{|\mathbf{i}|=s}\mult(\mathfrak{p}[\mathbf{i}],M)
\prod_{i\in\mathbf{i}}\eta_i.
\end{equation}
By definition, the weights $\eta_1,\ldots \eta_N$ on $W$ are linear forms of $\l_1,\ldots \l_r$, the basis of $(\CC^*)^r$, and we denote the coefficient of $\l_j$ in $\eta_i$ by $\coeff(\eta_i,j)$, $1\le i\le N, 1\le j \le r$, and introduce 
\[\deg(\eta_1,\ldots, \eta_N;m)=\#\{i;\;\coeff(\eta_i,m)\neq 0\}\}.\]
It is clear from the formula \eqref{mdegformula} that 
\begin{equation*}
\deg_{\l_m}\mdeg{I,S} \le \deg(\eta_1,\ldots, \eta_N;m)
\end{equation*}
holds for any $1\le m \le r$.

\subsection{The Rossman formula} \label{subsec:rossman} 

The Rossmann equivariant localisation formula is an improved version of the Atiyah-Bott/Berline-Vergne localisation for singular varieties sitting in a smooth ambient space. 
Let $Z$ be a complex manifold with a holomorphic $T$-action, and let
$M\subset Z$ be a $T$-invariant analytic subvariety with an isolated
fixed point $p\in M^T$. Then one can find local analytic coordinates
near $p$, in which the action is linear and diagonal. Using these
coordinates, one can identify a neighborhood of the origin in $\TT_pZ$
with a neighborhood of $p$ in $Z$. We denote by $\tc_pM$ the part of
$\TT_pZ$ which corresponds to $M$ under this identification;
informally, we will call $\tc_pM$ the $T$-invariant {\em tangent cone}
of $M$ at $p$. This tangent cone is not quite canonical: it depends on
the choice of coordinates; the equivariant dual of
$\Sigma=\tc_pM$ in $W=\TT_pZ$, however, does not. Rossmann named this
 the {\em equivariant multiplicity of $M$ in $Z$ at $p$}:
\begin{equation}\label{emult}
   \emu_p[M,Z] \overset{\mathrm{def}}= \epd{\tc_pM,\TT_pZ}.
\end{equation}

\begin{remark}
In the algebraic framework one might need to pass to the {\em
tangent scheme} of $M$ at $p$ (cf. Fulton \cite{fulton}). This is canonically
defined, but we will not use this notion.
\end{remark}
The analog of the Atiyah-Bott formula for singular subvarieties of smooth ambient manifolds is the following statement.
\begin{proposition}[Rossmann's localisation formula \cite{rossmann}]\label{rossman} Let $\mu \in H_T^*(Z)$ be an equivariant class represented by a holomorphic equivariant map $\mathfrak{t} \to\Omega^\bullet(Z)$. Then 
\begin{equation}
  \label{rossform}
  \int_M\mu=\sum_{p\in M^T}\frac{\emu_p[M,Z]}{\mathrm{Euler}^T(\TT_pZ)}\cdot\mu^{[0]}(p),
\end{equation}
where $\mu^{[0]}(p)$ is the differential-form-degree-zero component
of $\mu$ evaluated at $p$.  
\end{proposition}

\section{Proof of residue formula for tautological integrals}\label{sec:prooftauintegrals}

In this section we prove Theorem \ref{tauintegral} and it equivariant counterpart, Theorem \ref{equivariantintegral}.

\subsection{Sieve on fully nested Hilbert schemes}\label{subsec:sieve}
The sieve argument presented in this subsection works in a more general setting. Let $V$ be a rank $r$ vector bundle over $X$ and $\Phi(c_1,\ldots ,c_{(k+1)r})$ be a Chern polynomial in the Chern roots of the tautological bundle $V^{[k+1]}$ over $\Hilb^{k+1}(X)$. Since
\begin{equation}\label{n!}
n! \int_{\GHilb^{k+1}(X)}\Phi = \int_{\GHilb^{[k+1]}(X)} \Phi
\end{equation}
one can work over ordered Hilbert schemes, and we keep the notation $\Phi$ for the pulled-back form. 

The first step in our strategy is to pull-pack integration to the fully nested Hilbert scheme $NX^{k+1}(X)$, which admits plenty of approximating bundles (coming from the different factors) to build linear combinations with punctual supports. A key observation of \cite{junli,rennemo} is that these tautological bundles can be combined into a sieve formula,  which decomposes $\Phi$ as a sum 
\[\Phi=\sum_{\a \in \Pi(k+1)}\Phi^\a\]
of forms indexed by partitions of $\{1, \ldots k+1\}$. For any partition $\{1,\ldots, k+1\}=\a_1 \sqcup \ldots \sqcup \a_s$ the form $\Phi^\a$ is supported on the approximating punctual subset  
\[\supp(\Phi^\a)=N^{\a}_0(X)=\HC^{-1}(\Delta_{\a})\]
and $\Phi^\a$ is a linear combination of the classes $\Phi(V^{[\b]})$ for those partitions $\b \in \Pi(k+1)$ which are refinements of $\a$, that is, $\b \le \a$ holds.  

\begin{definition} Let $\a \in \Pi(k+1)$ be a partition and $\Phi$ a homogeneous symmetric polynomial in the Chern roots of $V^{[k+1]}$. Define the class $\Phi^\a \in H^*(N^{k+1}(X))$ inductively by putting $\Phi^{[0],\ldots [k]}=\Phi(V^{[0],\ldots [k]})=\Phi(V\oplus \ldots \oplus V)$ and for $\a>([0],\ldots [k])$  
\[\Phi^\a=\Phi(V^{[\a]})-\sum_{\b < \a}\Phi^\b.\]
\end{definition}

\begin{remark}
Let $\Lambda=[0,\ldots, k]$ be the trivial partition. The formula above gives us 
\[\Phi^{\Lambda}(V^{[k+1]})=\sum_{\b \in \Pi(k+1)}(-1)^{|\b|-1}(|\b|-1)!\Phi(V^\b).\]
\end{remark}

\begin{proposition}\label{thm:support} The restriction of $\Phi^\a$ to $N^{k+1}(X) \setminus N^\a(X)$ vanishes, that is, the support of $\Phi^\a$ is $N^\a(X)=\HC^{-1}(\Delta_{\a})\subset N^{k+1}(X)$.
In particular, the support of $\Phi^\Lambda$ is $N^{k+1}_0(X)=\pi_\Lambda^{-1}(\GHilb^k_0(X))$, the punctual part of $N^{k+1}(X)$ where all components are punctual subschemes supported at some point of $X$.  
\end{proposition}

\begin{proof} This follows from easy inclusion-exclusion and induction, for the proof see \cite{rennemo}.
\end{proof}

\begin{lemma}\label{lemma:support} Let  $\Phi$ be a properly supported Chern polynomial in the sense of Definition \ref{geocond}. Then 
\begin{enumerate}
\item  $\pi_\Lambda^*\Phi$ is also properly supported, that is, it is represented by a form whose support intersects the punctual part $N^{k+1}_0(X)$ only in the curvilinear part $CN^{k+1}(X)=\pi_\L^{-1}(\CHilb^{k+1}(X))$. 
\item More generally, for $\a=(\a_1,\ldots, \a_s) \in \Pi(k+1)$ the form $\Phi^\a$ is represented by a form whose support intersects the punctual part $N^{k+1}_0(X)$ only in the curvilinear part
\[CN^\a(X)=\pi_\a^{-1}(\CHilb^{[\a_1]}(X)\times \ldots \times \CHilb^{[\a_s]}(X))\]
\end{enumerate}
\end{lemma}

\begin{proof}
If $\omega$ is a properly supported form representing $\Phi$, then $\pi^*\omega$ represents $\pi_\Lambda^*\Phi$, and the statement follows from the fact that support of the pull-back form under a proper map is equal to the pre-image of the support.  
\end{proof}

The first step in proving Theorem \ref{tauintegral} is to pull back the integral over the fully nested Hilbert scheme, and apply the sieve formula:

\begin{equation}\label{motivic1}
\int_{\GHilb^{[k+1]}(X)}\Phi=\int_{N^{k+1}(X)}\pi^* \Phi=\sum_{\a \in \Pi(k)} \int_{N^{k+1}(X)} \Phi^\a
\end{equation}

\subsection{Reduction to equivariant integration for $X=\CC^n$}\label{subsec:reduction}

The sieve formula reduces integration of forms supported on $\GHilb^{k+1}(X)$ to small neighborhoods of the punctual Hilbert scheme sitting over various diagonals under the Hilbert-Chow morphism.    
Take a small neighborhood $\widehat{N}^{k+1}(\bU)$ of $\widehat{CN}^{k+1}(X)$ in $N^{k+1}(X)$, constructed in the previous section, which fibers over $\flag_k(TX)$ and hence fibers over $X$. This bundle can be pulled back along a classifying map $\tau: X \to B\GL(m)$ from the universal bundle 
\[\mathbf{E}=E\GL(m) \times_{\GL(m)} BN^{k+1}(\CC^m)\]
where with the notations of diagram \eqref{diagramsix} 
\[BN^{k+1}(\CC^m)=(\mu \circ \rho \circ \pi_\Lambda)^{-1}(0)=\pi_\Lambda^{-1}(\BHilb^{k+1}_0(\CC^m))\]
 is the preimage in $\widehat{N}^{k+1}(\CC^m)$ of the balanced Hilbert scheme supported at the origin of $\CC^m$. We get a commutative diagram 
\begin{equation*}
\xymatrix{\widehat{CN}^{k+1}(X) \ar[rd]^\pi \ar@{^{(}->}[r] & \widehat{N}^{k+1}(\bU) \ar[r] \ar[d]^{\pi} & \mathbf{E}  \ar[d] \\
 & X \ar[r]^{\tau}  &  B\GL(m) } 
\end{equation*}
which induces a diagram of cohomology maps
\begin{equation*}
\xymatrix{H^*(\widehat{N}^{k+1}(\bU))  \ar[d]^{\pi_*} &  H^*(\mathbf{E}) \ar[l] \ar[d]^{\res} \\
 H^*(X)   &  H^*(B\GL(m)) \ar[l]^{\mathrm{Sub}}} 
\end{equation*}
Here 
\begin{itemize}
\item $\res$ is the equivariant push-forward (integration) map along the fiber $BN^{k+1}(\CC^m)$, and in the next section we develop an iterated residue formula derived from equivariant localisation.
\item $\mathrm{Sub}$ is the Chern-Weil map, which is the substitution of the Chern roots of $X$ into the generators $\lambda_1,\ldots, \lambda_m$ of $H^*(B\GL(m))=H_{\GL(m)}^*(pt)=\QQ[\l_1,\ldots, \l_m]$,
\end{itemize}

Commutativity of the diagram tells us that for any form $\omega$ supported on some neighborhood $\widehat{N}^{k+1}(\bU) \subset \widehat{N}^{k+1}(X)$ we have  
\[\int_{\widehat{N}^{k+1}(X)}\omega =\int_{X} \int_{BN^{k+1}(\CC^n)} \omega |_{\{\l_1,\ldots, \l_m\} \to \text{Chern roots of } TX}\]
We apply this formula for the form $\Phi^\a$ coming from the sieve. According to Proposition \ref{thm:support} and Lemma \ref{lemma:support} for $\a=(\a_1,\ldots, \a_s) \in \Pi(k+1)$ the form $\Phi^\a$ is compactly supported in a neighborhood of 
\[\widehat{CN}^\a(X)=\pi_\a^{-1}(\CHilb^{[\a_1]}(X)\times \ldots \times \CHilb^{[\a_s]}(X)) \subset \widehat{N}^{k+1}(X)\]
and hence
\begin{equation}\label{reduceintegraltoaffinespace}
\boxed{\int_{\GHilb^{[k+1]}(X)}\Phi=\int_X \sum_{\a \in \Pi(k)} \int_{BN^{k+1}(\CC^m)} \Phi^\a |_{\l_1,\ldots, \l_m \to TX}}
\end{equation}
In the next section, using equivariant localisation we explain how to reduce the equivariant integration of a form which is supported on a small neighborhood of $CN^\a(\CC^m)$ to an integral over $CN^\a(\CC^m)$ itself. 

\subsection{Localisation over the flag} 

Let $r,m,k$ satisfy the numerical conditions of Theorem \ref{equivariantintegral} and let $V$ be a rank $r$ T-equivariant bundle over $\CC^m$ with T-equivariant Chern roots $\theta_1,\ldots, \theta_r$.  We aim to develop a formula for the equivariant integral
\[\int_{\GHilb^{[k+1]}(\CC^m)} c_d(V^{[k+1]}) \text{ for } d\ge \dim(\GHilb^{[k+1]}(\CC^m)=(k+1)m.\]
\begin{remark} In the multipoint formula we will have $V=f^*TN$, or more precisely, we look at $\Euler(f^*TN^{[k+1]}/W)$ for some subbundle $W$, consistent with the factors of $N^{k+1}(X)$
\end{remark}

By \eqref{reduceintegraltoaffinespace} we integrate over the balanced nested 
Hilbert scheme $BN^{k+1}(\CC^n)$ with baricenter at the origin. 
 The Hilbert sieve gives a decomposition $c_d=\sum_{\a\in \Pi(k+1)}c_d^\a$
where $c_d^\a$ is compactly supported in a small neighborhood of the curvilinear diagonal part 
\[\widehat{CN}^\a(\CC^m)=\pi_\a^{-1}(\CHilb^{[\a_1]}(\CC^m)\times \ldots \times \CHilb^{[\a_s]}(\CC^m)) \subset \widehat{N}^{k+1}(\CC^m)\]
Let  $BN^{\a}_\epsilon(\CC^m)$ denote a small neighborhood of $\widehat{CN}^\a(\CC^m)$ in $BN^{k+1}(\CC^m)$. 
Then 
\[\boxed{\int_{BN^{k+1}(\CC^m)}c_d(V^{[k+1]})=\sum_{\a \in \Pi(k+1)} \int_{BN^{\a}_\e(\CC^m)} c_d^\a(V^{[k+1]})}\]
First we focus on the innermost (deepest) term in the sieve, namely 
\[c_d^\Lambda(V^{[k+1]})=c_d(V^{[k+1]}) + \Sigma_{\b \in \Pi(k+1)\setminus \Lambda} (-1)^{|\b|-1}(|\b|-1)! c_d(V^{\b} ).\] 
supported in the small neighborhood 
$BN^{\Lambda}_\e(\CC^m)=BN^{k+1}_\e(\CC^m)$. Diagram \eqref{diagramsix} 
restricted to the origin in $\CC^m$ gives 
\begin{equation}\label{diagram7}
\xymatrix{
\mathbf{CN}^{k+1}(\CC^n) \ar@{^{(}->}[r] \ar[d] & \mathbf{BN}^{k+1}_\e(\CC^m) \ar[d] \\
\widehat{CN}^{k+1}(\CC^n) \ar@{^{(}->}[r] \ar[d]^{\rho} & BN^{k+1}_\e(\CC^m) \\
\Hom^{\nondeg}(\CC^k,\CC^m)/B_k=\flag_k(\CC^m)& } 
\end{equation}
Then by pulling back the sieve to the bold resolutions we get 
\[\boxed{\int_{BN^{k+1}(\CC^m)}c_d^\Lambda(V^{[k+1]})=\int_{\mathbf{BN}^{k+1}_\e(\CC^m)} c_d^\Lambda(V^{[k+1]})}\]
\begin{remark} In Diagram \eqref{diagramsix} and Diagram \eqref{diagram7} the bold top spaces are equivariant resolutions of the possibly singular spaces below them. Their introduction is completely formal: the equivariant localisation formula on a smooth space is more convenient than introducing a smooth ambient space and applying the Rossmann formula. We can and will pull-back the sieve to the bold resolutions, and after proving and applying the Residue Vanishing Theorem we will return to integrals over $\CHilb^{k+1}(\CC^m)$, and hence the explicit form of these resolutions is irrelevant. In the Residue Vanishing Theorem only their dimension is important. 
\end{remark}
Let $\l_1,\ldots, \l_m\in \mathfrak{t}^*$ denote the torus weights for the diagonal $T \subset \GL(m)$ action on $\CC^m$ with respect to the basis $e_1,\ldots, e_m \in \CC^m$ and let 
\[\ff=(\Span(e_1) \subset \Span(e_1,e_2)\subset \ldots \subset \Span(e_1,\ldots,e_k) \subset \CC^m)\]
denote the standard flag in $\CC^m$ fixed by the parabolic $P_{m,k} \subset \GL(m)$.

The Atiyah-Bott-Berline-Vergne  localisation formula of Proposition \ref{abbv} over the flag $\flag_k(\CC^m)$ for $c_d=c_d(V^{[k+1]}$ gives  
\begin{equation} \label{flagloc}
\int_{\mathbf{BN}^{\a}_\e(\CC^m)} c_d^\Lambda= \sum_{\sigma\in S_n/S_{n-k}}
\frac{c^\Lambda_{d,\sigma(\ff)}}{\prod_{1\leq j \leq
k}\prod_{i=j+1}^m(\lambda_{\sigma\cdot
    i}-\lambda_{\sigma\cdot j})},
\end{equation}
where 
\begin{itemize}
\item $\sigma$ runs over the ordered $k$-element subsets of $\{1,\ldots, m\}$ labeling the torus-fixed flags $\sigma(\ff)=(\Span(e_{\sigma(1)}) \subset \ldots \subset \Span(e_{\sigma(1)},\ldots, e_{\sigma(k)}) \subset \CC^m)$ in $\CC^m$.
\item $\prod_{1\leq j \leq k}\prod_{i=j+1}^n(\lambda_{\sigma(i)}-\lambda_{\sigma(j)})$ is the equivariant Euler class of the tangent space of $\flag_k(\CC^m)$ at $\s(\ff)$.
\item if $\mathbf{BN}_{\sigma(\ff)}=\rho^{-1}(\sigma(\ff)) \subset \mathbf{BN}^{[k+1]}_\e$ denotes the fiber then $c^\Lambda_{d,\sigma(\ff)}=(\int_{\mathbf{BN}_{\sigma(\ff)}} c_d^\omega)^{[0]}(\sigma(\ff))\in S^\bullet \mathfrak{t}^*$ is the differential-form-degree-zero part evaluated at $\sigma(\ff)$.
\end{itemize}
The Chern roots of the tautological bundle over $\flag_k(\CC^m)$ at the fixed point $\sigma(\ff)$ are represented by $\l_{\s(1)}, \ldots ,\l_{\s(k)}\in \mathfrak{t}^*$ and therefore if $\nu$ is a Chern polynomial of the tautological bundle then 
\begin{equation}\label{alphasigmaf}
\nu_{\s(\ff)}=\sigma \cdot \nu_\ff=\nu_\ff(\l_{\s(1)}, \ldots ,\l_{\s(k)})\in S^\bullet \mathfrak{t}^*,
\end{equation}
is the $\sigma$-shift of the polynomial $\nu_{\ff}=(\int_{\mathbf{BN}_\ff}\nu)^{[0]}(\ff)\in S^\bullet \mathfrak{t}^*$ corresponding to the distinguished fixed flag $\ff$. By \eqref{taurestrictedtocurvi} the restriction of $V^{[k+1]}$ to the curvilinear part $\mathbf{CN}^{[k+1]}(\CC^m)$ is the tensor product
\[V^{[k+1]}|_{\mathbf{CN}^{[k+1]}(\CC^m}=V \otimes  \calo_{\CC^m}^{[k+1]}\]
The test curve model formulated in Theorem \ref{bszmodel} (3) then tells that  
\[V^{[k+1]}=V \otimes \calo_{\CC^m}^{[k+1]}=V \otimes \cale\]
where $\cale$ is the tautological bundle over $\grass_{k+1}(S^\bullet \CC^m)$. Hence the Chern roots of $V^{[k+1]}$ on $\mathbf{CN}^{k+1}$, and hence at the torus fixed points, are  the pairwise sums formed from Chern roots of $V$ and Chern roots of $\cale$. This means that  
\begin{equation}\label{cdff}
c_{d,\ff}^\Lambda(V^{[k+1]})=c_{d,\ff}(\l_1,\ldots, \l_k)
\end{equation}
is a polynomial in the tautological bundle.

\subsection{Transforming the localisation formula into iterated residue}\label{subsec:transform}
We transform the right hand side of \eqref{flagloc} into an iterated residue motivated by B\'erczi--Szenes \cite{bsz}. This step turns out to be crucial in handling the combinatorial complexity of the fixed point data in the Atiyah-Bott localisation formula and condense the symmetry of this fixed point data in an efficient way which enables us to prove the residue vanishing theorem.

To describe this formula, we will need the notion of an {\em iterated
  residue} (cf. e.g. \cite{szenes}) at infinity.  Let
$\omega_1,\dots,\omega_N$ be affine linear forms on $\CC^k$; denoting
the coordinates by $z_1,\ldots, z_k$, this means that we can write
$\omega_i=a_i^0+a_i^1z_1+\ldots + a_i^kz_k$. We will use the shorthand
$h(\bz)$ for a function $h(z_1\ldots z_k)$, and $\dbz$ for the
holomorphic $n$-form $dz_1\wedge\dots\wedge dz_k$. Now, let $h(\bz)$
be an entire function, and define the {\em iterated residue at infinity}
as follows:
\begin{equation}
  \label{defresinf}
 \ires \frac{h(\bz)\,\dbz}{\prod_{i=1}^N\omega_i}
  \overset{\mathrm{def}}=\left(\frac1{2\pi i}\right)^k
\int_{|z_1|=R_1}\ldots
\int_{|z_k|=R_k}\frac{h(\bz)\,\dbz}{\prod_{i=1}^N\omega_i},
 \end{equation}
 where $1\ll R_1 \ll \ldots \ll R_k$. The torus $\{|z_m|=R_m;\;m=1 \ldots
 k\}$ is oriented in such a way that $\res_{z_1=\infty}\ldots
 \res_{z_k=\infty}\dbz/(z_1\cdots z_k)=(-1)^k$.
We will also use the following simplified notation: $\sires \overset{\mathrm{def}}=\ires.$

In practice, one way to compute the iterated residue \eqref{defresinf} is the following algorithm: for each $i$, use the expansion
 \begin{equation}
   \label{omegaexp}
 \frac1{\omega_i}=\sum_{j=0}^\infty(-1)^j\frac{(a^{0}_i+a^1_iz_1+\ldots
   +a_{i}^{q(i)-1}z_{q(i)-1})^j}{(a_i^{q(i)}z_{q(i)})^{j+1}},
   \end{equation}
   where $q(i)$ is the largest value of $m$ for which $a_i^m\neq0$,
   then multiply the product of these expressions with $(-1)^kh(z_1\ldots
   z_k)$, and then take the coefficient of $z_1^{-1} \ldots z_k^{-1}$
   in the resulting Laurent series.

\begin{proposition}[{\rm B\'erczi--Szenes \cite{bsz}, Proposition 5.4}]\label{ABtoresidue}  For any homogeneous polynomial $Q(\bz)$ on $\CC^k$ we have
\begin{equation}\label{flagres}
\sum_{\sigma\in S_m/S_{m-k}}
\frac{Q(\lambda_{\sigma(1)},\ldots ,\lambda_{\sigma(k)})}
{\prod_{1\leq j\leq k}\prod_{i=j+1}^m(\lambda_{\sigma\cdot
    i}-\lambda_{\sigma\cdot j})}=\sires
\frac{\prod_{1\leq i<j\leq k}(z_j-z_i)\,Q(\bz)\dbz}
{\prod_{i=1}^k\prod_{j=1}^m(\lambda_j-z_i)}.
\end{equation}
\end{proposition}

\begin{remark}
  Changing the order of the variables in iterated residues, usually,
  changes the result. In this case, however, because all the poles are
  normal crossing, formula \eqref{flagres} remains true no matter in
  what order we take the iterated residues.
\end{remark}

This together with \eqref{flagloc},\eqref{alphasigmaf} and \eqref{cdff} gives
\begin{corollary}\label{propflag} Let $k\le m$. Then  
\begin{equation*}
\int_{\mathbf{BN}^{k+1}_\e(\CC^m)} c_d^\Lambda=\sires
\frac{\prod_{1\leq i<j\leq k}(z_i-z_j)\, c_{d,\ff}(z_1, \ldots ,z_k)\dbz}
{\prod_{l=1}^k\prod_{i=1}^m(\lambda_i-z_l)}
\end{equation*}
\end{corollary} 
In the formula above 
\begin{itemize}
\item $c_{d,\ff}=\int_{\mathbf{BN}_{\ff}} c_d^\omega$ is the integral over the 
fiber, which we will calculate using a second equivariant localisation.
\item $z_1,\ldots, z_k$ are the $T$-weights of the tautological bundle over $\grass_k(\symdot)$. Equivalently, the $\GL(m)$ action on $\Hom^\ff(\CC^k,\CC^m)$ reduces to $\GL(\CC_{[k]}) \subset \GL(m)$, and $z_1,\ldots z_k$ are the weights of the $T^k_\bz \subset \GL(\CC_{[k]})$ action. 
\end{itemize}
\subsection{Second equivariant localisation over the base flag}\label{subsec:secondlocalization} Next we proceed a second equivariant localisation on $\mathbf{BN}_\ff$ to compute $c_{d,\ff}(z_1,\ldots ,z_k)$.  
The torus fixed points sit in the curvilinear locus $\mathbf{CN}_\ff \subset \mathbf{BN}_\ff$, and this curvilinear fiber $\mathbf{CN}_\ff \subset \mathbf{BN}_\ff$ sits over 
\[\CHilb_\ff=\rho^{-1}(\ff)\simeq \overline{P_{m,k}\cdot p_{m,k}} \subset \grass_k(\sym^{\le k}\CC_{[k]})\]
where  
\[p_{m,k}=\mathrm{Span}(e_1,e_2+e_1^2, \ldots, \sum_{\tau \in \mathcal{P}(k)}e_{\tau}) \in \grass_k(\sym^{\le k}\CC_{[k]}),\] 
and $P_{m,k} \subset \GL_m$ is the parabolic subgroup which preserves $\ff$. Equivalently, 
\[\CHilb_\ff=\overline{\phi(\Hom^\ff(\CC^k,\CC^m)}\]
 where 
 \[\Hom^\ff(\CC^k,\CC^m)=\{\psi \in \Hom(\CC^k,\CC^m): \psi(e_i)\subset \CC_{[i]} \text{ for } i=1,\ldots, k\}\]
and $\CC_{[i]} \subset \CC^n$ is the subspace spanned by $e_1,\ldots, e_i$.  

We use the diagram \eqref{diagram8} of embeddings into configuration space to describe weights at fixed points.  The torus fixed points in $\CHilb^{k+1}_\ff \subset \grass_k(\sym^{\le k}\CC^k)$ are subspaces of the form
\[W_\tau=\mathrm{Span}(e_{\tau_1}, e_{\tau_2} , \ldots , e_{\tau_k})\]
parametrised by $k$-tuples $\tau=(\tau_1,\ldots, \tau_k)$ where 
\begin{equation}\label{tauconditions}
\tau_i \subset \{1,\ldots, i\} \text { such that } \Sigma \tau_i \le i. 
\end{equation}
These are called admissible $k$-tuples, and those $\tau$'s which correspond to fixed points in $\CHilb^{k+1}_\ff$ are called curvilinear admissible subsets, the set of these is $\mathcal{P}(k)$. 
 
 A point in $\mathbf{CN}_\ff$ has the form  $\xi=(\xi_A: A\subset \{0,\ldots, 
 k\})$. If this is torus fixed, then $\phi^\grass(\xi_{\{0,1\ldots, 
 k\}})=W_\tau$ for some $\tau=(\tau_1, \ldots, \tau_k)$ and 
$\phi^\grass(\xi_A)=W_{A,\tau} \subset W_\tau$
is a torus-fixed subset of dimension $|A|-1$, hence 
\[W_{A,\tau}=\mathrm{Span}(e_{\tau_j}:j\in \Lambda_A)\]
for some $\Lambda_A \subset \{1,\ldots, k\}$ with $|\Lambda_A|=|A|-1$. The $T_\bz^k$ weights on the tautological bundle $\cale$ over $\grass_k(\symdot)$ are $z_1,\ldots, z_k$, hence for a subset $A \subset \{0,\ldots, k\}$ the equivariant Chern roots (i.e. torus weights) of $W_{A,\tau}$ are $\{z_{\tau_j}: j \in \Lambda_{A}\}$
where $z_{\tau_j}=\sum_{i \in \tau_j} z_i$. A simple but crucial observation is the following
\begin{lemma}\label{crucial1}
If $\tau \neq \{[1],[2],\ldots, [k]\}$, then at least one integer $2\le i \le k$ does not appear in $\tau$. Hence the $T^k_\bz$-weight of $W_{A,\tau}$ does not depend on $z_i$ for any $A$. 
\end{lemma}

Let $\mathcal{F}_\ff$ denote the set of torus fixed points on $\mathbf{CN}_\ff$. The ABBV localisation formula on $\mathbf{BN}_\ff$ gives  
\begin{equation}\label{intnumberone} 
\boxed{\int_{\mathbf{BN}^{k+1}_\e(\CC^m)}c_d^\Lambda(V^{[k+1]})=\\
=\sum_{F\in \mathcal{F}_\ff} \sires \frac{
 \prod_{i<j}(z_i-z_j) c_{d,F}^\Lambda (z_1,\ldots ,z_k)}{
\Euler_F(\mathbf{BN}_\ff)  \prod_{l=1}^k\prod_{i=1}^m(\lambda_i-z_l)} \,\dbz}
  \end{equation}
where 
\[c_{d,F}^\Lambda(\bz)=c_{d,F}^\Lambda(V^{[k+1]})=c_{d,F}(V^{[k+1]}) + \Sigma_{\a \in \Pi(k+1)\setminus \Lambda} (-1)^\a  c_{d,F}(V^{\a} ).\]

Next we identify $c_{d,F}(V^\a)$.
Let $\a=(\a_1,\ldots, \a_s) \in \Pi(k+1)$ be a partition of the $k+1$ points. Recall from \eqref{taurestrictedtocurvi} that on the curvilinear locus $V^\a$ is the tensor product 
\[V^{\a}=V^{[\a_1]}\oplus \ldots \oplus V^{[\a_s]}=(V \otimes \calo_{\CC^m}^{[\a_1]}) \oplus \ldots \oplus (V \otimes \calo_{\CC^m}^{[\a_s]})\]
and by the test curve model formulated in Theorem \ref{bszmodel} (3) the 
tautological bundle is $\calo_{\CC^m}^{[k+1]}/\calo_\grass=\cale$, where 
$\cale$ is the tautological bundle over $\grass_k(\symdot)$.
Hence the equivariant Euler class of $V^{\a}=V^{[\a_1]}\oplus \ldots \oplus V^{[\a_s]}$ at the fixed point $F=(W_{A,\tau})$  is the product of all weights
\begin{equation}\label{euler}
\Euler_{F}(V^\a)=\prod_{i=1}^s \Euler(V_{\a_i,\tau}(\bz))=\prod_{\ell=1}^r\prod_{i=1}^s \theta_\ell \prod_{j\in \Lambda_{\a_i}} (\theta_\ell+z_{\tau_j})
\end{equation}
More generally, the $d$th equivariant Chern class of $V^\a$ at $F$ is the $d$th elementary symmetric polynomial  
\begin{equation}\label{chern}
c_{d,F}(V^\a)=\s_d(\{\theta_\ell,\ldots, \theta_\ell, \theta_\ell+z_{\tau_j}\}:1\le \ell \le r, 1\le i\le s, j\in \Lambda_{\a_i})
\end{equation}
where we have $s$ copies of $\theta_\ell$ for all $1\le \ell \le r$.
Note that the total degree of $c_{d,F}(V^\a)$ in $z_1,\ldots, z_k$ satisfies
\begin{equation}\label{degreeinz}
\deg_{\bz}c_{d,F}(V^\a)\le r\sum_{i=1}^s |\Lambda_{\a_i}|=r(k-s+1)
\end{equation}

\subsection{The first residue vanishing theorem}

\begin{proposition}[B\'erczi-Szenes, \cite{bsz}]
  \label{vanishprop}
Let $p(\bz)$ and $q(\bz)$ be polynomials in the variables $z_1\ldots 
z_k$, and assume that $q(\bz)$ is a product of linear factors:
$q(\bz)=\prod_{i=1}^NL_i$; set $\dbz=dz_1\dots dz_k$. Then
\[ \ires\frac{p(\bz)\dbz}{q(\bz)} = 0
\]
if for some $l\leq k$, the following holds:
\[\deg(p(\bz);l,l+1,\dots,k)+k-l+1<\deg(q(\bz);l,l+1,\dots,k)\]
\end{proposition}

We apply this proposition for the terms in \eqref{intnumberone}. 
\begin{theorem}{\textbf{First Residue Vanishing Theorem}}\label{vanishing1}  Let $m,r,k$ satisfy the numerical conditions of Theorem \ref{equivariantintegral}. Assume $\a=(\a_1,\ldots, \a_s)\in \Pi(k+1)$ is a partition with $s>1$. Then the corresponding residue
\begin{equation}\label{eqn:residue}
\sires \frac{\prod_{1\le i<j \le k}(z_i-z_j) c_{d,F}(V^\a)}{
\Euler_F(\mathbf{BN}_\ff)  \prod_{l=1}^k\prod_{i=1}^m(\lambda_i-z_l)} \,\dbz
\end{equation}
vanishes at every fixed point $F=(V_{A,\tau})$ on $\mathbf{BN}_\ff$.
\end{theorem}
\begin{proof}
We use the shorthand notation $\deg(p(\bz))=\deg(p(\bz);1,2,\dots, k)$ for the total degree of a polynomial in $z_1,\ldots, z_k$.   
By \eqref{degreeinz} for. $s>1$ 
\[\deg(c_{d,F}(V^\a))\le r(k-1)\]
On the other hand 
\begin{eqnarray}
\deg(\Pi_{l=1}^k\Pi_{i=1}^m(\lambda_i-z_l)) = mk\\
\deg(\Pi_{1\le i<j\le k}(z_i-z_j))={k \choose 2}\\
\deg(\Euler_F(\mathbf{BN}_\ff))=\dim(\mathbf{BN}_\ff)={k \choose 2}
\end{eqnarray}
We apply Proposition \ref{vanishprop} with $l=1$, and check the total degree of the rational expression. If 
\[r(k-1)+{k \choose 2}+k<mk+{k\choose 2}\]
holds, then the iterated residue vanishes. Rearranging this inequality we get
\[r<(m-1)\frac{k}{k-1}\]
which is the numerical assumption in Theorem \ref{tauintegral} and Theorem \ref{equivariantintegral}.
\end{proof}

The residue vanishing theorem reduces the formula \eqref{intnumberone} to 
\begin{equation}\label{intnumberthree} 
\boxed{\int_{\mathbf{BN}^{k+1}_\e(\CC^m)}c_d^\Lambda(V^{[k+1]})=\sum_{F\in \mathcal{F}_\ff} \sires \frac{
 \prod_{i<j}(z_i-z_j) c_{d,F}(V^{[k+1]})}{
\Euler_F(\mathbf{BN}_\ff)  \prod_{l=1}^k\prod_{i=1}^m(\lambda_i-z_l)} \,\dbz}
  \end{equation}
However, the bundle $V^{[k+1]}$ is pulled-back to $\mathbf{BN}^{k+1}_\e(\CC^m)$ from a small neighborhood $\CHilb^{k+1}_\e(\CC^m)$ of $\CHilb^{k+1}(\CC^n)$. Hence the right hand side of \eqref{intnumbertwo} is the localisation formula for the integral over the small neighborhood $\CHilb^{k+1}_\e(\CC^m)$, which we formulate in the following corollary.
\begin{corollary} Let $m,r,k$ satisfy the numerical conditions of Theorem \ref{equivariantintegral}. Then 
\[\boxed{\int_{\mathbf{BN}^{k+1}_\e(\CC^m)}c_d^\Lambda(V^{[k+1]})=\int_{\CHilb^{k+1}_\e(\CC^m)} c_d(V^{[k+1]})}\]
and this is given by the formula 
\begin{equation}\label{intnumberfourb} 
\boxed{\int_{\CHilb^{k+1}_\e(\CC^m)}c_d(V^{[k+1]})=\sum_{F\in \mathcal{P}_\ff} \sires \frac{
 \prod_{i<j}(z_i-z_j) c_{d,F}(V^{[k+1]})}{
\Euler_F(\CHilb^{k+1}_\ff(\CC^m))  \prod_{l=1}^k\prod_{i=1}^m(\lambda_i-z_l)} \,\dbz}
  \end{equation}
where $\mathcal{P}_\ff=\pi_\Lambda(\mathcal{F}_\ff)$ is the set of $T_\bz^k$-fixed points on $\CHilb^{k+1}_\ff=\rho^{-1}(\ff)$ where $\rho: \widehat{\CHilb}^{k+1}(\CC^m) \to 
\flag_k(\CC^m)$ as in Diagram \eqref{diagramsix}.  
\end{corollary}

Recall from \eqref{tauconditions} that the fixed points in $\mathcal{P}_\ff$ are parametrised by sequences $\tau=\tau_1-\ldots -\tau_k$ where 
\begin{equation}
\tau_i \subset \{1,\ldots, i\} \text { such that } \Sigma \tau_i \le i. 
\end{equation}
We call these admissible sequences, and those admissible sequences which correspond to fixed points in $\CHilb^{k+1}_\ff$ are called curvilinear admissible sequences.

In the next section we prove a local structure theorem for properly supported forms $\a$. This model roughly says that the support $\mathrm{supp}(\a)$ of such a form is birational to the total space of a bundle $B$ over the punctual part $\supp_0(\a)=\supp(\a)\cap \CHilb^{k+1}(\CC^n)$.

\subsection{The local model for properly supported forms}
Let $\a \in \Omega^\bullet(\GHilb^{k+1}(\CC^m))$ be a properly supported form as in Definition \ref{geocond}. This means that $\supp(\a)$ is locally irreducible at every point and 
\[\supp_0(\a):=\supp(\a) \cap \GHilb^{k+1}_0(\CC^m) \subset \CHilb^{k+1}(M)\]
Recall that $\Curv^{k+1}(\CC^m) \subset \CHilb^{k+1}(\CC^m)$ is  the nonsingular open locus which parametrises curvilinear subschemes. Let $B \to \Curv^{k+1}(\CC^m)$ denote the normal bundle of $\Curv^{k+1}(\CC^m)$ in $\GHilb^{k+1}(\CC^m)$.
\begin{proposition}{(Haiman \cite{haiman})} $B=\calo_{\CC^m}^{[k+1]}/\calo$ is a rank $k$ bundle.
\end{proposition}
Note that the tautological bundle $B$ extends over the whole $\GHilb^{k+1}(\CC^m)$, and in particular, over the closure $\CHilb^{k+1}(\CC^m)=\overline{\Curv^{k+1}(\CC^m)}$. For properly supported $\a$  the support $\supp(\a)$ is locally irreducible at points in $\CHilb^{k+1}(\CC^m)$, hence 
\[\supp_\nabla(\a):=\supp(\a) \cap \GHilb_\nabla(\CC^m) \subset B|_{\CHilb^{k+1}(\CC^m)}\]
Hence $\supp(\a)$ locally modeled over $\supp(\a)_0$ as a bundle, which means that 
\begin{equation}\label{localmodelgeneral}
\xymatrix{\supp_\nabla(\a) \ar[d]  \ar[r]^\rho &  B  \ar[d]  \\
  \supp_0(\a)  \ar@{^{(}->}[r] &  \CHilb^{k+1}(M)}
 \end{equation}
there is a topological isomorphism $\rho$ from $\supp_\nabla(\a)$ to the total space of $B$. This local model combined with Thom-isomorphism gives 
\begin{equation}
\boxed{\int_{\CHilb^{k+1}_\e(\CC^m)} c_d(V^{[k+1]})=\int_{\CHilb^{k+1}(\CC^m)} \frac{c_d(V^{[k+1]})}{\Euler(B)}}
\end{equation}

\subsection{The second residue vanishing theorem}
We can follow the same argument we developed over the nested Hilbert scheme $\widehat{N}^{k+1}(\CC^m)$, but now for $\widehat{\CHilb}^{k+1}(\CC^m)$: we transform the localisation into iterated residue over $\flag_k(\CC^m)$, then we apply a second localisation over the fiber with the residual $k$-dimensional torus with weights $z_1, \ldots, z_{k}$. The fiber $\CHilb^{k+1}_\ff$ over the flag $\ff=(\Span(e_1) \subset \ldots \subset \Span(e_1,\ldots, e_k)=\CC_{[k]})$ is singular, but it sits in the smooth ambient space 
\[\flag_\ff=\flag_k(\Sym^{\le k}\CC_{[k]})\] 
and hence the Rossmann localisation formula of Proposition \ref{rossman} gives
\begin{equation}\label{intnumbertwo} 
\int_{\CHilb^{k+1}(\CC^m)}\frac{c_d(V^{k+1})}{\Euler(B)}=
\sum_{F\in \mathcal{P}_\ff} \sires \frac{\emu_F[\CHilb^{k+1}_\ff,\flag_\ff]
\prod_{i<j}(z_i-z_j) c_{d,F}(V^{k+1})}{
\Euler_F(B)\cdot  \Euler_F(\flag_\ff) \prod_{l=1}^k\prod_{i=1}^m(\lambda_i-z_l)} \,\dbz
  \end{equation}
Note that at the fixed point $F=\pi=(\pi_1,\ldots, \pi_k)$ corresponding to the admissible sequence $(\pi_1,\ldots, \pi_k)$ the Euler class of $B$ is 
\[\Euler_F(B)=z_{\pi_1} \ldots z_{\pi_k}.\]
The tangent space of $\flag_\ff$ at $\pi=(\pi_1,\ldots, \pi_k)$ is 
\[\Euler_\pi(\flag_\ff)=\prod_{l=1}^k\prod_{\Sigma \tau \leq l}^{\tau\neq\pi_1\ldots \pi_l}
(z_{\tau}-z_{\pi_l})\]
We prove in \cite{b0} the following vanishing theorem
\begin{theorem}{\textbf{Second Residue Vanishing Theorem (\cite{b0} Theorem 6.1)}} Let $m \le k+1$. Then 
\begin{enumerate}
\item All terms but the one corresponding to $\pi_\dist = ([1], [2],\ldots , [k])$ vanish in \eqref{intnumbertwo}, leaving us with 
\begin{equation}\label{eqn:residue2}
\sires \frac{\emu_{\pi_\dist}[\CHilb^{k+1}_\ff,\flag_\ff]
\prod_{i<j}(z_i-z_j) c_{d,F}(V^{k+1})}{
\Pi_{\Sigma \tau \le l \le k}(z_\tau-z_l) \cdot  z_1 \ldots z_k \prod_{l=1}^k\prod_{i=1}^m(\lambda_i-z_l)} \,\dbz=0
\end{equation}
\item If $|\tau|\ge 3$ then $\emu_{\pi_\dist}[\CHilb^{k+1}_\ff,\flag_\ff]$ is divisible by $z_\tau-z_l$ for all $l \ge \Sigma \tau$. Let 
\[Q_k(\bz)=\frac{\emu_{\pi_\dist}[\CHilb^{k+1}_\ff,\flag_\ff]}{\prod_{\substack{|\tau| \ge 3\\\Sigma \tau \le l \le k}}z_\tau-z_l}\]
denote the quotient polynomial. Then we get the simplified formula:
\begin{equation} \label{intnumberthreeb}
\sires \frac{Q_k(\bz)\,\prod_{i<j}(z_i-z_j) c_{d,F}(V^{k+1})}{
\prod_{i+j \le l \le k}
(z_i+z_j-z_l)  \prod_{l=1}^k\prod_{i=1}^m(\lambda_i-z_l)} \,\dbz
  \end{equation}
\end{enumerate} 
\end{theorem}

\begin{remark}\label{remarkq}
The geometric meaning of $Q_k(\bz)$ in \eqref{intnumberthreeb} is the following, see also \cite[Theorem 6.16]{bsz}. Let $T_k\subset B_k\subset \GL(k)$ be the subgroups of invertible
diagonal and upper-triangular matrices, respectively; denote the
diagonal weights of $T_k$ by $z_1, \ldots, z_k$.  Consider the $\GL(k)$-module of 3-tensors $\Hom(\CC^k,\sym^2\CC^k)$; identifying the
weight-$(z_m+z_r-z_l)$ symbols $q^{mr}_l$ and $q^{rm}_l$, we can
  write a basis for this space as follows:
\[ \Hom(\CC^k,\sym^2\CC^k)=\bigoplus \CC q^{mr}_l,\;  1\leq m,r,l \leq k.
\]
Consider the point $\epsilon=\sum_{m=1}^k\sum_{r=1}^{k-m}q_{mr}^{m+r}$
in  the $B_k$-invariant subspace
\begin{equation*}
  \label{nhmodule}
    W_k = \bigoplus_{1\leq m+r\leq l\leq k} \CC q^{mr}_l\subset
\Hom(\CC^k,\sym^2\CC^k).
\end{equation*}
Set the notation $\OO_k$ for the orbit closure
$\overline{B_k\epsilon}\subset W_k$, then $Q_k(\bz)$ is the $T_k$-equivariant
Poincar\'e dual $Q_k(\bz) = \epd{\OO_k,W_k}_{T_k}$,
which is a homogeneous polynomial of degree
$\dim(W_k)-\dim(\OO_k)$. For small $k$ these polynomials are the following (see \cite[Section 7]{bsz}):
\[Q_2=Q_3=1, Q_4=2z_1+z_2-z_4\]
\[Q_5=(2z_1+z_2-z_5)(2z_1^2 +3z_1z_2-2z_1z_5+2z_2z_3-z_2z_4-z_2z_5-z_3z_4+z_4z_5).\] 
\end{remark}

Hence we arrive to the final formula
\begin{equation}\label{intnumberfour} 
\boxed{\int_{\CHilb^{k+1}(\CC^m)}\frac{c_d(V^{k+1})}{\Euler(B)}=
 \sires \frac{
Q_k(\bz) \prod_{i<j}(z_i-z_j) \s_d(\{\theta_\ell,\theta_\ell+z_{\tau_i}\}:1\le \ell \le r, 1\le i\le k)}{
z_1\ldots z_k \prod_{i+j \le l \le k}(z_i+z_j-z_l)  \prod_{l=1}^k\prod_{i=1}^m(\lambda_i-z_l)} \,\dbz}
  \end{equation}
This completes the proof of Theorem \ref{tauintegral} and Theorem \ref{equivariantintegral}.

\section{Tautological push-forwards on Hilbert schemes}\label{push-forwards}

In this section, following the argument outlined in the previous section, we formulate an extension of the tautological integral formula in Theorem \ref{tauintegral},  where push-forward to the point is replaced by a pushforward along a smooth map between varieties.

Let $f:M \to N$ be a generic holomorphic map between the smooth complex manifolds $M$ and $N$. 
Following the notations introduced in \S \ref{subsec:reformulation}, we let $V \to M$ be a rank $r$ vector bundle, $V^{[k]} \to \Hilb^{[k]}(M)$ the tautological bundle over the ordered Hilbert scheme and $f^{[k]}=f^{\times k} \circ \HC$, which we restrict to the main component $\GHilb^k(M) \subset \Hilb^{[k]}(M)$:
{\small \[\xymatrix{ V^{[k]} \ar[d] & & \\
 \GHilb^{k}(M) \ar[r]^-{f^{[k]}} & N^k }\]}
We aim to calculate 
\begin{equation}\label{pushforward} 
f^{[k]}_* [\Phi(V^{[k]})] \in H^*(N^k)
\end{equation}
for specific Chern polynomials $\Phi$ under assumptions on the map $f$ and the support of $\Phi$. We recall some notations and terminology from \S \ref{subsec:reformulation}:
\begin{itemize}
	\item We will use the standard notation $f_*[\Psi]$ for the push-forward 
	of a cohomology class $\Psi\in H^*(M)$. Occasionally, we will indicate the 
	source space in the index for clarity, as in  $f_*[\Psi]_M$. 
	\item We will write $ f^{\times k}: M^k \to N^k$, and $ f^{[k]} =  f^{\times r}\circ \HC$ where $\HC:\Hilb^{[k]}M \to M^k$ is the ordered Hilbert-Chow. We can and will replace $k$ with a subset of $\{1,\ldots, k\}$ in these notations, which refers to the restriction to the points formed by this subset.
	\item For a partition $\a \in \Pi(k)$ the embeddings of the diagonals will be denoted by $\iota_{\Delta_\a M}:\Delta_\a M\hookrightarrow M^k$ and $i_{\Delta_\a N}: \Delta_\a N\hookrightarrow N^k$.
\end{itemize}
The sieve over the fully nested Hilbert scheme $N^{k+1}(M)$ explained in \S \ref{subsec:sieve} gives a decomposition
\[\Phi= \sum_{\a \in \Pi(k+1)} \Phi^\a,\] 
where for any partition $\{0,\ldots, k\}=\a_1 \sqcup \ldots \sqcup \a_s$ the form $\Phi^\a$ is supported on the approximating punctual subset  
\[\supp(\Phi^\a)=N^{\a}_0(X)=\HC^{-1}(\Delta_{\a})=\pi_\a^{-1}(\GHilb^\a_0(M))\]
where $\GHilb^{\a}_0(M)=\GHilb^{\a_1}_0(M) \times \ldots \times \GHilb^{\pi_s}_0(M)$.
Here $\Phi^\a$ is a linear combination of the classes $\Phi(V^{[\b]})$ for those partitions $\b \in \Pi(k+1)$ which are refinements of $\a$, that is, $\b \le \a$ holds. 
 In particular, over the small diagonal corresponding to the trivial partition $\Lambda=[0,\ldots, k]$ the sieve formula gives 
\[\Phi^{\Lambda}(V^{[k+1]})=\sum_{\b \in \Pi(k+1)}(-1)^{|\b|-1}(|\b|-1)!\Phi(V^\b).\]

Then $f^{[k+1]}$ factors over the $\a$-diagonal part as 
\begin{equation}\label{diagram:fpi1}
\xymatrixcolsep{4pc} 
\xymatrix{\GHilb^{\a}_0(M) \ar[r]^-{f^{[\a_1]} \times \ldots \times f^{[\a_s]}} & \Delta_{\a_1}N \times \ldots \times \Delta_{\a_s}N \ar[r] \ar@{^{(}->}[d]^-{\iota_{\Delta_\a N}} & \Delta N \ar@{^{(}->}[d]^{\iota_{\Delta N}} \\
  & N^{\a_1}  \times \ldots \times N^{\a_s}  \ar[r]^-{=} & N^k}
\end{equation}
Then, with the same proof as for Theorem \ref{tauintegral} we arrive to the following push-forward analogue. 

\begin{theorem}[\textbf{Push-forwards on geometric components of Hilbert schemes}]\label{thm:pushforward}
Let $f:M \to N$ be a Thom-Boardman map between the smooth complex manifolds $M$ 
and $N$. Let $V$ be a rank $r$ vector bundle over $M$. Assume that $r,k,n$ 
satisfy the inequalities
\[k\le m \le r \le (m-1)\frac{k}{k-1},\]
and $\Phi=c_{(k+1)n}(V^{[k+1]})$ is properly supported on $\GHilb^{k+1}(X)$. Then the map $f^{[k+1]}: \GHilb^{k+1}(M) \to N^{k+1}$ has the following closed form:
\[f^{[k+1]}_* [\Phi]=\sum_{(\a_1,\ldots, \a_s) \in \Pi(k+1)} (\iota_{\Delta_\a N})_* f^{\times s}_* \left [\res_{\bz^{\a_1}=\infty}\ldots \res_{\bz^{\a_s}=\infty} \mathcal{R}^\a(\theta_i, \bz) d\bz^{\a_s}\ldots d\bz^{\a_1} \right ]_{\Delta_\a M}\]
where $\mathcal{R}^\a(\a_i, \bz^{\a_1},\ldots, \bz^{\a_s})$ is the rational expression 
\[\Phi(V(\bz^{\a_1}) \oplus \ldots \oplus V(\bz^{\a_s})) 
\prod_{l=1}^s \frac{(-1)^{|\a_l|-1} \prod_{1\le i<j \le |\a_l|-1}(z_i^{\a_l}-z_j^{\a_l})Q_{|\a_l|-1}d\bz^{\a_l}}
{\prod_{i+j\le l\le 
|\l_l|-1}(z_i^{\a_l}+z_j^{\a_l}-z_l^{\a_l})(z_1^{\a_l}\ldots z_{|\a_l|-1}^{\a_l})^{n+1}}\prod_{i=1}^{|\l_l|-1} 
s_X\left(\frac{1}{z_i^{\a_l}}\right)\]
identical to the one in Theorem \ref{tauintegral}. Here 
\begin{enumerate}
\item For each partition $\a=(\a_1,\ldots, \a_s)$ we introduce a set of variables 
\[\bz^{\a_i}=\{z^{\a_i}_1,\ldots, z^{\a_i}_{|\a_i|-1}\}\]
for each element of the partition.  
\item $V(z)$ stands for the bundle $V$ tensored by the line $\mathbb{C}_z$, which is the representation of a torus $T$ with weight $z$, and hence its Chern roots are $z+\theta_1,\ldots ,z+\theta_r$. Moreover 
	\[V(\bz^{\a_l})=V \oplus V(z^{\a_l}_1) \oplus \ldots \oplus V(z^{\a_l}_{|\a_l|-1})\]
	has rank $r|\a_l|$, and 
	\[ \Phi(V(\bz^{\a_1}) \oplus \ldots \oplus V(\bz^{\a_s})) =  \Phi(\theta_i^1,z^{\a_1}_1+\theta_i,\ldots z^{\a_1}_{|\a_1|-1}+\theta_i, \ldots, \theta_i^s,z^{\a_r}_1+\theta_i,\ldots z^{\a_r}_{|\a_r|-1}+\theta_i)\]
	Note that we have $s$ copies of the roots $\theta_1,\ldots , \theta_r$ of $V$, and we think of the $i$th copy $\theta_1^i,\ldots , \theta_r^i$ as the Chern roots of $V=V^i$ sitting over the $i$th copy of $M$ in $M^s$.
	\item The iterated residue is a homogeneous symmetric polynomial of degree $ns$ in the Chern roots $\theta^i_j$, that is, the Chern roots of $V^1 \oplus \ldots \oplus V^s$ over $\Delta_\a M=M^s$, and $f^{\times s}$ sends this to $\Delta_\a N=N^s$.
\end{enumerate}
\end{theorem}

We finally formulate a special case of Theorem \ref{thm:pushforward} which we will refer at in the proof of the main theorem. For the set-up let $\a=(\a_1,\ldots, \a_s) \in \Pi(k+1)$ be a partition and take the following extended version of diagram \eqref{diagram:fpi1}
\begin{equation}\label{diagram:fpi2} 
\xymatrixcolsep{5pc} 
\xymatrix{\GHilb^{\a}_0(M) \ar[r]^-{f^{[\a]}=f^{[\a_1]} \times \ldots \times f^{[\a_s]}} \ar@{^{(}->}[d]& \Delta_{\a_1}N \times \ldots \times \Delta_{\a_s}N  \ar@{^{(}->}[d]^-{\iota_{\Delta_\a N}}  & \\
\GHilb^{\a}(M) \ar[r]^-{f^{[\a]}=f^{[\a_1]} \times \ldots \times f^{[\a_s]}}  & N^{k+1}=N^{\a_1}  \times \ldots \times N^{\a_s}  &  \nabla_{\a_1}N \times \ldots \times \nabla_{\a_s}N \ar@{^{(}->}[l]\ar[lu]^\d }
\end{equation}
 where $\GHilb^{\a}(M)=\GHilb^{\a_1}(M) \times \ldots \times \GHilb^{\a_s}(M)$ is the $\a$-approximating scheme defined in Definition \ref{def:approximatingsets}, and $\nabla_{\a_i}N$ stand for the neighborhood of the diagonal $\Delta_{\a_i}N$ in $N^{\a_i}$.
Suppose that $V^{[k+1]}$ has a subbundle of the form $f^{[k+1]*}W\hookrightarrow V^{[k+1]}$ for some bundle $W$ over $N^{k+1}$ and this is consistent with an embedding 
\[HC_{\a}^*f^{\times (k+1)*}W\hookrightarrow V^{\a}.\]
for all partitions $\a$ where
\[HC_{\a}:\GHilb^{\a}(M) \to M^{k+1}\]
is the Hilbert-Chow for the $\a$-approximating set. In short, we say that 
$V^{[k+1]}$ has a consistent subbundle $f^{[k+1]*}W$.
\begin{corollary}\label{cor:pushforward} Suppose that $V^{[k+1]}$ has a consistent subbundle $f^{[k+1]*}W \hookrightarrow V^{[k+1]}$ for some bundle $W$ over $N^{k+1}$. Then 
\[f^{[k+1]}_* \left[\Euler \left(\frac{V^{[k]}}{f^{[k]*}W}\right)\right]_{\Hilb^{[k]}(M)}=\sum_{(\a_1,\ldots, \a_s) \in \Pi(k+1)} (\iota_{\Delta_\a N})_* f^{\times s}_* \left [\res_{\bz^{\a_1}=\infty}\ldots \res_{\bz^{\a_s}=\infty} \mathcal{R}^\a(\theta_i, \bz) d\bz^{\a_s}\ldots d\bz^{\a_1} \right ]_{\Delta_\a M}\]
where $\mathcal{R}^\a(\a_i, \bz^{\a_1},\ldots, \bz^{\a_s})$ is the rational expression 
\[\Euler \left(\frac{V(\bz^{\a_1}) \oplus \ldots \oplus V(\bz^{\a_s})}{f^{[\a]*}W}\right) 
\prod_{l=1}^s \frac{(-1)^{|\a_l|-1} \prod_{1\le i<j \le |\a_l|-1}(z_i^{\a_l}-z_j^{\a_l})Q_{|\a_l|-1}d\bz^{\a_l}}
{\prod_{i+j\le l\le 
|\l_l|-1}(z_i^{\a_l}+z_j^{\a_l}-z_l^{\a_l})(z_1^{\a_l}\ldots z_{|\a_l|-1}^{\a_l})^{n+1}}\prod_{i=1}^{|\l_l|-1} 
s_X\left(\frac{1}{z_i^{\a_l}}\right)\]
\end{corollary}
We will apply this formula with $V=f^*TN$ and $W=T_{\Delta N}$, the tangent bundle to the small diagonal.

\section{Proof of the main theorems}\label{sec:proofmain}
We start with the introduction and analysis of the multipoint Hilbert scheme $\GHilb^k(f)$, which plays a central role in the proof. This is followed by local analysis of the $k$-fold locus in $M$ at its punctual locus. We then give the proof of the multipoint residue formula. 
 \subsection{The multipoint Hilbert scheme}

Let $f: M\to N$ be a map between smooth compact complex manifolds of dimension $\dim(M)=m, \dim(N)=N$. This induces the $k$th Hilbert extension map
\[\mathrm{hf}^{[k]}: \Hilb^{k}(M) \to \Hilb^{k}(M \times N)\]
which sends the subscheme $\xi_I$ to $\xi_{(I,I_{\Gamma(f)})}$ where $I_{\Gamma(f)}$ is the ideal of the graph $\Gamma(f)$ of $f$. Heuristically, the map $f^{[n]}$ pushes $\xi_I$ onto the graph along the $Y$ axis.
\begin{lemma} $\mathrm{hf}^{[k]}$ is a regular embedding.
\end{lemma}
We introduce two subsets of $\Hilb^k(M)$ which play central role in our argument. 

\begin{definition} \begin{enumerate} \item The $f$-Hilbert scheme $\Hilb^k(f)$ is defined by the following pull-back diagram
\[\xymatrix{\Hilb^k(f) \ar@{^{(}->}[r] \ar@{^{(}->}[d]^-{j} & \Hilb^k(M) \times N \ar@{^{(}->}[d]^-{j} \\ 
\Hilb^k(M) \ar@{^{(}->}[r]^-{\mathrm{hf}^{[k]}} & \Hilb^k(M \times N)}\]
We use the same notation $f$ for the associated map 
\[f: \Hilb^k(f) \to N \ \ \ \xi \mapsto f(\xi).\]
\item The geometric component $\GHilb^k(f) \subset \Hilb^k(f)$ is the closure of the open part consisting of $k$ points on a level set of $f$: 
\[\GHilb^k(f)=\overline{\{\xi=\xi_1 \sqcup \ldots \sqcup \xi_k \in \Hilb^k(M): f(\xi_1)=\ldots =f(\xi_s) \in N\}}\subset \Hilb^k(f)\]
\end{enumerate}
\end{definition}
\begin{remark}Intuitively, both $\Hilb^k(f)$ and $\GHilb^k(f)$ contain subschemes of length $k$ in $M$ whose projection to the graph of $f$ is horizontal. However, not every such 'horizontal' subscheme can be approximated by $k$ different points on level sets. More precisely, a punctual subscheme $\xi \in \Hilb^k_p(f)$ whose support is $p\in M$ sits in $\GHilb^k_p(f)$ if and only if 
\[\xi =\lim_{i\to \infty} \xi_i \text{ where } \xi_i \subset f^{-1}(p_i) \text{ for some } p_i \to p\]
\end{remark}
In fact, the test curve model of Morin singularities gives the following description. 
\begin{proposition}\label{prop:hilbf}  Let $\Hilb^k_0(f)=\cup_{p\in M} \Hilb^k_p(f)$ denote the punctual part of $\Hilb^k(f)$, consisting of subschemes supported at {\it some} point on $M$, and let $\pi_M: \Hilb^k_0(f) \to M$ denote the projection. For the Morin algebra $A_{k}=k[t]/t^{k}$ of dimension $k$ let 
\[\Hilb_{A_k}(f)=\{\xi \in \Hilb^k_0(f): \calo_\xi \simeq A_k\}\]
Then 
\[\overline{\pi_M(\Hilb_{A_k}(f))}=\Sigma_{A_k}(f)\]
where we recall that $\Sigma_{A_k}(f)=\overline{\{p\in M: A_{f,p} \simeq A_k\}}$ is the Thom locus whose cohomology cycle gives the Thom polynomial $\Tp_k$. 
\end{proposition}
\begin{proof}
The test curve model for the map $f:M \to N$ says that for a generic point $p\in \Sigma_{A_k}(f)$ there is a test curve $h: (\CC,0) \to (M,p)$ whose $k$-jet is annihilated by $f$, that is, $j^kf \circ j^kh =0$.
The $k$-jet $j^kh$ in this equation is unique up to $\diff_k(\CC)$. The image of $j^kh$ in $M$ is a $k$-jet of a smooth curve $C_h$ at $p$, and it defines a curvilinear subscheme $\xi_h=(h_1,\ldots, h_n)\in \mathrm{Curv}^k_p(M)$ where $(h_1,\ldots, h_n)$ is the ideal of $h$ in some local coordinates. We claim that $\xi_h \in \Hilb^k(f)$, and this finishes the proof. To check this, note that $j^kf \circ j^kh =j^k(f \circ h)=0$ means that the first $k$ derivatives of the curve $f \circ h$ are zero. Equivalently, the first $k$ derivatives of the corresponding curve  
\[\mathrm{Graph}(C_h)=\{(j^kh(x),j^kf (j^kh(x)):x\in \CC\}  \subset M\times N\]
on $\mathrm{Graph}(f)$ are horizontal in $M\times N$.  
\end{proof}
The following lemma is a reformulation of the definition of the image of the $k$-fold locus.
\begin{lemma}
The $k$-point locus in $N$ is equal to  $n_k=f_*[\GHilb^k(f)]$.
\end{lemma}

\begin{proposition} Let $f:M \to N$ be a Thom-Boardman map and $p\in M$ a point of multiplicity $k$. If $\GHilb^k(f) \cap \GHilb^k_p(M)$ is nonempty, then 
\[\GHilb^k(f) \cap \Hilb^k_p(M)=\mathrm{Spec}(A_{f,p})\]
is the scheme given by the local algebra of $f$ at $p$, and this ideal sits in the curvilinear component $\CHilb^k_p(M)$.  
\end{proposition}
\begin{proof} The first part is trivial. The second part immediately follows 
from Proposition \ref{prop:boardman}.
\end{proof}

\subsection{Local analysis of multipoint locus}

Recall the $k$-fold locus $M_k(f)=\overline{\pi_1(M_k^\times)}$ of the map $f:M \to N$ where 
\[M_k^\times=\{(p_1,\ldots, p_k):f(p_1)=\ldots =f(p_k), p_i \neq p_i\}\]
and $\pi_1$ is projection to the first factor. The punctual $k$-fold locus is
\[M_k^0(f)=\pi_1(\Delta \cap \overline{M_k^\times})\]
where $\Delta$ is the small diagonal of the Cartesian product $M^{\times k}$.

In this section we study local behaviour of the $k$-fold locus $M_k(f)$ near its punctual part. We start with looking at Porteous points. 

\begin{definition} The point $p\in M$ is called $k$-Porteous point of the holomorphic map $f:M \to N$ if the local algebra of $f$ at $p\in \Sigma^k$ is isomorphic to $A_{\Sigma^k}=\CC[x_1,\ldots, x_k]/(x_1,\ldots, x_k)^2$. The multiplicity of $f_p$ is $k+1$, and its Thom-Boardman type is $(k,0,0,\dots)$, that is, $p \in \Sigma^k(f)$.
\end{definition}

Let $f:(\CC^m,0) \to (\CC^n,0)$ be a map germ with local algebra $A_{\Sigma^k}$. Then $\dim(\ker(df))=k$, and according to Proposition \ref{prop:boardman} and Proposition \ref{prop:hilbf} the corresponding point $\mathrm{Spec}(A_f)\in \Hilb^{k+1}(\CC^m)$ sits in 
\[\mathrm{Spec}(A_f) \in \CHilb^{k+1}(\CC^m) \cap \GHilb^{k+1}(f).\]  
\begin{proposition}\label{prop:kporteous}
Let $f:(\CC^m,0) \to (\CC^n,0)$ be a stable map germ with local algebra $A_{\Sigma^k}$, and let $\xi_f=\mathrm{Spec}(A_{\Sigma^k}) \in \CHilb^{k+1}_0(\CC^m)$ be the corresponding point in the Hilbert scheme. Then $\GHilb^{k+1}(f) \subset \GHilb^{k+1}(\CC^m)$ is locally irreducible at $\xi_f$.
\end{proposition}

\begin{proof}
We follow the procedure described in the previous section to construct a stable Porteous germ in any codimension. The rank 0 genotype of an $k$-Porteous singularity is 
\[g_k: (x_1,\ldots, x_k) \mapsto (x_ix_j :1\le i\le j \le k)\]
whose codimension is $l=k(k+1)/2-k={k \choose 2}$. A stable unfolding of codimension $l$ is the map $\tg_k:(\CC^{k^2(k+1)/2},0) \to (\CC^{k^2(k+1)/2+{k \choose 2}},0)$ 
\[
\tg_k: (x_i, a_{ij}^s: 1\le i,j,s \le k) \mapsto \\
(x_ix_j+\sum_{s=1}^k a_{ij}^s x_s, a_{ij}^s: 1\le i,j,s \le k)
\]
where we add a linear form $L_{ij}=a_{ij}^1x_1+\ldots +a_{ij}^kx_k$ to each coordinate function, with the restriction that $k$ out of the $k^2(k+1)/2$ coefficients $a_{ij}^s$ are equal to zero. The choice of this $k$-tuple is not unique, resulting in different stable unfoldings, and a possible choice is...

We would like to describe the level set $\tg_k^{-1}(\Lambda_{ij}, a_{ij})$ for $|\Lambda_{ij}|,|a_{ij}|<\epsilon$. This is equivalent to solving the system 
\begin{equation}
x_ix_j+\sum_{s=1}^k a_{ij}^s x_s-\Lambda_{ij}=0 \tag{ij}
\end{equation}
The equations indexed by (12),(13),...,(1k) can be written in the following form
\begin{equation}\begin{pmatrix}x_1+a_{12}^2 & a_{12}^3 & a_{12}^4 & \cdots & a_{12}^k \\
a_{13}^2 & x_1+a_{13}^3 & a_{13}^4 & \cdots & a_{13}^k \\
a_{14}^2 & a_{14}^3 & x_1+a_{14}^4 & & a_{14}^k \\
\cdots & & & \ddots & \\
a_{1k}^2 & & & & x_1+a_{1k}^k
\end{pmatrix} 
\begin{pmatrix}
x_2 \\ x_3\\ \vdots \\ \\ x_k
\end{pmatrix}=\begin{pmatrix}
\Lambda_{12} \\ \Lambda_{13} \\ \vdots \\ \\ \Lambda_{1k}
\end{pmatrix} \tag{12)-(1k}
\end{equation}   
For generic choice of $x_1,a_{1j}^s$ the $(k-1) \times (k-1)$ matrix $A$ on the left hand side is diagonalizable, and in an eigenbasis given by $J$ the equation can written as $(J^{-1}AJ)J^{-1}\mathbf{x}=J^{-1}\Lambda$. The trace of $A$ is linear, whereas the determinant of $A$ is a degree $k-1$ polynomial in $x_1$, hence the same is true for the diagonal matrix $J^{-1}AJ$. Therefore the diagonal entries of $J^{-1}AJ$ are linear forms in $x_1$, and in the eigenbasis $J$ the equation (12)-(1k) has the form 
\[\mathrm{diag}(x_1+\a_2,x_1+\a_3,\ldots, x_1+\a_k)\tilde{\mathbf{x}}=\tilde{\Lambda}.\]
for $\tilde{\bx}=J^{-1}\bx, \tilde{\Lambda}=J^{-1}\Lambda$, and some $\a_2,\ldots, \a_k \in \CC$. 
Note that 
\[x_1+\a_i \neq 0 \text{ for } 2\le j \le k\]
must hold, otherwise the map $\tg_k$ written in the bases $J$ would have less than $k(k+1)/2$ nonzero quadratic coordinate functions, so its local algebra could not be isomorphic to $(x_1,\ldots, x_k)/(x_1,\ldots, x_k)^2$, because any presentation of this algebra has at least $k(k+1)/2$ quadratic relations. Then  
\begin{equation}\label{transformJ}
\tx_j=\frac{\tilde{\Lambda}_{1j}}{x_1+\a_j} \text{ for } 2\le j \le k
\end{equation}
and with $\tx_1=x_1$ equation (11) can be rewritten in the basis $J$ as
\begin{equation}\nonumber
\Lt_{11}=\tx_1^2+\ta_{ss}^1\tx_1+\ta_{ss}^2\tx_2+\ldots +\ta_{ss}^k\tx_k
\end{equation}
Substituting \eqref{transformJ} we obtain 
\begin{equation}
\Lt_{11}= \tx_1^2\prod_{i \neq 1}(\tx_1+\a_i)+ \ta_{11}^1\tx_1\prod_{i \neq 1}(\tx_1+\a_i)+\ta_{11}^2\Lt_{12}\prod_{i \neq 1,2}(\tx_1+\a_i)+\ldots + \ta_{ss}^k\Lt_{1k}\prod_{i \neq 1,k}(\tx_1+\a_i) \tag{11} 
\end{equation}
The first term in (11) has degree $k+1$ in $\tx_1$, hence the $x_1$ coordinates of the $k+1$ points in $\tg_k^{-1}(\Lambda_{ij}, a_{ij})$ are the $k+1$ solutions (with multiplicity) of (11) and the corresponding $x_2,\ldots, x_k$ coordinates are given by \eqref{transformJ}.
 \begin{remark}
For $s>1$ equation (ss) has slightly different form 
\begin{multline} \Lt_{ss}= \Lt_{1s}^2 \prod_{i \neq 1,s}(\tx_1+\a_i)+ \ta_{ss}^1\tx_1(\tx_1+\a_s)^2\prod_{i \neq 1,s}(\tx_1+\a_i)+\tag{ss} \\
+\ta_{ss}^2\Lt_{12}(\tx_1+\a_s)^2\prod_{i \neq 1,2,s}(\tx_1+\a_i)+\ldots + \ta_{ss}^k\Lt_{1k}(\tx_1+\a_s)^2\prod_{i \neq 1,k,s}(\tx_1+\a_i) \nonumber
\end{multline}
but these are generically also degree $k+1$ polynomials in $\tx_1$. On the other hand, the equations $(ij)$ for $1<i\neq j$ have degree less than $k+1$ in $x_1$, hence in order to get $k+1$ points in the level set $\tg_k^{-1}(\Lambda_{ij}, a_{ij})$ these equations must vanish. Hence the level set $\tg_k^{-1}(\Lambda_{ij}, a_{ij})$ has (maximal) multiplicity $k+1$ if and only if 
\begin{enumerate}[(i)]
\item The equations (11)-(kk) are collinear.
\item The equations (ij) for $1<i \neq j$ are identically $0$. 
\end{enumerate}
Equations given by (i) and (ii) cut out the $k+1$-fold locus $\CC^m_{k+1}$ from the source space $\CC^m$. 
\end{remark}
To prove local irreducibility we note that if the matrix $A$ in (12)-(1k) is 
not diagonalizable, we can still use its Jordan normal form and the 
corresponding level set is given by the limit of level sets corresponding to 
diagonalizable matrices.
\end{proof}
Proposition \ref{prop:kporteous} shows irreducibility of $\GHilb^{k+1}(f)$ at points $p$ where $f_p$ has multiplicity $k+1$. Next we study local irreducibility at higher multiplicity points. These points naturally sit in the closure of the locus where the multiplicity is $k+1$.

\begin{proposition}\label{prop:k+lporteous}
Let $f:(\CC^m,0) \to (\CC^n,0)$ be a stable map germ with local algebra $A_{\Sigma^{k+\ell}}$ for some $\ell>0$. Then $\GHilb^{k+1}(f)$ is locally irreducible at points of $\GHilb^{k+1}_0(f)$ 
\end{proposition}
\begin{proof} Since $\xi_f=\Spec(A_{\Sigma^{k+\ell}})$ is a $k+\ell$-Porteous point, its kernel $\ker(f_0)$ is a $k+\ell$-dimensional subspace in $\CC^m$. Let $\xi_i \in \GHilb^{k+1}(f)$ be a sequence of $k$-fold points on $\CC^m$ with limit $\xi =\lim_{i\to \infty} \xi_i \in \GHilb^{k+1}_0(f)$ Then $\xi$ is a $k$-Porteous point with a $k$-dimensional kernel $\ker(\xi) \subset \ker(f_0)$, and any this kernel uniquely determines the point $\xi$.  Hence 
\[\GHilb^{k+1}_0(f)=\grass_k(\ker(f_0))\]
and by Proposition \ref{prop:kporteous} the $k+1$-fold locus $\GHilb^{k+1}(f)$ fibers over this Grassmannian with locally irreducible fibers. 
\end{proof}
In fact, by taking limits of the subscheme determined by the points. labeled by $1,\ldots, i$ for $1\le i\le k$, we get a fibration 
\[\xymatrix{\GHilb^{k+1}_\nabla(f)\ar@{^{(}->}[r] \ar[d]& \widehat{\CHilb}^{k+1}_\nabla(\CC^m) \ar[d] \\
 \flag_k(\ker(f_0))  \ar@{^{(}->}[r] \ar[d]& \flag_k(\CC^m) \ar[d] \\
\GHilb^{k+1}_0(f) \ar@{^{(}->}[r] & \grass_k(\CC^m) 
}\]
where $\GHilb^{k+1}_\nabla(f)$ is a local neighborhood of $\GHilb^{k+1}_0(f)$ in $\GHilb^{k+1}(f)$.
\begin{proposition}\label{prop:hilbfloc}
Let $f:M \to N$ be a stable Thom-Boardman map. Then $\GHilb^{k+1}_0(f) \subset \CHilb^{k+1}(M)$, and $\GHilb^{k+1}(f)$ is locally irreducible at points of the punctual part $\GHilb^{k+1}_0(f)$.
\end{proposition}
 \begin{proof}
 Local irreducibility follows from Proposition \ref{prop:k+lporteous} and Proposition \ref{prop:boardman}. Indeed, for a Thom-Boardman map every singularity locus $\Sigma_A(f) \subset M$ contains Porteous-points in its closure, and local irreducibility is a closed condition. 
 \end{proof}

\subsection{Proof of the multipoint residue formula} We will work with the following diagram
\[\xymatrix{& V \ar[d] & & \\
\Hilb^k(f) \ar@{^{(}->}[r]  & \GHilb^k(M) \ar[r]^-{f^{[k]}} & N^k & \nabla_N \ar@{^{(}->}[l] \ar[d]^{\pi_{\Delta N}}\\
& & & \Delta_N \ar@{^{(}->}[lu]^i}\]
where $V$ is a to-be-defined vector bundle such that $\Hilb^k(f)$ is the zero set some section of it. Such vector bundle does not exist on the whole $\Hilb^k(M)$, but we can get around this problem as follows. Let $\nabla_N \subset N^k$ be (a tubular neighborhood of) the normal bundle of the small diagonal $\Delta_N$. Note that 
\[\Hilb^k(f) \subset (f^{[k]})^{-1}(\nabla_N) \subset \Hilb^k(M)\]
and $\Hilb^k(f)$ is compact in $(f^{[k]})^{-1}(\nabla_N)$. Hence it is enough to define $V$ over $U=(f^{[k]})^{-1}(\nabla_N)$, and take any topological extension. Our calculations will take place in $\nabla_{N}$ and in compactly supported cohomology in the open set $U$. 
\begin{proposition}\label{prop:sectionk}
Let  $T_v\to\nabla_N$ be the tautological pull-back bundle on $\nabla_{N}$  and $T_h=\pi_{\Delta N}^*T(\Delta N)\to\nabla_N$ be the pull-back of the diagonal tangent from the base. Then there is a bundle embedding  $f^{[k]*}T_h\hookrightarrow f^*TN^{[k]}$  consistent with an embedding 
\[f^{[\a]*}T_h \hookrightarrow f^{[\a_1]*}TN_1\oplus \ldots \oplus f^{[\a_s]*}TN_s\]
for all partitions $\a=(\a_1,\ldots, \a_s) \in \Pi(k)$, and the quotient bundle
\[f^*TN^{[k]}/f^{[k]*}T_h\]
has a transversal section presenting $\GHilb^k(f)$ as a complete intersection in $U$.
\end{proposition} 

\begin{proof} 
Eliminating the terms of degree $k+1$ results in an algebra homomorphism
$J_k(u,1) \twoheadrightarrow J_{k-1}(u,1)$, and the chain
$J_k(u,1) \twoheadrightarrow J_{k-1}(u,1)
\twoheadrightarrow \ldots \twoheadrightarrow J_1(u,1)$ induces
an increasing filtration on $J_k(u,1)^*$:
\begin{equation}\label{filtration} 
J_1(u,1)^* \subset J_2(u,1)^* \subset \ldots \subset J_k(u,1)^*.
\end{equation}
Choosing coordinates on $\CC^u$ and $\CC^v$ a $k$-jet $f \in J_k(u,v)$ can be identified with the set of derivatives at the
origin, that is the vector $(f'(0),f''(0),\ldots, f^{(k)}(0))$, where
$f^{(j)}(0)\in \mathrm{Hom}(\mathrm{Sym}^j\CC^u,\CC^v)$. Here $\sym^l$ stands for the symmetric tensor product. This way we
get the isomorphism 
\begin{equation}\label{identification}
J_k(u,v) \simeq J_k(u,1) \otimes \CC^v \simeq \oplus_{j=1}^k\mathrm{Hom}(\mathrm{Sym}^j\CC^u,\CC^v).
\end{equation}

Similarly, the linear dual $J_i(u,1)^*$ may be interpreted as a set of
differential operators on $\CC^u$ of degree at most
$i$, and in particular, by taking symbols, we have the following 
isomorphism of filtered $\GL(n)$-modules.
\begin{equation}\label{eq:dual}
  J_k(u,1)^* \cong \Sym^{\le k}\CC^u \overset{\mathrm{def}}=
\oplus_{l=1}^k \sym^l\CC^u,
\end{equation}
Given a regular $k$-jet $f: (\CC,0) \to (\CC^n,0)$ in $J_k^{\reg}(1,n)$ we may push forward the differential operators of order $k$ on $\CC$ (with constant coefficients) to $\CC^n$ along $f$ which gives us a map 
\begin{equation}\label{embedgrass2}
\tilde{f}:J_k(1,1)^* \to \grass(k,J_k(m,1)^*)
\end{equation}
We can set the  fibered version of this argument as follows. Let $J_k^{\reg}(\CC,M \to M$ denote the bundle of $k$-jets of germs of parametrized regular curves in $M$, that is, where the first derivative $f'\neq 0$ is nonzero. Then  
\[J_k^{\reg}(\CC,M)=\diff_M \times_{\diff_k(m)} J_k^\reg(1,m).\]
where $\diff_M$ denotes the principal $\diff_k(m)$-bundle over $M$ formed by all local polynomial coordinate systems on $M$.
The dual bundle is $\cald_M^{\le k}$, the bundle of at most $k$th order differential operators over $M$. Then $\cald_M^{\le 0}=\calo_M$, and we let $\cald_M^k=\cald_M^{\le k}/\cald_M^{\le 0}$. We have a filtration of bundles 
\begin{equation}\label{dfiltration2}
\calo_M=\cald_M^{\le 0} \subset \cald_M^{\le 1} \subset \ldots \subset \cald_M^{\le k}
\end{equation}
where the graded component $\cald_M^{\le i}/\cald_M^{\le i-1}\simeq \Sym^iT_M$ but this filtration is not split in general, so $\cald_M^{\le k} \nsimeq \Sym^{\le k}T_M$.
Then, again, $\cald_M^k$ is the associated bundle 
\[\cald_M^k=\diff_M \times_{\diff_k(m)} J_k(m,1)^*=\diff_M \times_{\diff_k(m)} \symdot.\]
Given a regular $k$-jet $f:(\CC,0) \to (M,p)$ we may push forward the differential operators of order k on $\CC$ to $M$ along $f$ and obtain a $k$-dimensional subspace of $\cald^{\le k}_{M,p}$. This gives the bundle map 
\begin{equation}\label{diffoppushforward2}
J_k^\reg(\CC,M) \to \grass_k(\cald_M^{k})
\end{equation} 
which is the fibred version of the map \eqref{embedgrass1}. Note that $\diff_k(1)=\jetreg 11$ acts fibrewise on the jet bundle $J_k^{\reg}(\CC,M)$ via the composition map \eqref{comp} and the map \eqref{diffoppushforward2} is $\diff_k(1)$-invariant resulting an embedding 
\begin{equation}\label{invdiffoppush2}
\tau:\mathrm{Curv}^k_0(M)=J_k^{\reg}(\CC,M)/\diff_k(1) \hookrightarrow \grass_k(\cald_M^{k})
\end{equation}
A $k$-jet of a map modulo reparametrisation gives a curvilinear subscheme, hence the equation above. Finally, the tautological bundle $E$ over $\grass_k(\cald_M^{k})$ restricts to the tautological bundle $\calo_M^{[k]}$ over the curvilinear locus:
\[\tau^*E=\calo_M^{[k]}\]
A section $s$ of $E$ is given by a dual element of $\cald_M^{k}$, that is, an element of 
\[J_k^{\reg}(M,\CC)=\diff_M \times_{\diff_k(m)} J_k^\reg(n,1).\] 
The $k$-jet of $f:M \to N$ defines an element of 
\[j^kf \in J_k^{\reg}(M,\CC) \otimes f^*TN=H^0(E \otimes f^*TN)\]
The restriction of this section to $\mathrm{Curv}^k_0(M) \subset \grass_k(\cald_M^{k})$ gives the desired section. 
\end{proof}

\begin{remark}
An alternative proof goes as follows. By taking the linear part of a jet of a function, we obtain a bundle $J_k(N,\C) \to T^*N$. We will refer to a section of this bundle as a {\it $k$-straightening of the manifold $N$.} 

As a first step, we construct the desired section over the punctual part $\GHilb^k_0(M)=\cup_{p\in M} \GHilb^k_p(M)$. Given a map $f: M \to N$, and a punctual ideal $I \subset \mathfrak{m}_p$ we can define a linear map 
\[T_p^*N \to \mathfrak{m}_p/I\]
which, in turn,  is a point in 
\[T_pN \otimes \mathfrak{m}_p/I=f^*TN^{[k]}_I\] 
\begin{enumerate}
\item We start with an element $\omega \in T^*_{f(p)}N$. Using a local $k$-straightening we associate a $k$-jet $j^k\omega \in J_k(T_{f(p)}N,\CC)$ of a function at $f(p)$ on $N$.
\item Then we the pull-back $f_\omega=j^k\omega \circ f \in J_k(T_pM,\CC)$ gives a $k$-jet of a function on $M$ at $p$, vanishing at $p$.
\item Finally we reduce this $k$-jet modulo $I$, i.e we take the image of $f_\omega$ under the quotient map 
\[J_k(T_pM,\CC)\simeq \mathfrak{m}_p/\mathfrak{m}_p^k \to \mathfrak{m}_p/(I,\mathfrak{m}_p^k)=\mathfrak{m}_p/I,\ \ f_\omega \mapsto \hat{f}_\omega.\]
\end{enumerate}
The desired map is $\omega \to \hat{f}_\omega$.
The second step is that we extend this section analytically to a small tubular neighborhood of the punctual locus $\Hilb^k_0(M)$.
Finally, we need to check that $\Hilb^k(f)$ is indeed the zero set of the constructed section. By definition, $\hat{f}_\omega=0$ is equivalent to $f_\omega=j^k\omega \circ j^kf \in I$,  and this says that $I \in \Hilb^k(f)$.

\end{remark} 

Proposition \ref{prop:hilbfloc} and Proposition \ref{prop:sectionk} together prove that $\Euler(f^*TN^{[k]}/f^{[k]*}T_h)$ is properly supported, which we postulate in the following  
\begin{corollary}
Let $f:M \to N$ be a stable Thom Boardman map, whose $k$-jet induces the section $s_f$ of $f^*TN^{[k]}/f^{[k]*}T_h$ presenting $\GHilb^k(f)$ as a local complete intersection in $U$. Then 
\[[\GHilb^k(f)]=\mathrm{Euler}(f^*TN^{[k]}/f^{[k]*}T_h)\]
is properly supported.
\end{corollary}
		
We are ready to deduce the multipoint formula. Recall that we defined the $k$-point locus in $N$ as 
$f_*[\Hilb^k(f)]$ and our goal is to evaluate this cohomology class in $N$.
As $f(\Hilb^k(f))$ is a subset of the diagonal $\Delta N\subset N^{\times k}$, our calculations will take place in a neighborhood of $\Delta N$, which we will identify with the normal bundle $\d: \nabla_{N} \to \Delta_N$ of $\Delta N$ in $N^{\times k}$, and in compactly supported cohomology in the open set $U=\left(f^{[k]}\right)^{-1}\nabla_{ N}$. Since $\Euler(T_v)=c_{n(k-1)}(T_v)$ is exactly the Thom class of the bundle $\nabla_N$, the Thom isomorphism tells us   
\[  \hcpt(\nabla_N)\simeq H^*(\Delta_N)\cdot c_{n(k-1)}( T_v).
 \] 
We also, clearly have $T\nabla_N\simeq TN_1\oplus \ldots \oplus TN_k \simeq T_h\oplus T_v$. We apply Corollary \ref{cor:pushforward} with $V=f^*TN$ and $W=T_h$.
{\small \begin{multline}\label{finalformulak}
	[f(\Hilb^{k}(f))]=
	\pi_{\Delta N*}f^{[k]}_*\left[\Euler\left(\frac{f^*TN^{[k]}}{f^{[k]*}T_h}\right)\right]_{U\subset\GHilb^k(M)}=\\ 
 \sum_{(\a_1,\ldots, \a_s) \in \Pi(k)} \frac1{c_{n(k-1)}(T_v)} (\iota_{\Delta_\a N})_* f^{\a}_* \left [\res_{\bz^{\a_1}, \ldots ,\bz^{\a_s}=\infty} c_{n(k-1)} \left(\frac{f^*TN(\bz^{\a_1}) \oplus \ldots \oplus f^*TN(\bz^{\a_s})}{f^{[\a]*}T_h})\right) \mathcal{R}^\a(\theta_i, \bz) d\bz^{\a_s}\ldots d\bz^{\a_1} \right ]_{\Delta_\a M}
\end{multline}}
Recall that here 
\begin{itemize}
\item $\Delta_\a N=\Delta_{\a_1}N \times \ldots \times \Delta_{\a_s}N$, $\Delta_\a M=\Delta_{\a_1}M \times \ldots \times \Delta_{\a_s}M$, $f^\a=f^{\times s} \circ \HC: \Delta_\a M \to \Delta_\a N$.
\item $\iota_{\Delta_\a N}: \Delta_\a N \to N^{\times k}$ is the embedding.
\item Push forward $\pi_{\Delta N*}$ along the normal bundle of the small diagonal in $N$ is division by the Euler class by the Thom isomorphism.
\item For $1\le t \le s$ we set $f^*TN(\bz^{\a_t})=f^*TN \oplus f^*TN(z^{\a_t}_1) \oplus \ldots \oplus f^*TN(z^{\a_t}_{|\a_t|-1})$.
\end{itemize}
The key final observation is that the push-forward map $f^\a_*$ is a $H_T^\bullet(N)$-module homomorphism, and hence
\[\frac1{c_{n(k-1)}(T_v)} (\iota_{\Delta_\a N})_* f^{\a}_* c_{(k-1)n} \left(\frac{f^*TN(\bz^{\a_1}) \oplus \ldots \oplus f^*TN(\bz^{\a_s})}{f^{[\a]*}T_h}\right)=f^\a_* c_{(k-s)n}(f^*TN(\bz^{\a_1})^{0} \oplus \ldots \oplus f^*TN(\bz^{\a_s})^{0})\]
where 
\[f^*TN(\bz^{\a_t})^0=f^*TN(z^{\a_t}_1) \oplus \ldots \oplus f^*TN(z^{\a_t}_{|\a_t|-1})\ \ 1\le t \le s\]
stands for the reduced bundles. In short, pull-back and push-forward along the diagonal $\Delta^\a N$ kills the horizontal direction from the tautological bundle $V^\a(\bz)$. Hence the generic term in \eqref{finalformulak} can be written as 
{\small \begin{align*}	
& f^\a_* \left [\res_{\bz^{\a_1}, \ldots ,\bz^{\a_s}=\infty} c_{(k-s)n}(f^*TN(\bz^{\a_1})^{0} \oplus \ldots \oplus f^*TN(\bz^{\a_s})^{0})\mathcal{R}^\a(\theta_i, \bz) d\bz^{\a_s}\ldots d\bz^{\a_1} \right ]=\\
& =f^\a_* \res_{\bz^{\a_1}, \ldots ,\bz^{\a_s}=\infty} \prod_{t=1}^s  \frac{\prod_{1\le i<j \le |\a_t|-1}(z_i^{\a_t}-z_j^{\a_t})Q_{|\a_t|-1} c_{(k-s)n}(f^*TN(\bz^{\a_1})^{0} \oplus \ldots \oplus f^*TN(\bz^{\a_s})^{0})}
{(-1)^{|\a_t|-1}  \prod_{i+j\le l\le 
|\a_t|-1}(z_i^{\a_t}+z_j^{\a_t}-z_l^{\a_t})(z_1^{\a_t}\ldots z_{|\a_t|-1}^{\a_t})^{n+1}}\prod_{i=1}^{|\a_t|-1} 
s_M\left(\frac{1}{z_i^{\a_m}}\right) =\\
&=\prod_{t=1}^s f^{\a_t}_*  \res_{\bz^{\a_t}=\infty}   \frac{\prod_{1\le i<j \le |\a_t|-1}(z_i^{\a_t}-z_j^{\a_t})Q_{|\a_t|-1} c_{(|\a_t|-1)n}(f^*TN(\bz^{\a_t})^{0} }
{(-1)^{|\a_t|-1}  \prod_{i+j\le l\le 
|\a_t|-1}(z_i^{\a_t}+z_j^{\a_t}-z_l^{\a_t})(z_1^{\a_t}\ldots z_{|\a_t|-1}^{\a_t})^{n+1}}\prod_{i=1}^{|\a_t|-1} 
s_M\left(\frac{1}{z_i^{\a_t}}\right) 
\end{align*}}
The last equality hols because the Euler class of a sum is the product of Euler classes 
\[c_{(k-s)n}(f^*TN(\bz^{\a_1})^{0} \oplus \ldots \oplus f^*TN(\bz^{\a_s})^{0})=\prod_{t=1}^s c_{(|\a_t|-1)n}(f^*TN(\bz^{\a_t})^{0})\]
and the residue variables for $\a_i$ and $\a_j$ are disjoint for $1 \le i\neq j \le s$. But
\[c_{(|\a_t|-1)n}(f^*TN(\bz^{\a_t})^{0})=\prod_{i=1}^{|\a_t|-1}\prod_{j=1}^n(z_i^{\a_t}+\theta_j)\]
and hence  
\[c_{(|\a_t|-1)n}(f^*TN(\bz^{\a_t})^{0}) \cdot \prod_{i=1}^{|\l_t|-1} s_M\left(\frac{1}{z_i^{\a_t}}\right)=\frac{\prod_{i=1}^{|\a_t|-1}\prod_{j=1}^n(z_i^{\a_t}+\theta_j)}{\prod_{i=1}^{|\a_t|-1}\prod_{j=1}^m(z_i^{\a_t}+\l_j)}=\frac{1}{(z_1^{\a_t} \ldots z_{|\a_t|-1}^{\a_t})^{m-n}}\prod_{i=1}^{|\a_t|-1}c_f\left(\frac{1}{z_i^{\a_t}}\right)\]
When we substitute this, and compare with the residue formula of \cite{bsz} for Thom polynomials (see also in the Introduction), we arrive at 
\[\res_{\bz^{\a_t}=\infty}   \frac{\prod_{1\le i<j \le |\a_t|-1}(z_i^{\a_t}-z_j^{\a_t})Q_{|\a_t|-1} c_{(|\a_t|-1)n}(f^*TN(\bz^{\a_t})^{0} }
{(-1)^{|\a_t|-1}  \prod_{i+j\le l\le 
|\a_t|-1}(z_i^{\a_t}+z_j^{\a_t}-z_l^{\a_t})(z_1^{\a_t}\ldots z_{|\a_t|-1}^{\a_t})^{m-n+1}}\prod_{i=1}^{|\a_t|-1} 
c_f\left(\frac{1}{z_i^{\a_t}}\right)=\Tp^{m \to n-1}_{A_{|\a_t|-1}}.\]
Hence the term corresponding to $\a=(\a_1,\ldots, \a_s)$ is 
\begin{equation}
\prod_{i=1}^s f_*^{\a_t} \Tp^{m \to n-1}_{A_{|\a_t|-1}}
\end{equation}
which completes the proof of Theorem \ref{mainthm1} and Theorem \ref{mainthm2}.

\bibliographystyle{abbrv}
\bibliography{thom.bib}
\end{document}